\documentclass{amsart}
\usepackage{graphicx}
\usepackage[latin2]{inputenc}
\usepackage{amsmath,amssymb,amsbsy,amsthm,amsfonts,a4wide}
\usepackage{mathtools}
\usepackage[margin=1.9cm]{geometry}

\usepackage{soul,xcolor}

\DeclareMathOperator*{\essinf}{ess\,inf}
\DeclareMathOperator{\spt}{supp}

\begin{document}
\newcommand{\ddd}{\,{\rm d}}

\def\note#1{\marginpar{\small #1}}
\def\tens#1{\pmb{\mathsf{#1}}}
\def\vec#1{\boldsymbol{#1}}

\def\norm#1{\left|\!\left| #1 \right|\!\right|}
\def\fnorm#1{|\!| #1 |\!|}
\def\abs#1{\left| #1 \right|}
\def\ti{\text{I}}
\def\tii{\text{I\!I}}
\def\tiii{\text{I\!I\!I}}
\def\diam{\mathop{\mathrm{diam}}\nolimits}

\def\diver{\mathop{\mathrm{div}}\nolimits}
\def\dist{\mathop{\mathrm{dist}}\nolimits}
\def\grad{\mathop{\mathrm{grad}}\nolimits}
\def\Div{\mathop{\mathrm{Div}}\nolimits}
\def\Grad{\mathop{\mathrm{Grad}}\nolimits}

\def\tr{\mathop{\mathrm{tr}}\nolimits}
\def\cof{\mathop{\mathrm{cof}}\nolimits}
\def\det{\mathop{\mathrm{det}}\nolimits}

\def\lin{\mathop{\mathrm{span}}\nolimits}
\def\pr{\noindent \textbf{Proof: }}
\def\pp#1#2{\frac{\partial #1}{\partial #2}}
\def\dd#1#2{\frac{\d #1}{\d #2}}

\def\T{\mathcal{T}}
\def\R{\mathcal{R}}
\def\Re{\mathbb{R}}
\def\bx{\vec{x}}
\def\bz{\vec{z}}
\def\be{\vec{e}}
\def\bef{\vec{f}}
\def\bec{\vec{c}}
\def\bs{\vec{s}}
\def\ba{\vec{a}}
\def\bn{\vec{n}}
\def\bphi{\vec{\varphi}}
\def\btau{\vec{\tau}}
\def\bc{\vec{c}}
\def\bg{\vec{g}}
\def\measure{\mathop{\mathrm{meas}}\nolimits}
\def\mM{\mathcal{M}}
\def\mL#1#2{\mathcal{L}^{\alpha_{#2}}_{#1}}

\def\bW{\tens{W}}
\def\bQ{\tens{Q}}
\def\bA{\tens{A}}
\def\cA{\mathcal{A}}
\def\cB{\mathcal{B}}
\def\bT{\tens{T}}
\def\bP{\tens{P}}
\def\bD{\tens{D}}
\def\bO{\tens{O}}
\def\bG{\tens{G}}
\def\bF{\tens{F}}
\def\bB{\tens{B}}
\def\bE{\tens{E}}
\def\bV{\tens{V}}
\def\bS{\tens{S}}
\def\bI{\tens{I}}
\def\bH{\tens{H}}
\def\beps{\tens{\varepsilon}}
\def\bi{\vec{i}}
\def\bv{\vec{v}}
\def\bfi{\vec{\varphi}}
\def\bk{\vec{k}}
\def\b0{\vec{0}}
\def\bom{\vec{\omega}}
\def\bw{\vec{w}}
\def\p{\pi}
\def\bu{\vec{u}}
\def\bq{\vec{q}}

\def\ID{\mathcal{I}_{\bD}}
\def\IP{\mathcal{I}_{p}}
\def\Pn{(\mathcal{P})}
\def\Pe{(\mathcal{P}^{\eta})}
\def\Pee{(\mathcal{P}^{\varepsilon, \eta})}

\newtheorem{Theorem}{Theorem}[section]
\newtheorem{Conjecture}{Conjecture}[section]
\newtheorem{Example}{Example}[section]
\newtheorem{Lemma}{Lemma}[section]
\newtheorem{Remark}{Remark}[section]
\newtheorem{Definition}{Definition}[section]
\newtheorem{Corollary}{Corollary}[section]

\numberwithin{equation}{section}

\title[Nonlinear elliptic systems with linear growth]{On the existence of integrable solutions to\\ nonlinear elliptic systems and variational problems \\ with linear growth}

\author[L.~Beck]{Lisa Beck}
\address{Institut f\"{u}r Mathematik, Universit\"{a}t Augsburg\\
 Universit\"{a}tsstr. 14, 86159~Augsburg, Germany}
\email{lisa.beck@math.uni-augsburg.de}

\author[M.~Bul\'{i}\v{c}ek ]{Miroslav Bul\'i\v{c}ek}\thanks{The authors acknowledge the support of the ERC-CZ project LL1202, financed by M\v{S}MT}

\address{Mathematical Institute, Faculty of Mathematics and
Physics, Charles University \\ Sokolovsk\'{a}~83,
186~75~Prague~8, Czech~Republic}
\email{mbul8060@karlin.mff.cuni.cz}

\author[J. M\'{a}lek]{Josef M\'alek}
\address{Mathematical Institute, Faculty of Mathematics and
Physics, Charles University \\ Sokolovsk\'{a}~83,
186~75~Prague~8, Czech~Republic}
\email{malek@karlin.mff.cuni.cz}

\author[E. S\"{u}li]{Endre S\"{u}li}
\address{Mathematical Institute, University of Oxford\\ Andrew Wiles Building, Woodstock Rd., Oxford OX2 6GG, UK}
\email{endre.suli@maths.ox.ac.uk}

\subjclass[2000]{}

\begin{abstract}
We investigate the properties of certain elliptic systems leading, a~priori, to solutions that belong to the space of Radon measures. We show that if the problem is equipped with a so-called
asymptotic Uhlenbeck structure, then the solution can in fact be understood as a standard weak solution, with one proviso: analogously as in the case of minimal surface equations, the attainment of the boundary value is penalized by a measure supported on (a subset of) the boundary, which, for the class of problems under consideration here, is the part of the boundary where a Neumann boundary condition is imposed.
\end{abstract}

\maketitle

\section{Setting of the problem}\label{Sec1}

\subsection{Introduction} A challenging problem in mathematical analysis is to understand the behavior of solutions to systems of nonlinear partial differential equations, or of minimizers to associated variational problems, that exhibit linear growth of the minimized quantity with respect to the unknown. We focus in this paper on one such class and consider two types of problems. The first one is a nonlinear elliptic system consisting of~$N$ equations, considered on a bounded open set in $\mathbb{R}^d$, where the unknown solution~$\bu$ and its `flux'~$\bT$ are related in such a way that $\nabla \bu$ is a~priori bounded. The second type represents an interesting nonlinear problem in linearized elasticity, with the stress~$\bT$ and the displacement~$\bu$, considered as unknowns, assumed to be related in such a way that the linearized strain, $\beps(\bu)\coloneqq\frac12 (\nabla \bu + (\nabla \bu)^T)$, is a priori bounded. In the latter case the a priori bound controls merely the symmetric part of the gradient of the displacement, which makes the analysis of the associated boundary-value problem different from the one for systems of the first type. While the gradient of the unknown function in the first case (or the symmetric part of the displacement gradient in the second case) is bounded, the associated `flux'~$\bT$ can, a~priori, only be guaranteed to belong to the space of Radon measures. The aim of this paper is to show that if the problem is equipped with a so-called
\emph{asymptotic Uhlenbeck structure}, then the solution can in fact be understood as a standard weak solution, with one proviso: analogously as in the case of minimal surface equations, the attainment of the boundary value is penalized by a measure supported on (a subset of) the boundary, which, for the class of problems under consideration here, is the part of the boundary where a Neumann boundary condition is imposed. This result is formulated precisely in Section~\ref{Sec2} and is based on a novel tool that we call
\emph{renormalized regularity}. The objective of this introductory section is to formulate the problems under consideration and motivate the concept of solution by means of existing results from the literature concerning nonlinear elliptic problems with linear growth.

\subsection{Formulation of the problems} We first consider the following problem: for a bounded, connected, Lipschitz domain $\Omega \subset \mathbb{R}^d$, $d\ge 2$, with Dirichlet boundary $\Gamma_D$ and Neumann boundary $\Gamma_N$, which are {relatively open subsets of $\partial \Omega$} such that $\Gamma_{D} \cap \Gamma_{N}=\emptyset$ and $\overline{\Gamma_D \cup \Gamma_N}=\partial \Omega$, a given vector field $\bef:\Omega \to \mathbb{R}^N$, with $N\in \mathbb{N}$, a given $\bg:\Gamma_N \to \mathbb{R}^N$, a given boundary datum $\bu_0:\overline\Omega \to \mathbb{R}^N$, and a given \emph{bounded} mapping $\bD \colon \mathbb{R}^{d\times N} \to \mathbb{R}^{d\times N}$, we seek a couple $(\bu, \bT)$ such that $\bu:\overline{\Omega} \to \mathbb{R}^N$, $\bT:\overline{\Omega} \to \mathbb{R}^{d\times N}$ and
\begin{equation}\label{elastic-g}
\begin{aligned}
-\diver \bT &= \bef, \qquad \nabla \bu=\bD(\bT) &&\textrm{ in } \Omega,\\
\bu&=\bu_0 &&\textrm{ on } \Gamma_D,\\
\bT\bn &=\bg &&\textrm{ on } \Gamma_N,
\end{aligned}
\end{equation}
where $\bn$ denotes the unit outward normal vector on $\Gamma_N$. When $\Gamma_D=\emptyset$, $\bef$ and $\bg$ will be assumed to satisfy a standard compatibility condition (cf.~\eqref{D3} below).

Note that if~$\bD$ has an inverse $\bD^{-1}$ on the range of~$\bD$, which we shall always assume to be the case here, then we can rewrite the system~\eqref{elastic-g} in the following, more familiar, form (which however will not be used in what follows):
\begin{equation}\label{elastic-g2}
\begin{aligned}
-\diver \bD^{-1}(\nabla \bu) &= \bef \qquad &&\textrm{ in } \Omega, \\
\bu&=\bu_0 &&\textrm{ on } \Gamma_D,\\
\bD^{-1}(\nabla \bu)\bn &=\bg &&\textrm{ on } \Gamma_N.
\end{aligned}
\end{equation}

As a prototypical example one can consider
\begin{equation}\label{LSt2g}
\bD(\bT) = \frac{\bT}{(1+|\bT|^a)^{\frac{1}{a}}},\qquad a>0.
\end{equation}
In this case, $\bT= \bD^{-1}(\nabla \bu) = (1- |\nabla \bu|^a)^{-\frac{1}{a}} \nabla \bu$.

We adopt the following natural assumptions reflecting~\eqref{LSt2g}: there exist constants $C_0 \geq 0$ and $C_1,C_2>0$ such that, for all $\bT \in \mathbb{R}^{d\times N}$,
\begin{align}
\bD(\bT) \cdot \bT &\ge C_1|\bT| -C_0,\label{AE2-g}\\
|\bD(\bT)|&\le C_2. \label{AE1-g}
\end{align}
Furthermore, assuming that $\bD \colon \mathbb{R}^{d \times N} \rightarrow \mathbb{R}^{d \times N}$ is a $\mathcal{C}^1$ mapping we define the mapping $\cA \colon \mathbb{R}^{d\times N}\to \mathbb{R}^{d\times N}\times \mathbb{R}^{d\times N}$ as\begin{equation}\label{DefcA-g}
\cA(\bT)\coloneqq \frac{\partial \bD(\bT)}{\partial \bT}, \qquad \textrm{i.e.,} \qquad \cA_{i\nu j \mu}(\bT)\coloneqq \frac{\partial \bD_{i\nu}(\bT)}{\partial \bT_{j\mu}}
\end{equation}
for all $i,j=1,\ldots, d$ and $\nu,\mu=1,\ldots,N$, and we require that~$\bD$ is uniformly $h$-monotone, i.e., we assume that there exists a positive nonincreasing continuous function $h \colon \mathbb{R}_+ \to \mathbb{R}_+$ such that, for all $\bT,\bB \in \mathbb{R}^{d\times N}$, one has
\begin{equation}\label{A1-g}
h(|\bT|)|\bB|^2\le (\bB, \bB)_{\cA(\bT)}\coloneqq \sum_{i,j=1}^d \sum_{\nu,\mu=1}^N \cA_{i\nu j \mu }(\bT) \bB_{i\nu}\bB_{j\mu}\le \frac{C_2|\bB|^2}{1+|\bT|}.
\end{equation}
We note that, for the example~\eqref{LSt2g}, the condition~\eqref{A1-g} holds with $h(s)=(1+s^a)^{-1-\frac{1}{a}}$, $s\in \mathbb{R}_+$, $a>0$, and $C_2 = \max\{1,2^{1-\frac{1}{a}}\}$.

The class of problems~\eqref{elastic-g} with the structural assumption~\eqref{LSt2g} is not just of theoretical importance: it is closely related to
\emph{limiting strain models} in continuum mechanics, and this is in fact the second type of problem we are interested in. Its formulation can be given in the following way:
for a bounded, connected, Lipschitz domain $\Omega \subset \mathbb{R}^d$, $d\ge 2$, with Dirichlet boundary $\Gamma_D$ and Neumann boundary $\Gamma_N$, which are {relatively open subsets of $\partial \Omega$} such that $\Gamma_{D} \cap \Gamma_{N}=\emptyset$ and $\overline{\Gamma_D \cup \Gamma_N}=\partial \Omega$, a given vector field $\bef:\Omega \to \mathbb{R}^d$, a given $\bg:\Gamma_N \to \mathbb{R}^d$, a given boundary datum $\bu_0:\overline\Omega \to \mathbb{R}^d$, and a given \emph{bounded} mapping $\beps^* \colon \mathbb{R}^{d\times d}_{sym} \to \mathbb{R}^{d\times d}_{sym}$, we seek a couple $(\bu, \bT)$, the displacement and the Cauchy stress tensor, such that $\bu:\overline{\Omega} \to \mathbb{R}^d$, $\bT:\overline{\Omega} \to \mathbb{R}^{d\times d}_{sym}$, and
\begin{equation}\label{elastic}
\begin{aligned}
-\diver \bT &= \bef, \qquad \beps(\bu)=\beps^*(\bT) &&\textrm{ in } \Omega,\\
\bu&=\bu_0 &&\textrm{ on } \Gamma_D,\\
\bT\bn &=\bg &&\textrm{ on } \Gamma_N,
\end{aligned}
\end{equation}
where $\beps(\bu)$ is the linearized strain tensor, i.e., $\beps(\bu)\coloneqq\frac12 (\nabla \bu + (\nabla \bu)^T)$. A prototypical example of a limiting strain model is one in which the linearized strain tensor and the Cauchy stress are related by the formula
\[\beps(\bu) = \beps^\ast(\bT)\coloneqq\frac{\bT}{(1 + |\bT|^a)^{\frac{1}{a}}},\qquad a>0.\]

Problem~\eqref{elastic} is then an almost direct analogue of problem~\eqref{elastic-g} with $N=d$; the only aspect in which the latter model differs from~\eqref{elastic-g} (and is therefore considerably more difficult) is that, in contrast with~\eqref{elastic-g}, one is forced to operate in the space of symmetric matrices and function spaces of symmetric gradients. We refer the interested reader to \cite{KRR2003,ref30,KRRref33,BMRW2013,BMRS2014} for a detailed overview of limiting strain models, their theoretical justification stemming from implicit constitutive theory, a discussion of their importance in modeling the responses of materials near regions of stress-concentration, where $|\bT|$ is large, and their mathematical analysis (see in particular the survey paper~\cite{BMRS2014} for more details).

{Analogously to problem~\eqref{elastic-g}}, we adopt the following natural assumptions associated with limiting strain models (see~\cite{BMRS2014}): there exist constants $C_0 \geq 0$ and $C_1,C_2>0$ such that, for all $\bT \in \mathbb{R}^{d\times d}_{sym}$,
\begin{align}
\beps^*(\bT) \cdot \bT &\ge C_1|\bT| -C_0,\label{AE2}\\
|\beps^*(\bT)|&\le C_2. \label{AE1}
\end{align}
We also assume that $\beps^* \colon \mathbb{R}^{d\times d}_{sym} \to \mathbb{R}^{d\times d}_{sym}$ is a $\mathcal{C}^1$ mapping and define the mapping $\cA \colon \mathbb{R}^{d\times d}_{sym}\to \mathbb{R}^{d\times d}_{sym}\times \mathbb{R}^{d\times d}_{sym}$ as
\begin{equation}\label{DefcA}
\cA(\bT)\coloneqq \frac{\partial \beps^*(\bT)}{\partial \bT}, \qquad \textrm{i.e.,} \qquad \cA_{ijkl}(\bT)\coloneqq \frac{\partial \beps_{ij}^*(\bT)}{\partial \bT_{kl}}.
\end{equation}
Moreover, we require that~$\beps^*$ is uniformly $h$-monotone, i.e., we assume that there exists a positive nonincreasing continuous function $h \colon \mathbb{R}_+ \to \mathbb{R}_+$, such that, for all $\bT,\bB \in \mathbb{R}^{d\times d}_{sym}$, one has
\begin{equation}\label{A1}
h(|\bT|)|\bB|^2\le (\bB, \bB)_{\cA(\bT)}\coloneqq \sum_{i,j,k,l=1}^d \cA_{ijkl}(\bT) \bB_{ij}\bB_{kl}\le \frac{C_2|\bB|^2}{1+|\bT|}.
\end{equation}

As a precursor to our choice of the appropriate function spaces for the solution pair, we make the following observations: the system ~\eqref{elastic-g} yields boundedness of $|\nabla \bu|$ as a direct consequence of the boundedness of the mapping $\bT \mapsto |\bD(\bT)|$ (compare also with~\eqref{LSt2g}); analogously, the system~\eqref{elastic} yields boundedness of $|\beps(\bu)|$. On the other hand, {for both problems, our assumptions will only supply} an $L^1(\Omega)$ norm bound on the unknown~$\bT$. This can be viewed as a counterpart of the situation one faces with minimal surface type equations, corresponding to $\bD^{-1}$ rather than~$\bD$ itself being a bounded mapping. For example, one can consider the following counterpart of~\eqref{LSt2g}:
\begin{equation}\label{LSt2gcp}
\bT = \frac{\nabla \bu}{(1+|\nabla \bu|^a)^{\frac{1}{a}}},\qquad a>0,
\end{equation}
which, in tandem with $ - \diver \bT = \bef$, corresponds to the classical minimal surface equation if one sets $N=1$ and $a=2$. Similarly to the minimal surface equation, the natural function space for~$\bT$ in problem~\eqref{elastic-g} is not $L^1(\Omega)^{d \times N}$ but rather the space of Radon measures $\mathcal{M}$. The main purpose of this paper is to show that such an extension of the notion of solution to the space of Radon measures {\em is unnecessary}, provided that one equips problem~\eqref{elastic-g} with asymptotically symmetric and asymptotic Uhlenbeck structures. We postpone the definitions of \emph{asymptotically symmetric structure} and \emph{asymptotic Uhlenbeck structure} to Section~\ref{Sec2}, where the main results of this paper are precisely stated.

{The purpose of the remaining part of this section is to introduce a suitable concept of solution. To this end, we will assume for the moment that the fourth-order tensor $\mathcal{A}$ is symmetric, which then guarantees the existence of a potential~$F$ for the nonlinearity~$\bD$. This allows us to link our problem with an associated problem in the field of Calculus of Variations, where problems of this type have been studied for some time; summarizing the available existence results and counterexamples helps to motivate the concept of weak solution used in our approach (see Subsection~\ref{SubSec4}).}

\subsection{Notation} Before proceeding further, we introduce the relevant notational conventions that will be used throughout the paper. We shall use the standard notations $L^{p}(\Omega)$ and $W^{1,p}(\Omega)$ for Lebesgue spaces and Sobolev spaces, respectively. The space of Radon measures on a set $\Omega' \subseteq \overline{\Omega}$ (which need not necessarily be an open subset of $\Omega$) will be denoted by $\mathcal{M}(\Omega')$. In addition, we denote by $W^{1,p}_{\Gamma_D}(\Omega)$ the Sobolev space of functions having zero trace on $\Gamma_D$; in what follows we shall suppose that $\Gamma_D$ is sufficiently smooth so as to ensure that the following characterization holds for all $p \in [1,\infty)$:
\begin{equation}\label{zerotrace}
W^{1,p}_{\Gamma_D}(\Omega)=\overline{\left\{u\in \mathcal{C}^{\infty}(\overline{\Omega}) \colon \spt u \cap \overline{\Gamma_D}=\emptyset\right\}}^{\|\cdot\|_{1,p}}.
\end{equation}
In addition, to simplify the notational conventions, when $\Gamma_D = \emptyset$ it will be understood that
\begin{equation*}
W^{1,p}_{\Gamma_D}(\Omega)\coloneqq\left\{u\in W^{1,p}(\Omega) \colon \int_{\Omega} u \ddd x=0\right\},
\end{equation*}
and $\bu_0$ will then be {supposed} to be identically equal to $\mathbf{0}$ on $\overline \Omega$.

In order to distinguish scalar-, vector- and tensor-valued functions, we shall use italic letters for scalars (e.g., $u$), boldface letters for vectors (e.g., $\bu\coloneqq(\bu_1,\ldots, \bu_N)$), capital bold letters for $d\times N$ matrices (e.g., $\bB\coloneqq(\bB_{i\nu})$), and calligraphic letters for fourth-order tensors (e.g., $\cA\coloneqq(\cA)_{i\nu j \mu}$ with $i,j=1,\ldots, d$ and $\mu,\nu=1,\ldots, N$). Moreover, in what follows the sub- and superscripts written in italics will be understood to take the values $i=1,\ldots,d$, while the sub- and superscripts indicated in Greek letters take the values $\nu=1,\ldots,N$. We also use the following abbreviations for function spaces of vector- and tensor-valued functions:
$$
\begin{aligned}
X(\Omega)^d&\coloneqq\underset{d-\textrm{times}}{\underbrace{X(\Omega)\times \ldots \times X(\Omega)}},\\
X(\Omega)^{d\times N}&\coloneqq\underset{d\times N-\textrm{times}}{\underbrace{X(\Omega)\times \ldots \times X(\Omega)}}.
\end{aligned}
$$
In addition, we shall frequently use the symbol $\langle a, b\rangle\coloneqq\langle a, b\rangle_{X^*, X}$ for a dual pairing and will omit the subscript $_{X^*, X}$ whenever there is no ambiguity regarding the choice of the spaces $X$ and $X^*$. Finally, $\mathbb{R}_+$ will signify the set of all nonnegative real numbers.

\subsection{Assumptions on the data}
Concerning the vector function $\bef$ appearing on the right-hand side of~\eqref{elastic-g} we assume that
\begin{equation}
\tag{D1} \bef \in L^2(\Omega)^N. \label{D1}
\end{equation}
This condition can be relaxed: the square-integrability of $\bef$ is assumed here for the sake of simplicity only.

Next, we need to assume certain compatibility of the data. We require that
\begin{equation}
\tag{D2}\label{proto-G2}
\begin{aligned}
&\bu_0 \in W^{1,\infty}(\Omega)^N \mbox{ with $\nabla \bu_0(x)$ for a.e.~$x \in \overline\Omega$ contained in a compact set $K$ in $\mathbb{R}^{d\times N}$},\\
&\mbox{which is, in turn, contained in the interior of $\bD(\mathbb{R}^{d \times N})$.}
\end{aligned}
\end{equation}
This condition is trivially satisfied for each constant function $\bu_0$ (thanks to~\eqref{AE2-g} and Brouwer's fixed point theorem, cf.~the proof of Lemma~\ref{basici}); this is so in particular when $\Gamma_D=\emptyset$ (where, in line with the notational convention adopted following equation \eqref{zerotrace}, we have taken $\bu_0$ to be identically~$\mathbf{0}$ on $\overline{\Omega}$). Otherwise, when $\Gamma_D \neq \emptyset$, a sufficient condition for~\eqref{proto-G2} to be satisfied is the requirement

\begin{equation}
\label{D2_simple}
\begin{aligned}
\bu_0 \in W^{1,\infty}(\Omega)^N \text{ with } \|\nabla \bu_0\|_{\infty} < C_1.
\end{aligned}
\end{equation}
\begin{Remark}
It should be noted here that, because of hypothesis~\eqref{AE1-g}, for any plausible solution pair $(\bu, \bT)$ obeying $\nabla \bu = \bD(\bT)$, the function $\nabla \bu$ should necessarily satisfy $\|\nabla \bu\|_\infty \leq C_2$, where, due to~\eqref{AE2-g}, $C_2 \geq C_1$. As the range of~$\bD$ is potentially a strict subset of the closed ball of radius $C_2$ in $\mathbb{R}^{d\times N}$,
the condition $\|\nabla \bu_0\|_{\infty} \leq C_2$ on $\bu_0$ would not guarantee that $\nabla \bu_0(\bx) \in \bD(\mathbb{R}^{d\times N})$ for all $\bx \in \overline\Omega$; on the other hand,~\eqref{D2_simple} does imply that $\nabla \bu_0(\bx) \in \bD(\mathbb{R}^{d\times N})$ for all $\bx \in \overline\Omega$. In fact, for the prototypical case~\eqref{LSt2g}, condition~\eqref{D2_simple} is equivalent to~\eqref{proto-G2}. In any case,
we shall assume~\eqref{proto-G2} rather than, the generally stronger requirement,~\eqref{D2_simple}.
\end{Remark}

The condition~\eqref{proto-G2}, which excludes the possibility that $\nabla \bu_0(\bx)$ touches the boundary of the set $\bD(\mathbb{R}^{d \times N})$, may also be reformulated as the requirement
\begin{equation}
\label{G2_analytic}
\begin{aligned}
\bu_0 \in W^{1,\infty}(\Omega)^N \text{ with } 0<C_1&\le \liminf_{n\to \infty}\bigg(\essinf_{x\in \Omega} \inf_{\bT \in \mathbb{R}^{d\times N}; |\bT|=1}\left(\bD(n\bT)-\nabla \bu_0(x)\right)\cdot \bT\bigg),
\end{aligned}
\end{equation}
where, by selecting $C_1$ to be sufficiently small, we may use the same constant $C_1$ as in~\eqref{AE2-g}. This equivalence between~\eqref{proto-G2} and~\eqref{G2_analytic} is stated in Appendix~\ref{Saux} as Lemma~\ref{G2-equiv}. Let us note at this stage that such a condition is necessary in order to obtain the required a~priori bounds on~$\bT$ in $L^1(\Omega)^{d\times N}$, and it is quite natural e.g., in problems of plasticity. If~$\bu$ is interpreted as the displacement of an elastic body from its initial configuration, then~\eqref{proto-G2} can be seen to be a ``safety condition'' on the displacement gradient, which motivates us to call it \emph{safety strain condition}.

For the problem~\eqref{elastic}, because of the presence of the symmetric gradient, we adopt the following \emph{safety strain condition}: we require that
\begin{equation}
\tag{D2$^*$}\label{proto-G2st}
\begin{aligned}
& \bu_0 \in W^{1,1}(\Omega)^d \text{ with } \mbox{$\beps(\bu_0(x))$ for a.e.~$x \in \overline\Omega$ is contained in a compact set $K$ in $\mathbb{R}^{d\times d}$},\\
&\mbox{which is, in turn, contained in the interior of $\beps^*(\mathbb{R}^{d \times d})$.}
\end{aligned}
\end{equation}
Similarly as before (see, again, Lemma~\ref{G2-equiv}), this can be reformulated as the requirement
\begin{equation}
\label{G2st_analytic}
 \bu_0 \in W^{1,1}(\Omega)^d \text{ with } 0< C_1 \le \liminf_{n\to \infty}\bigg(\essinf_{x\in \Omega} \inf_{\bT \in \mathbb{R}^{d\times d}_{sym}; |\bT|=1}\left(\beps^*(n\bT)-\beps(\bu_0(x))\right)\cdot \bT\bigg).
\end{equation}

Furthermore, $\bg$ will be assumed to be integrable over $\Gamma_N$, and in the case when $\Gamma_D=\emptyset$ (and therefore $\Gamma_N=\partial \Omega$) we shall also assume the usual natural compatibility condition on the Neumann datum $\bg$ and the source term~$\bef$:
\begin{equation}
\tag{D3}\label{D3}
\begin{aligned}
\bg&\in L^1(\Gamma_N)^N,\\
\b0&=\int_{\Omega} \bef \ddd x + \int_{\partial \Omega}\bg \ddd S &&\textrm{ if } \Gamma_N=\partial \Omega.
\end{aligned}
\end{equation}

\subsection{On the concept of a solution and its connection to related problems in the Calculus of Variations} \label{SubSec4}
Having introduced the minimal assumptions on the data we can now focus on the appropriate definition of a solution to~\eqref{elastic-g}. Thus, in the rest of this subsection we shall always assume that~$\bD$ satisfies~\eqref{AE2-g}--\eqref{A1-g} and that the data satisfy~\eqref{D1}--\eqref{D3}. A first, apparently natural, possibility to define a solution is the following (we note here again that, by definition, we set $\bu_0 = \mathbf{0}$ when $\Gamma_D =\emptyset$).

\begin{Definition}\label{def1.1}
Let~$\bD$ satisfy~\eqref{AE2-g}--\eqref{A1-g} and let the data satisfy~\eqref{D1}--\eqref{D3}. We say that a couple $(\bu,\bT)$ is a \emph{weak solution} to~\eqref{elastic-g} if $\bu-\bu_0 \in W^{1,\infty}_{\Gamma_D}(\Omega)^N$, $\bT \in L^1(\Omega)^{d\times N}$ and
\begin{equation}\label{TT1*-G}
\begin{split}
\int_{\Omega}\bT\cdot \nabla \bw\ddd x &= \int_{\Omega} \bef \cdot \bw \ddd x +\int_{\Gamma_N}\bg\cdot \bw \ddd S \qquad \textrm{ for all } \bw \in W^{1,\infty}_{\Gamma_D}(\Omega)^N,\\
\nabla \bu &=\bD(\bT) \qquad \textrm{ in } \Omega.
\end{split}
\end{equation}
\end{Definition}

Unfortunately, such a definition is too restrictive and in general the solution in this sense may not exist even if the data are arbitrarily smooth and small. Indeed, it was shown in~\cite{BMRW2013} that in certain cases (namely if $d=N=2$) one can directly link the problem~\eqref{elastic-g},~\eqref{LSt2g} (by reformulating a geometrically special subproblem in terms of the Airy stress function) to its counterpart~\eqref{elastic-g2},~\eqref{LSt2gcp}, where for the special choice $a=2$ one obtains the minimal surface equation, which is in general unsolvable in nonconvex planar domains ($d=2$). Moreover, the same holds true for general $a\neq 2$, which was also observed in the series of papers \cite{BiFu99,BiFu02,BiFu02b,ApBiFu10}. On the other hand, it was pointed out in~\cite{BiFu02b} and~\cite{ApBiFu10} that restricting oneself to the case of $a \in (0,2/d)$ (with the upper bound $a=2/d$ included for the planar case $d=2$), which would, in a certain sense, correspond to the assumption that the function $h$ from~\eqref{A1-g} satisfies
\begin{equation}
h(|\bT|)\ge \frac{C_1}{1+|\bT|^q} \qquad \text{for some } q < 1+\frac{2}{d} \text{ (or } q \leq 2 \text{ for } d=2\text{)},
\label{BiFu}
\end{equation}
one may observe (when~\eqref{LSt2gcp} holds) that the failure of solvability of~\eqref{TT1*-G} is only due to the presence of the boundary part $\Gamma_D$. Then, by a proper redefinition of the notion of a solution, namely by allowing the nonattainment of $\bu_0$ on $\Gamma_D$, one can still formulate a satisfactory definition of a solution. We note in passing that an analogous situation occurs when one considers the counterpart of~\eqref{LSt2g} exemplified by~\eqref{LSt2gcp}. Thus, in the rest of this section we shall make a link to the available results when this `opposite' extreme behavior to~\eqref{LSt2g} is considered and we attempt to introduce a notion of solution that is more appropriate than~\eqref{TT1*-G}, and which, nevertheless, encapsulates the relevant information. To do so, we restrict ourselves for the moment to the potential case; this allows us to look at our problem by means of tools from the Calculus of Variations where problems with linear growth have been studied for some time. Motivated by the available results we introduce a concept of solution to our problems. We emphasize however that our results, stated precisely in Section~\ref{Sec2}, are not proved by techniques from the Calculus of Variations; in fact, PDE methods will be used. We also rely on the symmetry of the considered structure in a much weaker sense when proving our main results than in the existing literature. It is noteworthy that the results presented in Section~\ref{Sec2} seem to be the first ones of this kind, where one benefits from the (asymptotic) Uhlenbeck structure for the symmetric gradient.

In the rest of this subsection, following the goal to motivate the concept of solution, we assume temporarily that~$\cA$ is symmetric, i.e.,
$$
\cA_{i\nu j \mu}(\bT)=\cA_{j\mu i\nu }(\bT) \textrm{ for all } \bT \in \mathbb{R}^{d\times N} \textrm{ and all } i,j=1,\ldots, d \textrm{ and } \nu,\mu=1,\ldots, N.
$$
If this is the case, one can define a potential $F \colon \mathbb{R}^{d\times N}\to \mathbb{R}$ as
\begin{equation}
F(\bT)\coloneqq\int_0^1 \bD(t\bT)\cdot \bT\ddd t ,\label{potende}
\end{equation}
and by the symmetry of~$\cA$ (cf. Lemma~\ref{basici} in Appendix~\ref{Saux}) it then follows that
\begin{equation}
\frac{\partial F(\bT)}{\partial \bT_{i\nu}}=\bD_{i \nu }(\bT). \label{derpot}
\end{equation}
In addition, since~$\cA$ is $h$-monotone, the function~$F$ is strictly convex (see Lemma~\ref{basici}), and one is directly led to the variational formulation of~\eqref{elastic-g}. We note here that we have in principle two options: the primal formulation and the corresponding dual formulation. To this end, we also introduce by the usual formula the (convex) conjugate function~$F^*$ of~$F$:
\begin{equation}\label{conjug}
F^*(\bB)\coloneqq\sup_{\bT \in \mathbb{R}^{d\times N}} \{ \bB \cdot \bT - F(\bT) \}.
\end{equation}
It then follows from this definition that (see Lemma~\ref{basici})
\begin{equation}
\begin{aligned}
F^*(\bB)&=\infty &&\textrm{if } \bB \notin \overline{\bD(\mathbb{R}^{d\times N})},\\
F^*(\bB)&= \bB \cdot \bD^{-1}(\bB)- F(\bD^{-1}(\bB)) &&\textrm{if } \bB \in \bD(\mathbb{R}^{d\times N}).
\end{aligned}\label{propFstar}
\end{equation}
Note that the value of $F^*(\bB)$ can be finite or infinite for $\bB \in \partial \bD(\mathbb{R}^{d\times N})$ depending on the structure of~$F$. Moreover, by differentiating the expression in the second line of~\eqref{propFstar} we have that
\begin{equation}\label{derpot2}
\frac{\partial F^*(\bB)}{\partial \bB}=\bD^{-1}(\bB) \qquad \textrm{if } \bB \in \bD(\mathbb{R}^{d\times N}).
\end{equation}
We refer the reader to Lemma~\ref{basici} and its proof given in Appendix~\ref{Saux} for the above relations. Having introduced the potential~$F$, we are directly led to the definition of a solution to~\eqref{elastic-g} in terms of minimizers of a variational problem. Let us define to this end the admissible class of tensor functions~$\bT$ as
$$
\mathcal{S}\coloneqq\left\{\bT\in L^1(\Omega)^{d\times N} \colon \int_{\Omega} \bT \cdot \nabla \bw -\bef\cdot \bw\ddd x =\int_{\Gamma_{N}}\bg \cdot \bw\ddd S \textrm{ for all }\bw \in W^{1,\infty}_{\Gamma_D}(\Omega)^N \right\}
$$
and the admissible class of vector functions~$\bu$ as
$$
\mathcal{S}^*\coloneqq\big\{\bu \in W^{1,\infty}(\Omega)^N \colon \bu - \bu_0 \in W^{1,\infty}_{\Gamma_D}(\Omega)^N \big\}.
$$
We can then introduce the following two variational problems.

\smallskip

\noindent
{\bf Primal problem:} Find $\bu \in \mathcal{S}^*$ such that, for all $\bv \in \mathcal{S}^*$,
\begin{equation}\label{minimiz-prim}
J^*(\bu)\coloneqq\int_{\Omega} F^*(\nabla \bu)- \bef \cdot \bu \ddd x -\int_{\Gamma_N} \bg \cdot \bu\ddd S\le \int_{\Omega} F^*(\nabla \bv)- \bef \cdot \bv \ddd x -\int_{\Gamma_N} \bg \cdot \bv\ddd S.
\end{equation}

\smallskip

\noindent
{\bf Dual problem:} Find $\bT \in \mathcal{S}$ such that, for all $\bW \in \mathcal{S}$,
\begin{equation}\label{minimiz}
J(\bT)\coloneqq\int_{\Omega} F(\bT)-\nabla \bu_0 \cdot \bT \ddd x \le \int_{\Omega} F(\bW)-\nabla \bu_0 \cdot \bW \ddd x.
\end{equation}

\smallskip

\noindent
Moreover, we can rewrite the definition of a weak solution to~\eqref{elastic-g} (in the sense of Definition~\ref{def1.1} above) as follows.

\smallskip

\noindent
{\bf Weak solution:} Find $(\bu, \bT) \in \mathcal{S}^*\times \mathcal{S}$ such that
\begin{equation}\label{defwekznovu}
\nabla \bu = \bD(\bT) \qquad \textrm{in } \Omega.
\end{equation}

Next, we make links between these various formulations and we also discuss the main difficulties. To do so, we state the following three lemmas, whose proofs are based on standard techniques from convex analysis and can be found in Appendix~\ref{app_minimizers_ws}.

\begin{Lemma}\label{AL1}
Let $\Omega \subset \mathbb{R}^d$ be a bounded Lipschitz domain, let the mapping~$\bD$ satisfy~\eqref{AE2-g}--\eqref{A1-g} with symmetric~$\cA$ and let the data satisfy~\eqref{D1}--\eqref{D3}. Assume that $(\bu,\bT)$ is a weak solution to~\eqref{TT1*-G}. Then,~$\bu$ solves~the primal problem~\eqref{minimiz-prim} and~$\bT$ solves the dual problem. Moreover, if a weak solution exists, then it is unique.
\end{Lemma}

\begin{Lemma} \label{AL3}
Let $\Omega \subset \mathbb{R}^d$ be a bounded Lipschitz domain, let the mapping~$\bD$ satisfy~\eqref{AE2-g}--\eqref{A1-g} with symmetric~$\cA$ and let the data satisfy~\eqref{D1}--\eqref{D3}; then, the infimum of $J(\bT)$ over $\bT \in \mathcal{S}$ is finite. Moreover, if the infimum is attained for some $\bT\in \mathcal{S}$, then there exists a function $\bu \in W^{1,\infty}(\Omega)^N$ such that $\bu-\bu_0\in W^{1,\infty}_{\Gamma_D}(\Omega)^N$ and $\nabla \bu = \bD(\bT)$ in $\Omega$, and, consequently, the couple $(\bu,\bT)$ is a unique weak solution.
\end{Lemma}

\begin{Lemma}\label{AL2}
Let $\Omega \subset \mathbb{R}^d$ be a bounded Lipschitz domain, let the mapping~$\bD$ satisfy~\eqref{AE2-g}--\eqref{A1-g} with symmetric~$\cA$ and let the data satisfy~\eqref{D1}--\eqref{D3}; then, there exists a unique minimizer~$\bu$ to~the primal problem. Moreover, there exists a unique $\bT\in L^1(\Omega)^{d\times N}$ such that $\nabla \bu = \bD(\bT)$ in $\Omega$ and
\begin{equation}\label{nerovnost-g}
\int_{\Omega} \bT \cdot \nabla (\bu-\bv)\ddd x \le \int_{\Omega} \bef\cdot (\bu-\bv)\ddd x + \int_{\Gamma_{N}} \bg \cdot (\bu-\bv)\ddd S
\end{equation}
for all $\bv \in W^{1,\infty}(\Omega)^N$ such that $\bv -\bu_0 \in W^{1,\infty}_{\Gamma_D}(\Omega)^N$ and $\bD^{-1}(\nabla \bv) \in L^1(\Omega)^{d\times N}$. Furthermore, if for each $\bw \in W^{1,\infty}_{\Gamma_D}(\Omega)^N$ there exists a sequence $(\bw^n)_{n \in \mathbb N}$ in $W^{1,\infty}_{\Gamma_D}(\Omega)^N$ such that
\begin{equation}\label{novyas}
\spt \nabla \bw^n \subset \{x\in \Omega \colon |\bT(x)|\le n\} \quad \text{and} \quad \bw^n \rightharpoonup^* \bw \textrm{ weakly-$*$ in }W^{1,\infty}_{\Gamma_D}(\Omega)^N,
\end{equation}
then $(\bu,\bT)$ is a weak solution.
\end{Lemma}

We elaborate on these results in more detail. We see from Lemmas~\ref{AL1}--\ref{AL2} that finding a weak solution to the problem under consideration is equivalent to finding a minimizer to the dual problem. {Unfortunately}, {even if} the infimum of $J$ exists and is finite, we cannot claim that it is attained for some $\bT \in \mathcal{S}$ because of the nonreflexivity of the space $L^1$, although it is in general attained for some $\bT \in \mathcal{M}(\Omega \cup \overline{\Gamma_N})^{d\times N}$. On the other hand, we see that the solution to the primal problem always exists and is unique, and therefore~$\bu$ is a good candidate for being the first component of a weak solution couple $(\bu, \bT)$ to the original weak formulation. However, because~\eqref{nerovnost-g} is only an inequality rather than an equality, we cannot claim that $\bT \in \mathcal{S}$; in particular,~$\bT$ may have a singular part that penalizes~$\bT$ in order to ensure that it belongs to
$\mathcal{S}$. It therefore seems natural, in view of Lemmas~\ref{AL3} and~\ref{AL2}, to define a weaker notion of solution. A natural idea is to weaken the topology of the set of admissible functions~$\bT$ so that it is merely a weakly-$*$ closed set. Therefore, we modify $\mathcal{S}$ to $\mathcal{S}^m$, where
$$
\mathcal{S}^m\coloneqq\left\{\bT\in \mathcal{M}(\Omega \cup \overline{\Gamma_N})^{d\times N} \colon \langle\bT, \nabla \bw\rangle-\int_{\Omega}\bef\cdot \bw\ddd x =\int_{\Gamma_{N}}\bg \cdot \bw\ddd S \textrm{ for all }\bw \in \mathcal{C}^1_{\Gamma_D}(\overline{\Omega})^N \right\}.
$$
One can then relax the notion of a minimizer to~the dual problem to the following.

\bigskip

\noindent
{\bf Relaxed dual problem:} Assume that $\bu_0 \in \mathcal{C}^1(\overline{\Omega})^N$. Find $\bT \in \mathcal{S}^m$ such that for all $\bW \in \mathcal{S}^m$ one has
\begin{equation}\label{minimiz5}
\mathcal{J}(\bT) \coloneqq \int_{\Omega} F(\bT^r)\ddd x+\mathcal{F}_{\infty}(\bT^s)-\langle \nabla \bu_0,\bT\rangle \le \int_{\Omega} F(\bW^r)\ddd x+\mathcal{F}_{\infty}(\bW^s)-\langle \nabla \bu_0, \bW\rangle.
\end{equation}
Here, we have used the decomposition of a general measure~$\bT$ as $\bT=\bT^r + \bT^s$, where\footnote{We denote here by $|\bT^r|$ the standard Euclidean norm of $\bT^r$, while $|\bT^s|$ denotes the Radon measure, which is the variation of the vector-valued measure $\bT^s$, defined, for any compact set $K$, as $|\bT^s|(K) \coloneqq \sup_{\{\bE\in \mathcal{C}(K)^{d\times N} \colon |\bE|\le 1\}} \langle \bT^s,\bE\rangle.$ } $|\bT^r|$ is absolutely continuous with respect to the Lebesgue measure and $\bT^s$ is a singular measure. Furthermore, $\mathcal{F}_{\infty}(\bT^s)$ denotes the \emph{recession functional}
$$
\mathcal{F}_{\infty}(\bT^s)\coloneqq\int_{\Omega \cup \overline{\Gamma_N}}F_\infty\left(\frac{\bT^s}{|\bT^s|}\right) \ddd |\bT^s|,
$$
where $F_{\infty}(\bT)\coloneqq\lim_{n\to \infty} \frac{1}{n} F(n\bT)$ stands for the recession function and $\bT^s / |\bT^s|$ for the Radon--Nikod\'{y}m density of $\bT^s$ with respect to $|\bT^s|$ (which is well-defined $|\bT^s|$-a.e.). Thanks to the convexity of~$F$ the recession function is well-defined (with values in $\mathbb{R}$ because of the linear growth of~$F$), it is strictly positive on $\mathbb{R}^{d \times N}
\setminus \{0\}$, positively $1$-homogeneous and convex. In turn, also the functional $\mathcal{J}$ is convex on $\mathcal{S}^m$ (cf. \cite[Proposition~2.37]{AMBFUSPAL00}). The significance of this \emph{relaxed problem} is encapsulated in the next lemma (see Appendix~\ref{app_minimizers_ws} for a sketch of its proof).

\begin{Lemma} \label{AL4}
Let $\Omega \subset \mathbb{R}^d$ be a bounded Lipschitz domain, let the mapping~$\bD$ satisfy~\eqref{AE2-g}--\eqref{A1-g} with symmetric~$\cA$ and let the data satisfy~\eqref{D1}--\eqref{D3}, with $\bu_0 \in \mathcal{C}^1(\overline{\Omega})^N$.
Then, the minimum of $\mathcal{J}$ in $\mathcal{S}^m$ is attained for some $\bT \in \mathcal{S}^m$, and for any two minimizers $\bT,\bar{\bT} \in \mathcal{S}^m$ we have
\begin{equation*}
 \bT^r = \bar{\bT}^r \quad \text{a.e.~in } \Omega \qquad \text{and} \qquad \langle\bT^s - \bar{\bT}^s, \nabla \bw\rangle = 0 \quad \textrm{for all } \bw \in \mathcal{C}^1_{\Gamma_D}(\overline{\Omega})^N.
\end{equation*}
Moreover, the unique minimizer~$\bu$ to the primal problem satisfies $\nabla \bu=\bD(\bT^r)$ and the couple $(\bu,\bT^r)$ satisfies~\eqref{nerovnost-g}.
\end{Lemma}

Hence, it is evident that for problems with linear growth the relaxed dual formulation is optimal. Indeed, it is known that one cannot expect a better behavior for general potentials~$F$, see e.g.~\cite{FINN65,soucek71,GIAMODSOU79,BILFUC03}, and therefore one cannot avoid the presence of a singular part in~$\bT$. On the other hand, it is of importance to identify at least the support of the singular measure $\bT^s$. Not only is this relevant from the point of view of PDE theory, but the location of the support of the singular measure also has important consequences in continuum mechanics when one considers limiting strain models~\eqref{elastic}. Here, the description of the singular part of~$\bT$, which is in the context of continuum mechanics the Cauchy stress tensor, directly leads to the prediction of sets where stress concentration may occur, and it therefore leads to the identification of the set where a crack in an elastic body may be initiated or where material discontinuities may appear. On the other hand, inspired by \cite{ApBiFu10,BiFu02b}, one may also hope that by assuming uniform monotonicity of~$\bD$ in a suitable sense (such as in~\eqref{BiFu}, for example,) one can avoid the presence of a singular part $\bT^s$ in the interior of the domain $\Omega$ and can show that such a singular part (if it exists) is supported only on the Neumann part, $\Gamma_N$, of the boundary. (The reader should compare this with~\cite{BiFu02b} and \cite[Sect. 7]{BMRW2013}, where an analogous result has been proved for the counterpart of our problem, and the problematic part of the boundary there is only $\Gamma_D$.) It seems quite surprising that there is no known example of a problem with an Uhlenbeck structure of the form~\eqref{LSt2gcp} that exhibits a singularity in the interior of the domain. Indeed, all known examples involving a singular measure in the interior of the domain $\Omega$ are either for nonsmooth coefficients or for very general non-Uhlenbeck structures. This also leads to the hypothesis that for problems of the form~\eqref{LSt2g} one can obtain a unique solution whose singular part is due to the presence of the Neumann part of the boundary $\Gamma_N$. This is, in fact, the main result of the paper, which, for the prototypical example with the particular choice~\eqref{LSt2g}, is as follows.

\begin{Theorem}\label{TH00}
Let $\Omega \subset \mathbb{R}^d$ be a bounded Lipschitz domain. Assume that the data $(\bef,\bg)$ satisfy~\eqref{D1}--\eqref{D3} and consider $a>0$. Then, there exists a pair $(\bu,\bT) \in \mathcal{S}^* \times L^1(\Omega)^{d\times N}$ satisfying
\begin{alignat}{3}
\label{TH00_weak}
\int_{\Omega}\bT\cdot \nabla \bw \ddd x & = \int_{\Omega} \bef \cdot \bw\ddd x \qquad &&\textrm{for all } \bw \in \mathcal{C}^1_0(\Omega)^N, \\
\label{TH00_limit_rep}
\nabla \bu &=\frac{\bT}{(1+|\bT|^a)^{\frac{1}{a}}} \qquad && \textrm{in } \Omega
\end{alignat}
and
\begin{equation}
\label{TH00_min}
\int_{\Omega} \bT \cdot \nabla (\bu-\bv)\ddd x \le \int_{\Omega} \bef \cdot (\bu-\bv)\ddd x + \int_{\Gamma_N}\bg \cdot (\bu-\bv)\ddd S
\end{equation}
for all $\bv \in W^{1,\infty}(\Omega)^N$ that are equal to $\bu_0$ on $\Gamma_D$ and satisfy, for some $\tilde{\bT}\in L^1(\Omega)^{d\times N}$, $\nabla \bv = (1+|\tilde{\bT}|^a)^{-\frac{1}{a}} \tilde{\bT}$ in~$\Omega$. Furthermore, there is a $\tilde{\bg}\in (\mathcal{C}^1_0(\Gamma_N)^d)^*$ such that
\begin{align}
\label{TH00-G}
\int_{\Omega}\bT\cdot \nabla \bw \ddd x = \int_{\Omega} \bef \cdot \bw\ddd x +\langle \bg-\tilde{\bg},\bw\rangle|_{\Gamma_N} \qquad \textrm{ for all } \bw \in \mathcal{C}^1_{\Gamma_D}(\overline{\Omega})^N.
\end{align}
In addition,~$\bT$ is the regular part of the minimizer to~\eqref{minimiz5}, for which the singular part $\bT^s$
has the unique representation
$$
\langle \bT^s ,\nabla \bw\rangle = -\langle \tilde{\bg} ,\bw\rangle \quad \text{for all }\bw \in \mathcal{C}^1_{\Gamma_D}(\overline{\Omega})^N
$$
and consequently $\diver \bT^s=0$ in $\Omega$ in the sense of distributions. 

If $a<2/d$ and $d\ge 3$ or if $a\le 1$ and $d=2$, then $\spt |\bT^s|\subset \overline{\Gamma_N}$ and $\bT\in L^{d(1-a)/(d-2)}_{loc}(\Omega)^{d\times N}$ for $d\ge 3$ and $\bT\in L^p_{loc}(\Omega)^{2\times N}$, $p \in [1,\infty)$, for $d=2$; if, on the other hand, $a\ge 2/d$ and $d \geq 3$ or $a>1$ and $d=2$, then $\spt |\bT^s|\subset \{x\in \Omega \cup \overline{\Gamma_N} \colon M(|\bT|)(x)=\infty\}$, where $M$ denotes the maximal function. Moreover, we have
$$
\frac{\bT^s}{|\bT^s|} = \lim_{\varepsilon \to 0_+} \frac{\nabla \bu^{\varepsilon}}{| \nabla \bu^{\varepsilon}|} = \lim_{\varepsilon \to 0_+}\left(\frac{\bT}{|\bT|}\right)^{\varepsilon}, \qquad (\varepsilon >0)
$$
where the limit is understood in $L^1(\Omega,|\bT^s|)$ and the superscript ${\varepsilon}$ denotes the standard $\varepsilon$-mollification.
\end{Theorem}

The next lemma, proved in Appendix~\ref{app_minimizers_ws}, provides additional information about the pair of functions $(\bu,\bT)$ whose existence is guaranteed by Theorem~\ref{TH00}.

\begin{Lemma}
\label{AL5}
Under the assumptions of Theorem~\ref{TH00} \textup(or, more generally, of Theorem~\ref{T1-G}\textup) the following is true: the function $\bu \in W^{1,\infty}(\Omega)^N$ is the unique minimizer of the primal problem, and, if $\tilde{\bg} = \mathbf{0}$, then $\bT \in L^1(\Omega)^{d\times N}$ is a minimizer of the dual problem.
\end{Lemma}

\begin{Remark}
In the special case $\Gamma_N = \emptyset$, Theorem~\ref{TH00} holds with $\tilde{\bg} = \b0$ and hence, in this case the function~$\bT$ is in fact a minimizer of $J$ in $\mathcal{S}$; this means that Theorem~\ref{TH00} guarantees the existence of a weak solution $(\bu, \bT)\in \mathcal{S}^* \times \mathcal{S}$ to~\eqref{elastic-g} with~$\bD$ given by~\eqref{LSt2g} for all $a>0$. This is a significant improvement of the theory for the problem~\eqref{elastic} developed earlier. More precisely, the equilibrium and constitutive equations in~\eqref{elastic} in combination with the requirement that the unknowns $(\bu, \bT)$ are spatially periodic were already analyzed in~\cite{BMS2014}, and the existence of weak solution was established for $a\in (0, 2/d)$. Also, the concept of renormalized solution was introduced there and its existence was proved for all $a>0$.

The problem~\eqref{elastic} with $\Gamma_N = \emptyset$ was investigated in~\cite{BMRS2014} and the authors proved the existence of a weak solution for $a\in (0, 1/d)$. A novel tool that enables us to achieve now the improvement to the full range $a > 0$ is presented in Subsect.~\ref{4.5}; we call this method \emph{renormalized} regularity. Based on the results stated in Theorem~\ref{TH00} and valid for the case $\Gamma_N = \emptyset$, one can conjecture that for a ``minimal surface" problem with Neumann boundary condition on the whole of $\partial \Omega$ a weak solution always exists.

Finally, we wish to emphasize that even in the general case $\Gamma_N \neq \emptyset$, we are able to show that the equilibrium equation $\diver \bT = \bef$ holds in $\Omega$ \textup(in the sense of distributions\textup) and
that, in analogy with the results available for the minimal surface equation, the presence of a nontrivial $\tilde{\bg}$ is essential in some cases. Of course, it is of interest to identify assumptions \textup(such as convexity of $\Omega$\textup) that guarantee $\tilde{\bg} = \b0$.
\end{Remark}

The rest of the paper is organized as follows. In Section~\ref{Sec2} we formulate more precisely the assumptions on the admissible choice of~$\bD$ and state the main result for the original problem~\eqref{elastic-g} and also for the limiting strain model~\eqref{elastic}. In particular, we do not assume the symmetry of $\mathcal{A}$ (which will be relaxed to an \emph{asymptotically} symmetric structure) and we also distinguish between the case when~\eqref{BiFu} is assumed and when an asymptotic Uhlenbeck structure is involved. The relevant properties of~$\bD$ (as well as of the potential~$F$ and its conjugate~$F^*$), which are used throughout the paper, are mostly standard results from convex analysis, but for the convenience of the reader they are provided in Appendix~\ref{Saux}. Moreover, we collect the proofs of the Lemmas~\ref{AL1}--\ref{AL5} in Appendix~\ref{app_minimizers_ws}. Section~\ref{S2} is then concerned with the proof of uniqueness of the solution. Section~\ref{S3} is the core of this paper: it contains the proof of the existence of solutions. Because of the linear growth setting we need to work here with approximations of the problem, for which various a priori estimates are derived. The proof of the main result relies heavily on the concept of a renormalized weak solution, on a new technique for the identification of the limit in~\eqref{TH00_limit_rep} of the approximations, and the justification of the weak formulation in~\eqref{TH00_weak}.

For the sake of brevity, we shall confine ourselves to the proofs for the limiting strain model (involving the symmetric gradient; cf.~\eqref{elastic}) because this model is considerably more difficult to analyze than the analogous model involving the full gradient. In fact, to the best of our knowledge, this is the first result of this kind where one benefits from the Uhlenbeck structure for the symmetric gradient.

\section{Statement of the main result}\label{Sec2}
Our first aim is to establish results concerning properties of the problem~\eqref{elastic-g} assuming that the data satisfy~\eqref{D1}--\eqref{D3} and the nonlinear function~$\bD$ satisfies~\eqref{AE2-g}--\eqref{A1-g} (in particular,~$\bD$ is $h$-monotone).

We saw in the previous section (motivated also by the results in~\cite{BiFu02b}) that the possibility for introducing a potential~$F$ was essential in order to overcome the difficulties with linear growth. We shall therefore assume in what follows that~$\cA$ is \emph{asymptotically symmetric}, i.e., by denoting
\begin{equation}\label{Asymm}
\cA^{s}(\bT)\coloneqq \frac12 (\cA(\bT)+\cA^T(\bT)), \quad \textrm{i.e.,} \quad \cA^{s}_{i\nu k \mu}(\bT)\coloneqq\frac12 (\cA_{i\nu k \mu}(\bT)+\cA_{k\mu i\nu}(\bT)),
\end{equation}
we assume that (with $h$ as in~\eqref{A1-g})
\begin{equation}
\frac{\left|\cA^{s}(\bT)-\cA(\bT)\right|^2}{h(|\bT|)}\le \frac{C_2}{1+|\bT|}. \label{A3-g}
\end{equation}
We note at this point that, independently of the assumption~\eqref{A3-g}, the function $h$ has, as a direct consequence of hypotheses~\eqref{AE2-g}--\eqref{A1-g}, the implied asymptotic property $h(s)\,s \rightarrow 0$ as $s\rightarrow \infty$.\footnote{This can be shown by taking $\bT_1 = \bT$ and $\bT_2 = \mathbf{0}$ in
the first three displayed lines of mathematics in the proof of Lemma~\ref{basici} to deduce first that
$\int_0^\infty h(s)\,{\rm d}s < \infty$. Hence, thanks to the assumed monotonicity of $h$, {we first find} $\sum_{k=1}^{\infty} h(k) < \infty$ and {then} $h(n)\, n
{\leq 2 \sum_{k=\lfloor n/2 \rfloor}^{\infty} h(k)} \to 0$ {as $n \to \infty$}. Given any real number $s\geq 1$, again thanks to the monotonicity of $h$, we have $0 \leq h(s)\,s \leq h( \lfloor s \rfloor)\, s = h( \lfloor s \rfloor) \lfloor s \rfloor s / \lfloor s \rfloor$, and therefore $h(s)\, s \to 0$ as $s \to \infty$ thanks to $h(n)\,n\to 0$ as $n\rightarrow \infty$, with $n=\lfloor s \rfloor$, and the fact that $1\leq s/[s]<2$.} Therefore, keeping in mind that $|\cA(\bT)|$ is bounded by $C_2 (1 + |\bT|)^{-1}$ due to~\eqref{A1-g}, we observe that condition~\eqref{A3-g} is in general not implied by our previous assumptions.

The second key assumption of the paper is twofold. We shall assume either the uniform monotonicity condition~\eqref{BiFu} on~$h$, which does not require further structure; or, if~\eqref{BiFu} is not valid, then we shall require that~$\bD$ has the \emph{asymptotic Uhlenbeck structure}\footnote{Nonlinear elliptic systems of the form
\[ - \diver\, (\mathcal{B}(|\nabla \bu|) \nabla \bu) = \bef,\]
where the coefficient $\mathcal{B}$ only depends on the matrix norm of the gradient of the solution, are referred to in the literature as systems with Uhlenbeck structure~\cite{Uhl}. Equivalently, we can write
\[ - \diver \bT = \bef \quad \mbox{with $~~\bT = \mathcal{B}(|\nabla \bu|) \nabla \bu$}.\]
In the present paper we shall be, instead, concerned with problems of the form
\begin{equation}\label{star}
 - \diver \bT = \bef \quad \mbox{with $\nabla \bu = \mathcal{H}(|\bT|)\bT$}.
\end{equation}
Since our assumptions on the nonlinear function $\mathcal{A}$ will be such that they will ensure an equivalent restatement of the relationship $\nabla \bu = \mathcal{H}(|\bT|) \bT$ as $\bT = \mathcal{B}(|\nabla \bu|) \nabla \bu$, it is natural to refer to the elliptic problems~\eqref{star} as having Uhlenbeck structure. Similarly, we call also the systems with $\mathcal{B}$ depending on the matrix norm of the symmetric part of the gradient $|\beps(\bu)|$ \emph{systems with Uhlenbeck structure}.}, i.e., we shall assume that there exists a nonnegative continuous function $g \colon \mathbb{R}_+\to \mathbb{R}_+$ with
\begin{equation}
g(t) \le C_2(1+t) \qquad \textrm{for all } t \in \mathbb{R}_+ \label{A4-g}
\end{equation}
such that, for all $\bT \in \mathbb{R}^{d\times N}$, one has
\begin{equation}
\frac{|g(|\bT|)\bD(\bT)-\bT|^2}{h(|\bT|)} \le C_2(1+|\bT|^3).\label{A5-g}
\end{equation}

Under each of these two additional assumptions we can now formulate the main result of the paper.

\begin{Theorem}\label{T1-G}
Let $\Omega \subset \mathbb{R}^d$ be a bounded Lipschitz domain. Assume that the data $(\bef,\bg)$ satisfy~\eqref{D1}--\eqref{D3} and that~$\bD$ satisfies~\eqref{AE2-g}--\eqref{A1-g} {and~\eqref{A3-g}}. In addition, let either ~\eqref{BiFu}, or~\eqref{A4-g} and~\eqref{A5-g} hold. Then, there exists a triple $(\bu,\bT,\tilde{\bg})\in W^{1,\infty}(\Omega)^N \times L^1(\Omega)^{d\times N}\times (\mathcal{C}^1_0(\Gamma_N)^d)^*$ such that
\begin{align*}
\bu&=\bu_0 \quad\qquad\textrm{ on } \Gamma_D,\\
\nabla \bu &=\bD(\bT)\qquad \textrm{ in } \Omega,
\end{align*}
which solves
\begin{equation}\label{TT1-G}
\int_{\Omega}\bT\cdot \nabla \bw \ddd x = \int_{\Omega} \bef \cdot \bw\ddd x +\langle \bg-\tilde{\bg},\bw\rangle|_{\Gamma_N} \qquad \textrm{ for all } \bw \in \mathcal{C}^1_{\Gamma_D}(\overline{\Omega})^N.
\end{equation}
In particular, $\diver \bT = \bef$ in the sense of distributions. In addition, the following inequality holds:
\begin{equation}\label{TT2-G}
\int_{\Omega} \bT \cdot \nabla (\bu-\bv)\ddd x \le \int_{\Omega} \bef \cdot (\bu-\bv)\ddd x + \int_{\Gamma_N}\bg \cdot (\bu-\bv)\ddd S,
\end{equation}
for all $\bv \in W^{1,\infty}(\Omega)^N$ that are equal to $\bu_0$ on $\Gamma_D$ and satisfy, for some $\tilde{\bT}\in L^1(\Omega)^{d\times N}$, $\nabla \bv =\bD(\tilde{\bT})$ in~$\Omega$. Moreover, the triple $(\bu,\bT,\tilde{\bg})$ is unique in the class of solutions satisfying~\eqref{TT1-G},~\eqref{TT2-G} provided that either $\Gamma_D\neq \emptyset$ or the integral mean-value of~$\bu$ is fixed.

Furthermore, there exists a $\bT^s\in \mathcal{M}(\Omega \cup \overline{\Gamma}_N)^{d\times N}$ having the unique representation
$$
\langle \bT^s ,\nabla \bw\rangle = -\langle \tilde{\bg} ,\bw\rangle \quad \text{for all }\bw \in \mathcal{C}^1_{\Gamma_D}(\overline{\Omega})^N
$$
and consequently $\diver \bT^s=0$ in $\Omega$ in the sense of distributions. 

Under the assumption~\eqref{BiFu}, we have in fact $\spt |\bT^s|\subset \overline{\Gamma_N}$, and $\bT\in L^{d(2-q)/(d-2)}_{loc}(\Omega)^{d\times N}$ for $d\ge 3$ and $\bT\in L^p_{loc}(\Omega)^{2\times N}$, $p \in [1,\infty)$, for $d=2$, while under the assumptions~\eqref{A4-g} and~\eqref{A5-g}, we have
$$\spt |\bT^s|\subset \{x\in \Omega \cup \overline{\Gamma_N} \colon M(|\bT|)(x)=\infty\},$$
where $M$ denotes the maximal function, and
$$
\frac{\bT^s}{|\bT^s|} = \lim_{\varepsilon \to 0_+} \frac{\nabla \bu^{\varepsilon}}{| \nabla \bu^{\varepsilon}|} = \lim_{\varepsilon \to 0_+}\left(\frac{\bT}{|\bT|}\right)^{\varepsilon}, \qquad (\varepsilon>0)
$$
where the limit is understood in $L^1(\Omega,|\bT^s|)$ and the superscript $\varepsilon$ denotes the standard mollification.

Finally, if~$\cA$ defined in~\eqref{DefcA-g} is symmetric, then~$\bT$ and $\bT^s$ are the regular and singular parts, respectively, of the minimizer to~\eqref{minimiz5}.
\end{Theorem}

The second theorem that we state here concerns the limiting strain problem~\eqref{elastic}, where instead of $\nabla \bu = \bD(\bT)$, with~$\bD$ as above, we consider $\beps(\bu) = \beps^\ast(\bT)$, and where, in analogy with~$\bD$, $\beps^\ast$ is a bounded function of its argument. Although the corresponding theorem has some similarities with the previous theorem, we must take into account the fact that only the symmetric part of the gradient (of the displacement~$\bu$) appears in the equation, and therefore we have adopted in Section~\ref{Sec1} slightly different assumptions on the possible structure of~$\beps^*$ than in the case when the full gradient $\nabla \bu$ depends nonlinearly on~$\bT$, see~\eqref{AE2}--\eqref{A1}.

Similarly as above in~\eqref{Asymm} and~\eqref{A3-g}, we also introduce the symmetric part of~$\cA$:
$$
\cA^{s}(\bT)\coloneqq \frac12 (\cA(\bT)+\cA^T(\bT)), \quad \textrm{i.e.,} \quad \cA^{s}_{ijkl}(\bT)\coloneqq\frac12 (\cA_{ijkl}(\bT)+\cA_{klij}(\bT)),
$$
and assume that
\begin{equation}
\frac{\left|\cA^{s}(\bT)-\cA(\bT)\right|^2}{h(|\bT|)}\le \frac{C_2}{1+|\bT|}. \label{A3}
\end{equation}
Furthermore, we shall either assume that $h$ satisfies~\eqref{BiFu} or assume that~$\beps^*$ has asymptotic Uhlenbeck structure, i.e., we assume that there exists a nonnegative continuous function $g \colon \mathbb{R}_+\to \mathbb{R}_+$ with
\begin{equation}
g(t) \le C_2(1+t) \qquad \textrm{for all } t \in \mathbb{R}_+ \label{A4}
\end{equation}
such that, for all $\bT \in \mathbb{R}^{d\times d}_{sym}$, one has
\begin{equation}
\frac{|g(|\bT|)\,\beps^*(\bT)-\bT|^2}{h(|\bT|)} \le C_2(1+|\bT|^3).\label{A5}
\end{equation}

Our main result for the limiting strain model is then the following theorem.
\begin{Theorem}\label{T1}
Let $\Omega \subset \mathbb{R}^d$ be a bounded Lipschitz domain and assume that the data $(\bef,\bg)$ satisfy~\eqref{D1},~\eqref{proto-G2st} and~\eqref{D3}. Assume further that~$\beps^*$ satisfies~\eqref{AE2}--\eqref{A1} together with~\eqref{A3} and that either~\eqref{BiFu} holds, or that~\eqref{A4} and~\eqref{A5} hold. Then, there exists a triple $(\bu,\bT,\tilde{\bg})\in W^{1,1}(\Omega)^d\times L^1(\Omega)^{d\times d}\times (\mathcal{C}^1_0(\Gamma_N)^d)^*$ such that
\begin{alignat*}{3}
\bu& \in W^{1,p}(\Omega)^d &&\quad\textrm{ for all } p\in [1,\infty),\\
\beps(\bu)& \in L^{\infty}(\Omega)^{d\times d},\\
\bu&=\bu_0 &&\quad\textrm{ on } \Gamma_D,\\
\beps(\bu)&=\beps^*(\bT) &&\quad\textrm{ in } \Omega,
\end{alignat*}
which solves
\begin{equation}\label{TT1}
\int_{\Omega}\bT\cdot \beps(\bw)\ddd x = \int_{\Omega} \bef \cdot \bw\ddd x +\langle \bg-\tilde{\bg},\bw\rangle|_{\Gamma_N} \qquad \textrm{ for all } \bw \in \mathcal{C}^1_{\Gamma_D}(\overline{\Omega})^d.
\end{equation}
In particular, $\diver \bT = \bef$ in the sense of distributions. In addition, the following inequality holds:
\begin{equation}\label{TT2}
\int_{\Omega} \bT \cdot (\beps(\bu)-\beps(\bv))\ddd x \le \int_{\Omega} \bef \cdot (\bu-\bv)\ddd x + \int_{\Gamma_N}\bg \cdot (\bu-\bv)\ddd S,
\end{equation}
for all $\bv \in W^{1,1}(\Omega)^d$ that are equal to $\bu_0$ on $\Gamma_D$ and satisfy, for some $\tilde{\bT}\in L^1(\Omega)^{d\times d}$, $\beps(\bv)=\beps^*(\tilde{\bT})$ in $\Omega$. Moreover, the triple $(\bu,\bT,\tilde{\bg})$ is unique in the class of solutions satisfying~\eqref{TT1},~\eqref{TT2} provided that either $\Gamma_D\neq \emptyset$ or the integral mean value of~$\bu$ is fixed.

Furthermore, there exists a symmetric $\bT^s\in \mathcal{M}(\Omega\cup \overline{\Gamma_N} )^{d\times d}$ which fulfills
$$
\langle \bT^s ,\beps (\bw)\rangle = -\langle \tilde{\bg},\bw\rangle \quad \text{for all } \bw \in \mathcal{C}^1_{\Gamma_D}(\overline{\Omega})^d
$$
and consequently $\diver \bT^s=0$ in $\Omega$ in the sense of distributions. 

Under the assumption~\eqref{BiFu}, we have in fact $\spt |\bT^s|\subset \overline{\Gamma_N}$, and $\bT\in L^{d(2-q)/(d-2)}_{loc}(\Omega)^{d\times d}$ for $d\ge 3$ and $\bT\in L^p_{loc}(\Omega)^{2\times 2}$, $p \in [1,\infty)$, for $d=2$, while under the assumptions~\eqref{A4} and~\eqref{A5}, we have
$$
\spt |\bT^s|\subset \{x\in \Omega \cup \overline{\Gamma_N} \colon M(|\bT|)(x)=\infty\},
$$
where $M$ denotes the maximal function, and 
\begin{equation}
\label{eqn_bTs_characterization}
\frac{\bT^s}{|\bT^s|} = \lim_{\varepsilon \to 0_+} \frac{\beps(\bu^{\varepsilon})}{|\beps(\bu^{\varepsilon})|} = \lim_{\varepsilon \to 0_+}\left(\frac{\bT}{|\bT|}\right)^{\varepsilon}, \qquad (\varepsilon>0)
\end{equation}
where the limit is understood in $L^1(\Omega,|\bT^s|)^d$ and the superscript $\varepsilon$ denotes the standard mollification.

Finally, if~$\cA$ defined in~\eqref{DefcA} is symmetric, then~$\bT$ and $\bT^s$ are the regular and singular parts of the minimizer to~\eqref{minimiz5} with $F(\cdot)$ defined through $F(\bT)\coloneqq \int_0^1 \beps^*(t\bT) \cdot \bT \, {\rm d}t$.
\end{Theorem}

We conclude this section by noting that although this special structure plays a crucial role in the proof, it can be relaxed to the following more general assumption. We can assume that there exist a $\mathcal{C}^1$-function $g \colon \mathbb{R}^{d\times N} \to \mathbb{R}_+$, a function $\mathcal{B}\in W^{1,\infty}(\Omega)^{d\times N \times d\times N}$ and constants $C_2 \geq C_1>0$ and $C_0 \geq 0$ such that
\begin{equation}
C_1|\bT|-C_0\le g(|\bT|) \le C_2(1+|\bT|) \qquad \textrm{for all } \bT \in \mathbb{R}^{d\times N} \label{A4MB}
\end{equation}
and, for all $\bT \in \mathbb{R}^{d\times d}_{sym}$ and almost all $x\in \Omega$, one has
\begin{equation}
\frac{|g(|\bT|)\mathcal{B}(x)\, \beps^*(\bT)-\bT|^2}{h(|\bT|)} \le C_2(1+|\bT|^3).\label{A5MB}
\end{equation}
Then all of the results stated above remain valid, with no essential changes to the proofs; for the sake of brevity we omit these proofs and will confine ourselves to some comments in Section~\ref{subsection_renormalized}.

\section{Uniqueness}\label{S2}
Here we prove the uniqueness of the triple $(\bu,\bT,\tilde{\bg})\in W^{1,1}(\Omega)^d\times L^1(\Omega)^{d\times d}\times (\mathcal{C}^1_0(\Gamma_N)^d)^*$ satisfying the properties asserted in Theorem~\ref{T1} (assuming its existence). Suppose, to this end, that $(\bu_1,\bT_1,\tilde{g}_1)$ and $(\bu_2,\bT_2,\tilde{g_2})$ are two such triples. Using~\eqref{TT2}, we obtain the following inequalities:
\begin{align}\label{TT21}
\int_{\Omega} \bT_1 \cdot (\beps(\bu_1)-\beps(\bv))\ddd x &\le \int_{\Omega} \bef \cdot (\bu_1-\bv) \ddd x + \int_{\Gamma_N}\bg \cdot (\bu_1-\bv)\ddd S,\\
\label{TT22}
\int_{\Omega} \bT_2 \cdot (\beps(\bu_2)-\beps(\bv))\ddd x &\le \int_{\Omega} \bef \cdot (\bu_2-\bv) \ddd x + \int_{\Gamma_N}\bg \cdot (\bu_2-\bv)\ddd S,
\end{align}
valid for all admissible $\bv \in W^{1,1}(\Omega)^d$ with $\bv = \bu_0$ on $\Gamma_D$ and the representation $\beps(\bv)=\beps^*(\tilde{\bT})$ for some $\tilde{\bT}\in L^1(\Omega)^{d\times d}$. Using the property, we see that we can set $\bv=\bu_1$ in~\eqref{TT22} and $\bv=\bu_2$ in~\eqref{TT21} respectively, which leads, after summing these two inequalities, to
\begin{equation*}
\int_{\Omega} (\bT_1 -\bT_2) \cdot (\beps^*(\bT_1)-\beps{^*}(\bT_2))\ddd x=\int_{\Omega} (\bT_1 -\bT_2) \cdot (\beps(\bu_1)-\beps(\bu_2))\ddd x \le 0.
\end{equation*}
Consequently, using Lemma~\ref{basici} (applied to~$\beps^*$ instead of~$\bD$), we have that $\bT_1=\bT_2$ in $\Omega$ and then necessarily also $\beps(\bu_1)=\beps(\bu_2)$. Since, by hypothesis, either $\Gamma_D$ is nonempty or $\bu_i$ has zero integral mean value, Korn's inequality leads also to $\bu_1=\bu_2$ in $\Omega$. Hence, to complete the proof it remains to discuss the behavior of $\tilde{\bg}_i$ on $\Gamma_N$. Since however we already know that $\bT_1=\bT_2$, it directly follows from~\eqref{TT1} that
\begin{equation*}
\langle \tilde{\bg}_1-\tilde{\bg}_2,\bw \rangle|_{\Gamma_N} = 0\qquad \textrm{ for all } \bw \in \mathcal{C}^1_0(\Gamma_N)^d,
\end{equation*}
and therefore $\tilde{\bg}_1=\tilde{\bg}_2$ in $(\mathcal{C}^1_0(\Gamma_N)^d)^*$.

\section{Existence}\label{S3}
Following~\cite{BMRW2013},~\cite{BMRS2014}, we introduce the following sequence of approximating problems: Find $(\bu^n,\bT^n)\in W^{1,n+1}_{\Gamma_D}(\Omega)^d\times L^{\frac{n+1}{n}}(\Omega)^{d\times d}$ with $\bT^n \in \mathbb{R}^{d\times d}_{sym}$ almost everywhere in $\Omega$ and such that
\begin{alignat}{2}\label{elasticm}
\int_{\Omega} \bT^n\cdot \beps(\bw)\ddd x&= \int_{\Omega} \bef \cdot \bw\ddd x + \int_{\Gamma_N}\bg \cdot \bw \ddd S &&\qquad\textrm{ for all } \bw \in W^{1,n+1}_{\Gamma_D}(\Omega)^d, \\
\beps(\bu^n)&=\beps^*(\bT^n) + \frac{\bT^n}{n(1+|\bT^n|^2)^{\frac{n-1}{2n}}} &&\qquad\textrm{ in } \Omega,\label{elasticm2}\\
\bu^n&=\bu_0 &&\qquad\textrm{ on } \Gamma_D.\label{elasticm3}
\end{alignat}
First, in order to ensure the meaningfulness of the expression on the right-hand side of~\eqref{elasticm}, we shall assume in what follows that $n\ge d$, and therefore $W^{1,n+1}(\Omega) \hookrightarrow \mathcal{C}(\overline{\Omega})$. It then follows from our assumptions on $\bef$ and $\bg$ that the equation~\eqref{elasticm} is meaningful. In order to show the existence of a solution to~\eqref{elasticm}--\eqref{elasticm3}, we note (using also the fact that $\beps^*(\bu)$ is monotone, see Lemma~\ref{basici}, applied to~$\beps^*$ instead of~$\bD$) that ~\eqref{elasticm2} can be restated, for each fixed $n$, in the following equivalent form:
$$
\bT^n = \bT^*_n(\beps(\bu^n)),
$$
where $\bT^*_n \colon \mathbb{R}^{d\times d}_{sym}\to \mathbb{R}^{d\times d}_{sym}$ is a continuous mapping such that, for all $\bE$, $\bE_1$, $\bE_2 \in \mathbb{R}^{d\times d}_{sym}$,
$$
\begin{aligned}
&\bT_n^*(\bE) \cdot \bE \ge -C(n)+ \tilde{C}(n)|\bE|^{n+1}, \quad |\bT_n^*(\bE)| \le C(n)(1+|\bE|^n),\\
&(\bT_n^*(\bE_1)-\bT_n^*(\bE_2))\cdot (\bE_1-\bE_2)\ge 0.
\end{aligned}
$$
Hence, the solvability of~\eqref{elasticm}--\eqref{elasticm3} follows, for any $n\ge d$ fixed, from standard monotone operator theory. Our goal is to let $n\to \infty$ in order to establish the existence of a solution to the original problem.

\subsection{First a~priori estimates}
Here, we recall some simple a~priori estimates. By noting the assumption~\eqref{proto-G2st} {and Korn's inequality}, we see that $\bw\coloneqq\bu^n-\bu_0$ is an admissible choice in~\eqref{elasticm}, and therefore we have the identity
\begin{equation*}
\int_{\Omega} \bT^n \cdot (\beps(\bu^n)-\beps(\bu_0))\ddd x = \int_{\Omega} \bef\cdot (\bu^n-\bu_0)\ddd x + \int_{\Gamma_N} \bg\cdot (\bu^n-\bu_0)\ddd S.
\end{equation*}
Next, using~\eqref{elasticm2} and the assumptions~\eqref{D1},~\eqref{D3} we arrive at the following inequality:
\begin{align}\label{kstart2}
\int_{\Omega} \frac{|\bT^n|^{1+\frac{1}{n}}}{n} + (\beps^*(\bT^n)-\beps(\bu_0))\cdot \bT^n\ddd x \le C(\bef,\bg,\bu_0)(1+\|\bu^n\|_{\mathcal{C}(\overline{\Omega})^d}).
\end{align}
We now estimate the two terms on the left-hand side. We start with the second one and note that, thanks to~\eqref{G2st_analytic}
(which is equivalent to~\eqref{proto-G2st}), there exists a constant $t_c$ such that
\begin{equation*}
(\beps^*(\bT) - \beps(\bu_0)) \cdot \bT \geq \frac{C_1 |\bT|}{2} \qquad \text{for all }|\bT|\ge t_c \text{ and a.e.~in } \Omega.
\end{equation*}
Consequently, thanks to~\eqref{AE1}, we have
\begin{equation*}
\int_\Omega (\beps^*(\bT^n)-\beps(\bu_0))\cdot \bT^n\ddd x \geq \frac{C_1}{2} \|\bT^n\|_1 - |\Omega| (C_1 + C_2 + \|\beps(\bu_0)\|_\infty)\, t_c .
\end{equation*}
Next, using~\eqref{AE1} and~\eqref{elasticm2}, we deduce the following estimate:
\begin{align}\label{kstart2.5}
|\beps(\bu^n)|\le C_2+ \frac{|\bT^n|^{\frac{1}{n}}}{n},
\end{align}
which, by Sobolev embedding, Korn's inequality and the inequality $\frac{d+1}{n} \leq 1 + \frac{1}{n}$, leads to
\begin{equation}
\label{stap_0}
\|\bu^n\|_{\mathcal{C}(\overline{\Omega})^d}^{d+1}\le C \|\beps(\bu^n)\|_{d+1}^{d+1}\le C\bigg(1+ \frac{1}{n^d}\int_{\Omega}\frac{|\bT^n|^{1+\frac{1}{n}}}{n}\ddd x\bigg).
\end{equation}
Hence, inserting the last two estimates into~\eqref{kstart2} and using Young's inequality to absorb the term on the right-hand side, we deduce, by taking $n$ sufficiently large and thereby $1/n^d$ sufficiently small, the inequality
\begin{equation}
\|\bT^n\|_1 + \frac{\|\bT^n\|_{1+\frac{1}{n}}^{1+\frac{1}{n}}}{n} \le C, \label{stap}
\end{equation}
where the constant $C$ depends only on $\Omega$, $d$, $\bf$, $\bg$, $\bu_0$, $C_1$, $C_2$ and $t_c$.
Thus, returning to~\eqref{kstart2.5}, we immediately have that
\begin{equation}\label{st2ap}
\|\beps(\bu^n)\|_{n+1} \le C,
\end{equation}
with the same dependencies of the constant~$C$.

\subsection{Limit $n\to \infty$}
It follows from~\eqref{stap},~\eqref{st2ap} and compact embedding that there exists a triple $(\bu,\bT,\overline{\bT})$ and a (sub)sequence that we do not relabel such that
\begin{alignat}{2}\label{C1}
\bT^n &\rightharpoonup^* \overline{\bT} && \qquad\textrm{ weakly-$*$ in } \mathcal{M}(\overline{\Omega})^{d\times d},\\
\bT^n & \rightharpoonup \bT && \qquad\textrm{ biting in } L^1(\Omega)^{d\times d},\label{C1.1}\\
\bu^n- {\bu_0} &\rightharpoonup \bu- {\bu_0} && \qquad\textrm{ weakly in } W^{1,d+1}_{\Gamma_D}(\Omega)^d,\label{C2}\\
\bu^n &\to \bu && \qquad\textrm{ strongly in } \mathcal{C}(\overline{\Omega})^d,\label{C2.5}\\
\frac{\bT^n}{n|\bT^n|^{1-\frac{1}{n}}} &\to \b0 &&\qquad\textrm{ strongly in } L^1(\Omega)^{d\times d}.\label{C3}
\end{alignat}
Let us recall at this stage the definition of convergence of $\bT^n$ to~$\bT$ in $L^1(\Omega)^{d\times d}$ in the weak biting sense, as the existence of an increasing sequence $(\Omega_k)_{k \in \mathbb{N}}$ of measurable subsets of $\Omega$ with $|\Omega\setminus \Omega_k|\to 0$ as $k\to \infty$ such that
\begin{equation*}
 \bT^n \rightharpoonup \bT \qquad \textrm{weakly in } L^1(\Omega_k)^{d \times d} \textrm{ for each } k \in \mathbb{N}.
\end{equation*}
With this definition and the uniform bound~\eqref{stap} at hand, the convergence~\eqref{C1.1} is in fact a direct consequence of Chacon's biting lemma, see~\cite{BaMur89}.

With the convergence~\eqref{C1}, we can now let $n\to \infty$ in~\eqref{elasticm} to deduce that
\begin{equation}
\langle \overline{\bT}, \beps (\bw)\rangle = \int_{\Omega} \bef \cdot \bw\ddd x +\langle \bg,\bw\rangle|_{\Gamma_N} \qquad \textrm{ for all } \bw \in \mathcal{C}^1_{\Gamma_D}(\overline{\Omega})^d, \label{Tbar_eq}
\end{equation}
which is~\eqref{TT1} with~$\bT$ replaced by $\overline{\bT}$ and with $\tilde{\bg} \equiv \b0$. Moreover, it directly follows from~\eqref{st2ap} and~\eqref{C2.5} that $\beps(\bu)\in L^{\infty}(\Omega)^{d\times d}$. Hence, to complete the proof, it remains to show that
\begin{equation}
\diver \overline{\bT}|_{\Omega} = \diver \bT \textrm{ in the sense of distributions},\label{RSh}
\end{equation}
that the passage from $\overline{\bT}$ to~$\bT$ in~\eqref{Tbar_eq} requires a correction via a measure $\tilde{\bg} \in (\mathcal{C}^1_0(\Gamma_N)^d)^*$ on $\Gamma_N$, and also that
\begin{equation}
\beps(\bu)=\beps^*(\bT) \textrm{ in } \Omega. \label{RShq}
\end{equation}
In addition to~\eqref{RSh}, we must be able to identify the behavior of $\overline{\bT}$ near the boundary $\Gamma_N$, where the boundary integral appears.
Therefore, in what follows we first focus on proving the pointwise convergence of $\bT^n$, i.e., that
\begin{align}\label{C1point}
\bT^n &\to \bT && \textrm{ almost everywhere in } \Omega,
\end{align}
from which, combined with~\eqref{elasticm2}, the equality~\eqref{RShq} as well as the convergence
\begin{align}\label{C1point*}
\beps(\bu^n) &\to \beps(\bu) && \textrm{ almost everywhere in } \Omega,
\end{align}
directly follow. Unfortunately, in general we will not be able to show that
\begin{align*}
\bT^n &\rightharpoonup \bT && \textrm{ weakly in } L^1_{loc}(\Omega)^{d\times d},
\end{align*}
so we shall skip this step and directly prove~\eqref{RSh}. In fact, once the pointwise convergence~\eqref{C1point} has been established, it is not difficult to show~\eqref{TT2}. Indeed, setting $\bw \coloneqq \bu^n - \bv$ in~\eqref{elasticm}, where $\bv$ is an admissible test function in~\eqref{TT2}, we obtain the identity
\begin{equation}\label{subas}
\int_{\Omega} \bT^n \cdot (\beps(\bu^n)-\beps(\bv))\ddd x = \int_{\Omega}\bef \cdot (\bu^n-\bv)\ddd x + \int_{\Omega} \bg\cdot (\bu^n-\bv)\ddd S.
\end{equation}
Hence, using~\eqref{C2.5}, we can easily let $n\to \infty$ in both terms on the right-hand side to obtain the
right-hand side of~\eqref{TT2}. In order to identify the limit also in the term on the left-hand side, we consider a $\tilde{\bT}\in L^1(\Omega)^{d \times d}$ such that $\beps(\bv)=\beps^*(\tilde{\bT})$ (the existence of such a $\tilde{\bT}$ is the assumption on admissible test functions). Then, using~\eqref{elasticm2}, we can rewrite the first term as
\begin{equation}\label{follows}
\begin{split}
\int_{\Omega} \bT^n \cdot (\beps(\bu^n)-\beps(\bv))\ddd x&=\int_{\Omega} \bT^n \cdot \bigg(\beps^*(\bT^n)+\frac{\bT^n}{n(1+|\bT|^2)^{\frac{n-1}{2n}}}-\beps(\bv)\bigg)\ddd x\\
&\ge \int_{\Omega} \bT^n \cdot \left(\beps^*(\bT^n)-\beps(\bv)\right)\ddd x\\
&= \int_{\Omega} (\bT^n-\tilde{\bT}) \cdot \left(\beps^*(\bT^n)-\beps(\bv)\right)\ddd x
+\int_{\Omega} \tilde{\bT} \cdot \left(\beps^*(\bT^n)-\beps(\bv)\right)\ddd x.
\end{split}
\end{equation}
Finally, using the monotonicity of~$\beps^*$ (recall that $\beps(\bv)=\beps^*(\tilde{\bT})$), see Lemma~\ref{basici}, the first term on the right-hand side is nonnegative and we can therefore use Fatou's lemma and~\eqref{C1point} to identify the limes inferior. The limit in the second term is the consequence of the assumption~\eqref{AE1} and Lebesgue's dominated convergence theorem. It then follows from~\eqref{follows} that
\begin{equation*}
\begin{split}
\liminf_{n\to \infty}\int_{\Omega} \bT^n \cdot (\beps(\bu^n)-\beps(\bv))\ddd x
&\ge \int_{\Omega} (\bT-\tilde{\bT}) \cdot \left(\beps^*(\bT)-\beps(\bv)\right)\ddd x
+\int_{\Omega} \tilde{\bT} \cdot \left(\beps^*(\bT)-\beps(\bv)\right)\ddd x\\
&=\int_{\Omega} \bT\cdot \left(\beps^*(\bT)-\beps(\bv)\right)\ddd x\\
&\!\!\!\!\overset{\eqref{RShq}}=\!\!\int_{\Omega} \bT \cdot \left(\beps(\bu)-\beps(\bv)\right)\ddd x.
\end{split}
\end{equation*}
Thus we can substitute the above inequality into~\eqref{subas} to deduce~\eqref{TT2}.

The essential ingredients of the proofs of the statements~\eqref{RSh} and~\eqref{C1point} are interior weighted estimates for $\nabla \bT^n$. These are established in the next subsection.
In order to show~\eqref{RSh}, a novel approach, called the (interior) \emph{renormalized} regularity of $\bT^n$, is used; see Subsection~\ref{subsection_renormalized}.

\subsection{Uniform interior higher differentiability} \label{4.5}
In this subsection, we establish uniform bounds on the solution to~\eqref{elasticm}--\eqref{elasticm3}; for the sake of simplicity, we omit writing the superscript $n$, so we replace $(\bu^n,\bT^n)$ by $(\bu,\bT)$, but we shall nevertheless trace the dependence of the bounds on $n$. First, recalling standard interior higher differentiability theory (see for example~\cite{FrNe81}), one can prove the existence of a strong solution that satisfies, pointwise,
$$
-\diver \bT = \bef \qquad \textrm{ in } \Omega.
$$
Hence, for an arbitrary $\eta \in \mathcal{C}^1_0(\Omega)$, we multiply this equation by $-\eta^2 \Delta \bu$ (which is well-defined via the regularity of~$\bT$ and identity~\eqref{elasticm2}), and after integration over $\Omega$ we obtain the following identity:
\begin{equation}\label{stkilo1}
\int_{\Omega} \diver \bT \cdot \triangle \bu\, \eta^2\ddd x = -\int_{\Omega} \bef \cdot \triangle \bu\, \eta^2\ddd x.
\end{equation}
This identity is the starting point for the analysis that follows. Henceforth, we shall use Einstein's summation convention, and any formal integration by parts that may occur in the course of the argument below will be
understood to be justified by the density of smooth functions in the relevant function space. In addition, in order to simplify the presentation, we denote $\partial_j\coloneqq\frac{\partial}{\partial x_j}$. First, we focus on the term on the left-hand side. By defining
\begin{equation}\label{GOI}
I\coloneqq \int_{\Omega}\eta^2 \partial_k \bT \cdot \partial_k \beps(\bu)\ddd x,
\end{equation}
our goal is to express the term $I$ as the left-hand side of~\eqref{stkilo1} and a certain pollution term. To do so, we integrate by parts in the term on the left-hand side of~\eqref{stkilo1} to deduce that
\begin{equation}\label{hlad}
\begin{aligned}
\tilde{I}&\coloneqq\int_{\Omega} \diver \bT \cdot \triangle \bu\, \eta^2\ddd x=\int_{\Omega} \partial_j \bT_{ij}\, \partial_{kk}\bu^i\, \eta^2\ddd x\\
&=-\int_{\Omega} \bT_{ij} \,\partial_{kkj} \bu^i\, \eta^2\ddd x-2\int_{\Omega} \eta\, \bT_{ij}\, \partial_{kk}\bu^i \, \partial_j\eta\ddd x\\
&=-\int_{\Omega} \bT_{ij}\, \partial_{kk} \beps_{ij}(\bu)\, \eta^2\ddd x-2\int_{\Omega} \eta\, \bT_{ij}\, \partial_{kk}\bu^i \,\partial_j\eta\ddd x\\
&=I+2\int_{\Omega}\eta\, \bT_{ij}\, \partial_{k} \beps_{ij}(\bu) \,\partial_k\eta\ddd x-2\int_{\Omega} \eta\, \bT_{ij} \,\partial_{kk}\bu^i \,\partial_j\eta\ddd x.
\end{aligned}
\end{equation}
Hence, inserting~\eqref{hlad} into~\eqref{stkilo1}, we get
\begin{equation}\label{stkilo2}
I=-2\int_{\Omega}\eta\, \bT_{ij}\, \partial_{k} \beps_{ij}(\bu)\, \partial_k\eta\ddd x
+2\int_{\Omega} \eta\, \bT_{ij}\, \partial_{kk}\bu^i \, \partial_j\eta\ddd x -\int_{\Omega} \bef \cdot \triangle \bu\, \eta^2\ddd x.
\end{equation}
Next, using the fact that $\partial_{kk}\bu_i = 2\partial_{k}\beps_{ik}(\bu)-\partial_i \beps_{kk}(\bu)$, we see that we can rewrite~\eqref{stkilo2} in the following, more compact, form:
\begin{equation}\label{stkilo3}
I=\int_{\Omega}\eta\, \partial_{k} \beps_{ij}(\bu)\, \bB_{ij}^k \ddd x,
\end{equation}
where
$$
\bB_{ij}^k\coloneqq-2\bT_{ij}\,\partial_k\eta+4\bT_{im}\,\partial_m \eta\, \delta_{jk}-2\bT_{km}\,\partial_m\eta\, \delta_{ij}-2\eta\,\bef^i\, \delta_{jk} + \bef^k\, \eta\, \delta_{ij}
$$
and $\delta_{ij}$ denotes the Kronecker delta. Next, we evaluate the terms on both sides of~\eqref{stkilo3} with the help of the definition of~$\cA$; see~\eqref{DefcA}. To this end we also introduce $\cB^n$ as
$$
\cB^n(\bT)\coloneqq\frac{\partial }{\partial \bT} \bigg(\frac{\bT}{n(1+|\bT|^2)^{\frac{n-1}{2n}}}\bigg).
$$
A straightforward calculation shows that
$$
\cB^n_{ijab}(\bT) = \frac{1}{n(1+|\bT|^2)^{\frac{n-1}{2n}}} \left( \delta_{ia} \, \delta_{jb} - \frac{n-1}{n} \frac{ \bT_{ij} \bT_{ab}}{1+|\bT|^2} \right),
$$
which implies that $\cB^n$ is a symmetric and positive definite operator for each $\bT\in \mathbb{R}^{d\times d}_{sym}$. Moreover, it satisfies
\begin{equation}
|\cB^n(\bT)|\le \frac{C}{n(1+|\bT|^2)^{\frac{n-1}{2n}}}.\label{cBest}
\end{equation}
Finally, using~\eqref{elasticm2} we can express the first partial derivatives of $\beps(\bu)$ as
\begin{equation*}
\partial_k \beps_{ij}(\bu)=\partial_k \beps^*_{ij}(\bT) + \partial_k \frac{\bT_{ij}}{n(1+|\bT|^2)^{\frac{n-1}{2n}}}=(\cA_{ijab}(\bT)+\cB^n_{ijab}(\bT))\,\partial_k \bT_{ab}.
\end{equation*}
Hence, returning to~\eqref{stkilo3}, and using the definition of $I$, see~\eqref{GOI}, we deduce that (note that only the symmetric part of~$\cA$ appears on the left-hand side of~\eqref{GOI})
\begin{equation*}
\begin{split}
\int_{\Omega}&(\eta\, \partial_k \bT, \eta\,\partial_k \bT)_{\cA^s(\bT)}+(\eta\, \partial_k \bT, \eta\,\partial_k \bT)_{\cB^n(\bT)}\ddd x = I \\
&= \int_{\Omega}(\cA_{ijab}(\bT)+\cB^n_{ijab}(\bT))\,\eta\,\partial_k \bT_{ab}\bB_{ij}^k\ddd x\\
&= \int_{\Omega}(\eta\, \partial_k \bT, \bB^k)_{\cA^s(\bT)}+(\eta\, \partial_k \bT, \bB^k)_{\cB^n(\bT)}\ddd x+\int_{\Omega}(\cA_{ijab}(\bT)-\cA^s_{ijab}(\bT))\,\eta\,\partial_k \bT_{ab}\,\bB_{ij}^k\ddd x.
\end{split}
\end{equation*}
Next, using the fact that $\cA^s$ is symmetric and positive definite for each~$\bT$ (see~\eqref{A1}), we see that $(\cdot,\cdot)_{\cA^s(\bT)}$ is a scalar product on $\mathbb{R}^{d\times d}_{sym}$ and the same holds true also for $(\cdot,\cdot)_{\cB^n(\bT)}$. Therefore, using the Cauchy--Schwarz inequality and Young's inequality to absorb the first integral, we see that
\begin{equation*}
\begin{split}
&\int_{\Omega}(\eta\, \partial_k \bT, \eta\,\partial_k \bT)_{\cA^s(\bT)}\ddd x \\
&\le \int_{\Omega}(\bB^k, \bB^k)_{\cA^s(\bT)}+(\bB^k, \bB^k)_{\cB^n(\bT)}\ddd x
+ 2 \int_{\Omega}|\cA(\bT)-\cA^s(\bT)| \, \eta \,|\nabla \bT|\,|\bB|\ddd x.
\end{split}
\end{equation*}
Thus, using~\eqref{A1},~\eqref{A3},~\eqref{cBest} and Young's inequality, we deduce that
\begin{equation*}
\begin{split}
\int_{\Omega}h(|\bT|)\,|\nabla \bT|^2\, \eta^2 \ddd x &\le \int_{\Omega}(\eta\, \partial_k \bT, \eta\,\partial_k \bT)_{\cA^s(\bT)}\ddd x\\
&\le C\int_{\Omega} \frac{|\bB|^2}{1+|\bT|}+ \frac{|\bB|^2}{n(1+|\bT|^2)^{\frac{n-1}{2n}}}\ddd x
+C\int_{\Omega}\frac{|\cA(\bT)-\cA^s(\bT)|^2\,|\bB|^2}{h(|\bT|)}\ddd x\\
&\le C\int_{\Omega} \frac{|\bB|^2}{1+|\bT|}+ \frac{|\bB|^2}{n(1+|\bT|^2)^{\frac{n-1}{2n}}}\ddd x.
\end{split}
\end{equation*}
Finally, using the definition of $\bB$, the above inequality reduces to
\begin{equation}
\begin{split}\label{uztoje}
&\int_{\Omega}h(|\bT|)\,|\nabla \bT|^2\, \eta^2+(\eta\, \partial_k \bT, \eta\,\partial_k \bT)_{\cA^s(\bT)}\ddd x\\
&\le C(\eta)\int_{\Omega} \frac{|\bT|^2+|\bef|^2}{1+|\bT|}+ \frac{|\bT|^2+|\bef|^2}{n(1+|\bT|^2)^{\frac{n-1}{2n}}}\ddd x\\
&\le C(\eta)\int_{\Omega}|\bef|^2 + |\bT| + \frac{1}{n}|\bT|^{1+\frac{1}{n}} \ddd x\le C(\eta),
\end{split}
\end{equation}
where the last inequality follows from the a~priori estimate~\eqref{stap} and the assumption~\eqref{D1}.

\subsection{Pointwise convergence result~\eqref{C1point}}

In this subsection, we use the bounds derived in the previous subsection. Hence, returning to our original notation $(\bu^n,\bT^n)$ for a solution to~\eqref{elasticm}, we are now interested in proving the pointwise convergence result~\eqref{C1point}. First, we introduce an auxiliary function $\tilde{h}$ as
\begin{equation}
\tilde{h}(t)\coloneqq\int_{t}^{\infty} \frac{h(s)}{(1+s)^2}\ddd s \qquad \textrm{ for all } t \in \mathbb{R}_+.
\end{equation}
We note that $\tilde{h} \colon \mathbb{R}_+\to \mathbb{R}_+$ is strictly monotonic decreasing and since $h$ is also nonincreasing, we have
$$
\tilde{h}(s)\le h(s).
$$
Then, by defining
\begin{align*}
\bB^n&\coloneqq \tilde{h}(|\bT^n|)\bT^n,\\
a^n&\coloneqq\tilde{h}(|\bT^n|),
\end{align*}
it follows from~\eqref{uztoje} and from the fact that $h(s)\le \frac{C}{1+s}$ (see \eqref{A1}) that, for all $n$ and all $\Omega_0 \subset \subset \Omega$,
$$
\|\bB^n\|_{L^\infty(\Omega)} + \|a^n\|_{L^\infty(\Omega)} + \int_{\Omega_0}|\nabla \bB^n|^2 + |\nabla a^n|^2\ddd x \le C + C\int_{\Omega_0}h(|\bT^n|)|\nabla \bT^n|^2 \ddd x\le C.
$$
Therefore, thanks to compact Sobolev embedding, there exist subsequences (not indicated) such that
\begin{alignat*}{2}
\bB^n &\rightharpoonup \bB &&\qquad\textrm{weakly in } L^1(\Omega)^{d\times d},\\
a^n &\rightharpoonup a &&\qquad\textrm{weakly in } L^1(\Omega),\\
\bB^n &\to \bB &&\qquad\textrm{strongly in } L^1(\Omega_0)^{d\times d},\\
a^n &\to a &&\qquad\textrm{strongly in } L^1(\Omega_0),\\
\bB^n & \to \bB &&\qquad\textrm{a.e. in } \Omega,\\
a^n &\to a &&\qquad\textrm{a.e. in } \Omega.
\end{alignat*}
Moreover, since $(\bT^n)_{n \in \mathbb{N}}$ is a bounded sequence in $L^1(\Omega)^{d \times d}$ (see~\eqref{stap}), we deduce that
$$
\int_{\Omega}\tilde{h}^{-1}(a^n)\ddd x\le C,
$$
where $\tilde{h}^{-1}$ denotes the nonnegative inverse function to $\tilde{h}$, which exists on $\tilde{h}(\mathbb{R}_+)$ and is decreasing and continuous. Consequently, using Fatou's lemma and the pointwise convergence of $a^n$, it follows that
$$
\tilde{h}^{-1}(a)\in L^1(\Omega) \implies a>0 ~\textrm{ a.e. in } \Omega.
$$
Here we have used that $\tilde{h}(\infty)=0$, and therefore $\tilde{h}^{-1}(0)=\infty$, which then implies that for $\tilde{h}^{-1}(a)$
to belong to $L^1(\Omega)$ it is necessary that $a>0$ a.e. on $\Omega$.

Finally, since
$$
\bT^n = \frac{\bB^n}{a^n},
$$
the above pointwise convergence result implies that
$$
\bT^n \to \tilde{\bT} ~\textrm{ a.e. in } \Omega,
$$
where
$$
\tilde{\bT}\coloneqq\frac{\bB}{a},
$$
which is a measurable function that is finite a.e. in $\Omega$. On the other hand, from the biting convergence~\eqref{C1.1} we have
weak convergence to~$\bT$ in $L^1(\Omega_k)^{d \times d}$, where $(\Omega_k)_{k \in \mathbb{N}}$ is an increasing sequence of subsets of
$\Omega$ such that $|\Omega\setminus \Omega_k|\to 0$ as $k\to \infty$. Because of the uniqueness of the limit we then have that
$$
\bT^n \to \bT ~\textrm{ a.e. in } \Omega_k \mbox{ for each $k \geq 1$}.
$$
Thanks to the properties of the sets $\Omega_k$ it then follows that
$$
\bT^n \to \bT ~\textrm{ a.e. in } \Omega.
$$
Moreover, using Fatou's lemma and~\eqref{stap}, we deduce that
\begin{equation}\label{Fat1}
\int_{\Omega}|\bT| \ddd x \le \liminf_{n\to \infty}\int_{\Omega}|\bT^n| \ddd x \le C,
\end{equation}
which completes the proof of~\eqref{C1point}.

\subsection{Uniform interior renormalized regularity of $\bT^n$}
\label{subsection_renormalized}
The next step is to strengthen~\eqref{C1point} and to obtain~\eqref{RSh}, i.e., we want to show that
\begin{equation}\label{ren5.6}
\int_{\Omega} \bT \cdot \beps (\bw) - \bef \cdot \bw \ddd x = 0 \qquad \textrm{ for all } \bw \in \mathcal{C}^1_0(\Omega)^{d}.
\end{equation}
To this end, we distinguish the cases where the lower bound~\eqref{BiFu} holds, or where the asymptotic Uhlenbeck structure~\eqref{A5} with~\eqref{A4} is available.

Let us first argue under the assumption of~\eqref{BiFu} involving a restrictive condition on $q$, namely $q < 1 + 2/d$. If $d \geq 3$, then, taking also into account~\eqref{uztoje}, we find
\begin{align*}
 \int_\Omega \big| \nabla \big( (1+ |\bT^n|)^{\frac{2-q}{2}} \big) \big|^2 \eta^2 \ddd x
 & \leq C \int_\Omega (1+ |\bT^n|)^{-q} |\nabla \bT^n|^2 \eta^2 \ddd x \\
 & \leq C C_1^{-1} \int_\Omega h(|\bT^n|) |\nabla \bT^n|^2 \eta^2 \ddd x \leq C_1^{-1} C(\eta),
\end{align*}
which first yields uniform boundedness of the sequence $(1+ |\bT^n|)^{\frac{2-q}{2}}$ in $W^{1,2}_{loc}(\Omega)$ and then, by the Rellich--Kondrashov theorem on compact embedding and after passage to a subsequence, strong convergence $\bT^n \to \bT$ in $L^1_{loc}(\Omega)^{d \times d}$ (note that, by hypothesis, $2/(2-q) < 2d/(d-2)$). In addition, using also the Gagliardo--Nirenberg continuous embedding theorem and the reflexivity of the Lebesgue spaces $L^r(\Omega)$ for $r \in (1,\infty)$, we deduce that $\bT\in L^{d(2-q)/(d-2)}_{loc}(\Omega)^{d\times d}$. If instead $q=d=2$, then we can proceed similarly and obtain that $\log (1+ |\bT^n|)$ is uniformly bounded in $W^{1,2}_{loc}(\Omega)$, which via the Trudinger--Moser inequality from~\cite{TRUDINGER67} implies that the sequence $\bT^n$ is bounded in $L^p_{loc}(\Omega)^{d \times d}$ for any $p \in [1, \infty)$, and hence, we have in particular strong convergence of $\bT^n \to \bT$ in $L^p_{loc}(\Omega)^{d \times d}$ for any $p\in [1,\infty)$. With these convergence results in hand, the claim that~\eqref{ren5.6} holds then follows immediately from the choice of the approximate solutions $\bT^n$ satisfying~\eqref{elasticm}. Moreover, we see that, under the assumption \eqref{BiFu}, the higher integrability of~$\bT$ stated in Theorem~\ref{T1} (or Theorem~\ref{T1-G}) holds. In particular, as a simple consequence, we obtain for the prototypical case \eqref{LSt2gcp} the higher integrability stated in Theorem~\ref{TH00}.

Otherwise, if we work under the assumptions~\eqref{A4} and~\eqref{A5}, we have to use more subtle arguments, which are inspired by the notion of renormalized weak solution first introduced in the present context in~\cite{BMS2014}. Thanks to the estimate~\eqref{uztoje}, we see that the weak solution of~\eqref{elasticm} is in fact the strong solution and therefore, pointwise in $\Omega$, we have
\begin{equation}\label{pointel}
-\diver \bT^n=\bef.
\end{equation}
Consequently, let $\bw \in \mathcal{C}^1_0(\Omega)^d$ and $\tau \in \mathcal{C}^1_0(\mathbb{R})$ be arbitrary. Then, by multiplying\footnote{If we were to assume the more general structure~\eqref{A4MB},~\eqref{A5MB}, we would multiply by $\tau(g(\bT^n))\bw$. The method then remains the same with only minor modifications and an adjusted definition of $G_k$ in~\eqref{def_G_k}.}~\eqref{pointel} by $\bw \tau(|\bT^n|)$ and integrating over $\Omega$, using integration by parts (note that all boundary terms vanish thanks to the assumption that $\bw$ has compact support in $\Omega$) we deduce that
\begin{equation}\label{renlim}
\int_{\Omega}\bT^n \cdot \beps(\bw)\, \tau(|\bT^n|)\ddd x = \int_{\Omega} \bef \cdot \bw\, \tau(|\bT^n|) \ddd x
-\int_{\Omega}\bT^n_{ij}\, \bw^i\, \partial_j \tau(|\bT^n|)\ddd x.
\end{equation}
Note here that since $\tau$ has compact support and the estimate~\eqref{uztoje} holds, the last integral in~\eqref{renlim} is meaningful. In addition, thanks to Lebesgue's dominated convergence theorem, using the pointwise convergence~\eqref{C1point} and the fact that $\tau$ has compact support (observe that $\tau(|\bT^n|)\,\bT^n \in L^\infty(\Omega)^{d \times N}$, uniformly w.r.t. $n$), we can let $n\to \infty$ in the first two integrals in~\eqref{renlim} to obtain the identity
\begin{equation}\label{ren2}
\int_{\Omega} \bT \cdot \beps (\bw)\, \tau(|\bT|) - \bef \cdot \bw\, \tau(|\bT|)\ddd x = -\lim_{n\to \infty} \int_{\Omega}\bT^n_{ij}\bw^i\,\partial_j \tau(|\bT^n|)\ddd x.
\end{equation}
Although we could evaluate the last term (similarly as in~\cite{BMS2014}), we have refrained from doing so here, as this is not necessary. Instead, we shall pass to the limit with $\tau$. To this end, we introduce a sequence of smooth nonincreasing functions $\tau_k \colon \mathbb{R}_+\to [0,1]$, which satisfy
$$
\tau_k(s)=\left\{\begin{aligned}
&1 &&0 \leq s\le k,\\
&0 &&s\ge 2k,
\end{aligned}
\right.
$$
and $|\tau'_k|\le \frac{C}{k}$. We then use $\tau_k$ instead of $\tau$ in~\eqref{ren2} and let $k\to \infty$. Since $\tau_k\nearrow 1$ and $\bT \in L^1(\Omega)^{d\times d}$, we can let $k \to \infty$ in the terms on the left-hand side of~\eqref{ren2} (with $\tau_k$ instead of $\tau$) to deduce that
\begin{equation}\label{ren3}
\int_{\Omega} \bT \cdot \beps (\bw) - \bef \cdot \bw \ddd x= -\lim_{k\to \infty}\lim_{n\to \infty} \int_{\Omega}\bT^n_{ij}\,\bw^i\, \partial_j \tau_k(|\bT^n|) \ddd x.
\end{equation}
Our objective now is to show that the right-hand side of~\eqref{ren3} vanishes.

The following calculations rely on the assumed asymptotic Uhlenbeck structure~\eqref{A5} with~\eqref{A4} (or~\eqref{A5MB} with~\eqref{A4MB}, respectively). First, for fixed $k,n$ we rewrite the term on the right-hand side of~\eqref{ren3} as
\begin{equation}\label{bulda}
\begin{aligned}
-\int_{\Omega}\bT^n_{ij}\,\bw^i\,\partial_j \tau_k(|\bT^n|)\ddd x&=-\int_{\Omega}(\bT^n_{ij}-g(|\bT^n|)\,\beps^*_{ij}(\bT^n))\,\bw^i\,\partial_j \tau_k(|\bT^n|)\ddd x\\
&\quad - \int_{\Omega}g(|\bT^n|)\,\beps_{ij}^*(\bT^n)\,\partial_j\tau_k(|\bT^n|)\,\bw^i \ddd x.
\end{aligned}
\end{equation}
To evaluate the second term, we introduce the new function
\begin{equation}
\label{def_G_k}
G_k(s)\coloneqq \int_0^s \tau'_k(t)g(t)\ddd t,
\end{equation}
and with the aid of this definition and integration by parts, we rewrite the second integral on the right-hand side of~\eqref{bulda} as follows:
\begin{equation}\label{bulda2}
\begin{aligned}
\int_{\Omega}&g(|\bT^n|)\,\beps_{ij}^*(\bT^n)\,\partial_j\tau_k(|\bT^n|)\,\bw^i \ddd x=\int_{\Omega}g(|\bT^n|)\,\beps_{ij}^*(\bT^n)\,\tau'_k(|\bT^n|)\,\partial_j|\bT^n|\,\bw^i \ddd x\\
&=\int_{\Omega}\beps_{ij}^*(\bT^n)\,\partial_jG_k(|\bT^n|)\,\bw^i \ddd x\\
&=-\int_{\Omega}\partial_j \beps_{ij}^*(\bT^n)\,G_k(|\bT^n|)\,\bw^i +\beps_{ij}^*(\bT^n)\,G_k(|\bT^n|)\,\partial_j\bw^i \ddd x\\
&=-\int_{\Omega}\cA^s_{ijlm}(\bT^n)\,\partial_{\ell} \bT^n_{lm}\,G_k(|\bT^n|)\,\bw^i\,\delta_{j\ell}\ddd x -\int_{\Omega}\beps_{ij}^*(\bT^n)\,G_k(|\bT^n|)\,\partial_j\bw^i \ddd x\\
&\quad -\int_{\Omega}(\cA_{ijlm}(\bT^n)-\cA^s_{ijlm}(\bT^n))\,\partial_j \bT^n_{lm}\,G_k(|\bT^n|)\,\bw^i\ddd x,
\end{aligned}
\end{equation}
where the last identity follows from the definition~\eqref{DefcA} of~$\cA$, with $\cA^s$ defined by~\eqref{Asymm}. Consequently, by substituting~\eqref{bulda2} into~\eqref{bulda} we have that
\begin{equation*}
\begin{aligned}
&\left|\int_{\Omega}\bT^n_{ij}\,\bw^i\,\partial_j \tau_k(|\bT^n|)\ddd x\right|\\
&\le \int_{\Omega}\Big(|\bT^n-g(|\bT^n|)|\,|\beps^*(\bT^n)|\,|\tau'_k(|\bT^n|)|+|\cA(\bT^n)-\cA^s(\bT^n)|\,G_k(|\bT^n|)\Big)|\bw|\,|\nabla \bT^n| \ddd x\\
&\quad +\left|\int_{\Omega}(|\bw|\,\partial_{\ell} \bT^n,\bE^{\ell}(\bw,\bT^n,\tau_k))_{\cA^s(\bT^n)}\ddd x \right| +\left|\int_{\Omega}\beps_{ij}^*(\bT^n)\,G_k(|\bT^n|)\,\partial_j\bw^i \ddd x\right|\\
&=:I^{n,k}_1+I_2^{n,k}+I_3^{n,k},
\end{aligned}
\end{equation*}
where
$$
\bE_{ij}^{\ell}(\bw,\bT^n,\tau_k)\coloneqq\frac{G_k(|\bT^n|)\bw^i \delta_{j \ell}}{|\bw|}.
$$
To proceed, we begin by noting that~\eqref{uztoje} still holds with $\eta \in \mathcal{C}^1_0(\Omega)$ replaced by $\eta \in W^{1,\infty}_0(\Omega)$, and therefore in particular with $\eta = |\bw|$ where $\bw \in \mathcal{C}^1_0(\Omega)^d$. We thus
have that
\begin{equation}
\begin{split}\label{uztojehlp}
&\int_{\Omega}h(|\bT^n|)\,|\nabla \bT^n|^2\, |\bw|^2+(|\bw|\, \partial_k \bT^n, |\bw|\,\partial_k \bT^n)_{\cA^s(\bT^n)}\ddd x\le C(\bw).
\end{split}
\end{equation}
Hence, using H\"{o}lder's inequality, we can bound $I_1^{n,k}$ as follows:
\begin{equation*}
\begin{aligned}
I_1^{n,k}&\le C\left(\int_{\Omega}|\bw|^2\,h(|\bT^n|)\,|\nabla \bT^n|^2\ddd x\right)^{\frac12}\\
&\qquad \times \left(\int_{\Omega}\frac{|\bT^n-g(|\bT^n|)\,\beps^*(\bT^n)|^2\,|\tau'_k(|\bT^n|)|^2+|\cA(\bT^n)-\cA^s(\bT^n)|^2\,G^2_k(|\bT^n|)}{h(|\bT^n|)} \ddd x\right)^{\frac12}\\
&\le C(\bw)\left(\int_{\Omega}(1+|\bT^n|^3)\,|\tau'_k(|\bT^n|)|^2+\frac{G^2_k(|\bT^n|)}{1+|\bT^n|} \ddd x\right)^{\frac12},
\end{aligned}
\end{equation*}
where the last inequality follows from~\eqref{A3},~\eqref{A5} and~\eqref{uztojehlp}. Similarly, using the Cauchy--Schwarz inequality, H\"{o}lder's inequality, the definition of $\bE^{\ell}$,~\eqref{uztojehlp} and the assumption~\eqref{A1}, we deduce that
\begin{equation*}
\begin{aligned}
I_2^{n,k}&\le \left(\int_{\Omega}(|\bw|\partial_{\ell} \bT^n,|\bw|\partial_{\ell} \bT^n)_{\cA^s(\bT^n)}\ddd x\right)^{\frac12} \left(\int_{\Omega}(\bE^{\ell}(\bw,\bT^n,\tau_k),\bE^{\ell}(\bw,\bT^n,\tau_k))_{\cA^s(\bT^n)}\ddd x \right)^{\frac12}\\
&\le C(\bw)\left(\int_{\Omega}\frac{G^2_k(|\bT^n|)}{1+|\bT^n|} \ddd x\right)^{\frac12}.
\end{aligned}
\end{equation*}
Finally, since $G_k$ is bounded for each fixed $k$ and $\tau_k$ is compactly supported, we can use Lebesgue's dominated convergence theorem and apply the pointwise convergence result~\eqref{C1point} to deduce that
$$
\begin{aligned}
&\limsup_{n \to \infty}I_1^{n,k}+I_2^{n,k}+I_3^{n,k} \\
&\le C(\bw)\bigg(\int_{\Omega}(1+|\bT|^3)\,|\tau'_k(|\bT|)|^2+\frac{G^2_k(|\bT|)}{1+|\bT|} \ddd x\bigg)^{\frac12}
+\left|\int_{\Omega}\beps_{ij}^*(\bT)\,G_k(|\bT|)\,\partial_j\bw^i \ddd x\right|\\
&\le C(\bw)\bigg(\int_{\{|\bT|>k\}}(1+|\bT|^3)\,|\tau'_k(|\bT|)|^2+\frac{G^2_k(|\bT|)}{1+|\bT|} \ddd x\bigg)^{\frac12} +C(\bw)\left|\int_{\{|\bT|>k\}}|G_k(|\bT|)| \ddd x\right|,
\end{aligned}
$$
where the second inequality is the consequence of the properties of $\tau_k$, the definition of $G_k$ and the assumption~\eqref{AE1}. Thus, using the assumption~\eqref{A4}, we see that for all $s\ge k$ we have
$$
|G_k(s)|=\int_k^s |\tau'_k(t)|\,g(t)\ddd t \le Ck^{-1}\int_k^{2k}(1+t)\ddd t \le C(1+k)\le C(1+s);
$$
substituting this bound into the above limit and using the fact $|\tau'_k(|\bT|)|\le C(1+|\bT|)^{-1}$ we deduce that
\begin{equation*}
\begin{aligned}
0 \leq \limsup_{n \to \infty}I_1^{n,k}+I_2^{n,k}+I_3^{n,k}&\le C(\bw)\int_{\{|\bT|>k\}}1+|\bT|\ddd x.
\end{aligned}
\end{equation*}
Therefore, since $\bT \in L^1(\Omega)^{d\times d}$, see~\eqref{Fat1}, we have that
\begin{equation*}
\limsup_{k\to \infty}\,\limsup_{n \to \infty}I_1^{n,k}+I_2^{n,k}+I_3^{n,k}=0.
\end{equation*}
Hence, substituting this relation into~\eqref{ren3} yields the desired identity~\eqref{ren5.6}.

\subsection{Derivation of~\eqref{TT1}}
Using the weak-$*$ density of $\mathcal{C}^1_0(\Omega)^{d}$ functions, see Lemma~\ref{dense}, we see that~\eqref{ren5.6} yields
\begin{equation}\label{ren6}
\int_{\Omega} \bT \cdot \beps (\bw) - \bef \cdot \bw \ddd x = 0 \qquad \textrm{ for all } \bw\in W^{1,1}_0(\Omega)^d; \; \beps(\bw)\in L^{\infty}(\Omega)^{d\times d}.
\end{equation}
This then leads to the definition of the normal component of the trace of~$\bT$ on $\partial \Gamma_N$ as
\begin{equation}\label{trace}
\langle \bT \bn, \bw\rangle|_{\Gamma_N}\coloneqq \int_{\Omega}\bT \cdot \beps(\tilde{\bw}) - \bef \cdot \tilde{\bw} \ddd x,
\end{equation}
for all $\bw \in \mathcal{C}_0(\Gamma_N)^d$ for which there exists an extension $\tilde{\bw}$ from $\Gamma_N$ to $\overline\Omega$ such that $\tilde{\bw} \in W^{1,1}_{\Gamma_D}(\Omega)^d$, $\beps(\tilde{\bw})\in L^{\infty}(\Omega)^{d\times d}$ and $\tilde{\bw}=\bw$ on $\Gamma_N$.
We note here that, thanks to~\eqref{ren6}, the definition of $\bT\bn$ does not depend on the choice of the extension $\tilde{\bw}$. Furthermore, since by Kirszbraun's extension theorem any Lipschitz function on $\Gamma_N$ can can be extended onto $\mathbb{R}^d$ by
preserving its Lipschitz constant, we see that $\bT\bn \in (W^{1,\infty}_0(\Gamma_N)^d)^*$ and therefore $\bT\bn \in (\mathcal{C}^1_0(\Gamma_N)^d)^*$.

Finally, we focus on the correct identifications of all limits. Comparing~\eqref{Tbar_eq} with~\eqref{ren6}, we see that
\begin{equation}\label{ren6.6}
\langle \overline{\bT}, \beps (\bw)\rangle =\int_{\Omega} \bT \cdot \beps(\bw)\ddd x \qquad \textrm{ for all } \bw \in \mathcal{C}^1_0(\Omega)^{d}.
\end{equation}
Moreover,~\eqref{Tbar_eq} and~\eqref{ren6.6} allow us to also define the trace of the measure $\overline{\bT}$ on $\Gamma_N$ as follows: {Denoting by $\mathcal{C}^1_0(\Gamma_N)^d$ the space of all functions $\bw$ for which there exists an extension $\tilde{\bw}\in \mathcal{C}^1_{\Gamma_D}(\overline{\Omega})^d$ such that $\tilde{\bw}=\bw$ on $\Gamma_N$, we may} define the distribution
\begin{equation}\label{trace3}
\langle \overline{\bT} \bn, \bw\rangle|_{\Gamma_N} \coloneqq \langle \overline{\bT}, \beps(\tilde\bw)\rangle - \int_{\Omega}\bef \cdot \tilde\bw \ddd x,
\end{equation}
{thus $\overline{\bT}\bn|_{\Gamma_N}\in (\mathcal{C}^1_0(\Gamma_N)^d)^*$}.
We note here that this definition is meaningful and does not depend on the choice of the extension~$\tilde{\bw}$. Indeed, let $\tilde{\bw}_1,\tilde{\bw}_2$ be two extensions; then, necessarily, $\tilde{\bw}_1-\tilde{\bw}_2\in \mathcal{C}^1_0(\Omega)^d$ and
$$
\begin{aligned}
&\langle \overline{\bT}, \beps(\tilde{\bw}_1)\rangle - \int_{\Omega}\bef \cdot \tilde{\bw}_1 \ddd x-\langle \overline{\bT}, \beps(\tilde{\bw}_2)\rangle + \int_{\Omega}\bef \cdot \tilde{\bw}_2 \ddd x\\
&=\langle \overline{\bT}, \beps(\tilde{\bw}_1-\tilde{\bw}_2)\rangle - \int_{\Omega}\bef \cdot (\tilde{\bw}_1-\tilde{\bw}_2) \ddd x\overset{\eqref{Tbar_eq}}=0.
\end{aligned}
$$
Finally, let us define $\tilde{\bg}\in (\mathcal{C}^1_0(\Gamma_N)^d)^*$ by
\begin{equation}\label{finmeasure}
\langle \tilde{\bg},\bw\rangle|_{\Gamma_N}\coloneqq\langle \overline{\bT}\bn - \bT\bn, \bw\rangle|_{\Gamma_N} \textrm{ for all } \bw \in \mathcal{C}^1_0(\Gamma_N)^d.
\end{equation}
Consequently, for $\bw \in \mathcal{C}^1_{\Gamma_D}(\overline{\Omega})^d$ we have
\begin{equation*}
\begin{split}
\int_{\Omega}\bT\cdot \beps(\bw)-\bef \cdot \bw \ddd x-\langle \bg-\tilde{\bg}, \bw\rangle|_{\Gamma_N}&\overset{\eqref{Tbar_eq}}=\int_{\Omega}\bT\cdot \beps(\bw)\ddd x-\langle\overline{\bT},\beps(\bw)\rangle +\langle \tilde{\bg}, \bw\rangle|_{\Gamma_N}\\
&\!\!\!\!\!\!\!\overset{\eqref{trace},\eqref{trace3}}= \langle \bT\bn - \overline{\bT}\bn,\bw \rangle|_{\Gamma_N}+\langle \tilde{\bg}, \bw\rangle|_{\Gamma_N}\\
&\overset{\eqref{finmeasure}}=0,
\end{split}
\end{equation*}
which is nothing else than~\eqref{TT1}. Hence, the proof is complete. 

\subsection{Identification of $\overline{\bT}$}
In this final part of the proof of our main theorem we provide an improved characterization of the weak-$*$ limit $\overline{\bT}$ in terms of $\beps(\bu)$; the discussion in this section was inspired by some ideas of Anzellotti, see~\cite{A1984}. We begin by decomposing $\overline{\bT}$ into its regular and singular parts, i.e.,
\begin{equation}\label{Idd1}
\overline{\bT}= \bT^r + \bT^s,
\end{equation}
where $\bT{^r} \in L^{1}(\Omega)^{d\times d}$ and $\bT^s\in \mathcal{M}(\overline{\Omega})^{d\times d}$ is a singular Radon measure supported on a set of zero Lebesgue measure. When we work under the assumption~\eqref{BiFu}, we can use the convergence $\bT^n \to \bT$ in $L^1_{loc}(\Omega)^{d \times d}$ established at the beginning of Subsect~\ref{subsection_renormalized}, which implies immediately that~$\bT^s$ can only be supported on~$\overline{\Gamma_N}$ and there is nothing to be proved for the identification of $\overline{\bT}$. Therefore, we focus in what follows on the asymptotic Uhlenbeck setting with~\eqref{A4} and~\eqref{A5}. We denote by $\mu\in \mathcal{M}(\overline{\Omega})$ a Radon measure that fulfills
\begin{equation}\label{Idd2}
\begin{split}
\bT^n\cdot \beps(\bu^n) \rightharpoonup^* \bT \cdot \beps(\bu) + \mu \quad \textrm{ weakly-$*$ in } \mathcal{M}(\overline{\Omega}).
\end{split}
\end{equation}
Note here that {the sequence $\bT^n \cdot \beps(\bu^n)$ is bounded in $L^1(\Omega)$ as a consequence of the a~priori estimates~\eqref{kstart2.5} and~\eqref{stap}, and that} the (possibly regular) measure $\mu$ is nonnegative, thanks to Fatou's lemma, the pointwise convergence of $\bT^n$ (see~\eqref{C1point}) and $\beps(\bu^n)$ (see~\eqref{C1point*}), and the boundedness from below of $\bT \cdot \beps^*(\bT)$, see~\eqref{AE2}.

Our first goal is to show that $\bT^r=\bT$ almost everywhere in $\Omega$ and that $\mu$ is a singular measure, i.e., it is supported on a set of zero Lebesgue measure. When $\mathcal{A}$ is symmetric and, consequently,~$\beps^*$ has a potential, this is a direct consequence of the inequality~\eqref{TT2}; see also the proof of Lemma~\ref{AL4} in the Appendix. In the general case (i.e., when, as is the case here, $\mathcal{A}$ is only assumed to be asymptotically symmetric), we have to use a different technique, which, in a certain sense, mimics the variational approach. Thanks to the monotonicity of $\beps^{*}$ we have that, almost everywhere in $\Omega$,
$$
\begin{aligned}
0&\le (\bT^n -\bB)\cdot \bigg(\beps^*(\bT^n)+\frac{\bT^n}{n(1+|\bT^n|^2)^{\frac{n-1}{2n}}}-\beps^*(\bB)-\frac{\bB}{n(1+|\bB|^2)^{\frac{n-1}{2n}}}\bigg)\\
&=(\bT^n -\bB)\cdot \bigg(\beps(\bu^n)-\beps^*(\bB)-\frac{\bB}{n(1+|\bB|^2)^{\frac{n-1}{2n}}}\bigg),
\end{aligned}
$$
where $\bB\in \mathcal{C}(\overline{\Omega})^{d\times d}$ is arbitrary. Thanks to~\eqref{C1point},~\eqref{C1point*}, the definition of $\mu$ and the boundedness of~$\beps^*$, we get that
\begin{align}
0&\le (\bT^n -\bB)\cdot \bigg(\beps(\bu^n)-\beps^*(\bB)-\frac{\bB}{n(1+|\bB|^2)^{\frac{n-1}{2n}}}\bigg) \rightharpoonup^* \lambda_{\bB} \quad \textrm{weakly-$*$ in } \mathcal{M}(\overline{\Omega}),\notag\\
\intertext{where}\label{lambdadef}
&0\le \lambda_{\bB}\coloneqq \bT \cdot \beps(\bu)+\mu -\bT^r \cdot \beps^*(\bB)-\bT^s \cdot \beps^*(\bB)-\bB\cdot \left(\beps(\bu)-\beps^*(\bB)\right)\quad \textrm{ in } \Omega \cup \overline{\Gamma_N}.
\end{align}
On the other hand, using~\eqref{elasticm},~\eqref{elasticm2},~\eqref{C1} and~\eqref{C2.5}, we can deduce for arbitrary $\varphi \in \mathcal{C}^{\infty}_{\Gamma_D}(\overline{\Omega})$ that
\begin{equation}\label{monotonr}
\begin{aligned}
\int_{\Omega}\bT\cdot \beps(\bu)\varphi \ddd x + \langle \mu,\varphi \rangle &= \lim_{n\to \infty}\int_{\Omega}\bT^n\cdot \beps(\bu^n)\varphi \ddd x \\
&=\lim_{n\to \infty}\int_{\Omega}\bT^n\cdot \beps(\bu^n \varphi) \ddd x-\lim_{n\to \infty}\int_{\Omega}\bT^n\cdot (\bu^n\otimes \nabla \varphi) \ddd x\\
&=\lim_{n\to \infty} \int_{\Omega}\bef \cdot \bu^n\varphi \ddd x+\int_{\Gamma_N} \bg\cdot \bu^n \varphi \ddd S -\lim_{n\to \infty}\int_{\Omega}\bT^n \cdot (\bu^n \otimes \nabla \varphi)\ddd x\\
&=\int_{\Omega}\bef \cdot \bu\,\varphi \ddd x+\int_{\Gamma_N} \bg\cdot \bu\, \varphi \ddd S -\langle \overline{\bT}, \bu \otimes \nabla \varphi\rangle.
\end{aligned}
\end{equation}
Since $\Omega$ is Lipschitz, we can use Lemma~\ref{dense} and find $\bu^{\varepsilon}\in \mathcal{C}^1_{\Gamma_D}(\overline{\Omega})^d$ such that (note that when $\varphi$ has compact support in $\Omega$ we can trivially take any $\varepsilon$-mollification of~$\bu$ as $\bu^{\varepsilon}$)
$$
\begin{aligned}
\bu^{\varepsilon}&\to \bu &&\textrm{ strongly in } \mathcal{C}(\overline{\Omega})^d \cap W^{1,1}(\Omega)^d,\\
\beps(\bu^{\varepsilon})&\rightharpoonup^* \beps(\bu) &&\textrm{ weakly-$*$ in } L^{\infty}(\Omega)^{d\times d}.
\end{aligned}
$$
Consequently, it follows from~\eqref{monotonr} and~\eqref{Tbar_eq} that
\begin{equation}\label{monotonr2}
\begin{aligned}
\int_{\Omega}\bT\cdot \beps(\bu)\varphi \ddd x + \langle \mu,\varphi \rangle &= \lim_{\varepsilon\to 0_+} \int_{\Omega}\bef \cdot \bu^{\varepsilon}\varphi \ddd x + \int_{\Gamma_N} \bg\cdot \bu^{\varepsilon}\varphi \ddd S -\langle \overline{\bT} , \bu^{\varepsilon} \otimes \nabla \varphi\rangle\\
&= \lim_{\varepsilon\to 0_+} \int_{\Omega}\overline{\bT} \cdot\beps(\bu^{\varepsilon})\varphi\ddd x \\
&= \int_{\Omega}\bT^{r}\cdot\beps(\bu)\varphi\ddd x +\lim_{\varepsilon\to 0_+} \langle \bT^s,\beps(\bu^{\varepsilon}) \varphi \rangle\\
&= \int_{\Omega}\bT^{r}\cdot\beps(\bu)\varphi\ddd x +\langle \tilde{\mu},\varphi\rangle,
\end{aligned}
\end{equation}
where $\tilde{\mu}\in \mathcal{M}(\overline{\Omega})$ is the weak-$\ast$ limit of $\bT^s\cdot\beps(\bu^{\varepsilon})$, which is necessarily absolutely continuous with respect to $|\bT^s|$. In addition, by a density argument, the above relation holds also for all $\varphi \in \mathcal{C}_{\Gamma_D}(\overline{\Omega})$. Comparing~\eqref{lambdadef} and~\eqref{monotonr2}, we see that
\begin{equation}\label{lambdaII}
0\le \lambda_{\bB}=\tilde{\mu}-\bT^s \cdot \beps^*(\bB)+\left(\bT^r-\bB\right)\cdot \left(\beps(\bu)-\beps^*(\bB)\right) \quad \textrm{ in } \Omega\cup \overline{\Gamma_N}.
\end{equation}
However, since $\tilde{\mu} $ and $|\bT^s|$ are singular measures (supported on a set of zero Lebesgue measure) and since $\beps(\bu)=\beps^*(\bT)$ in~$\Omega$, we deduce that
\begin{equation}\label{Minty3}
\begin{aligned}
(\bT^{r} -\bB)\cdot(\beps^*(\bT)-\beps^*(\bB))&\ge 0 \quad \textrm{ almost everywhere in } \Omega,\\
\tilde{\mu}-\bT^s \cdot \beps^*(\bB) & \ge 0 \quad \textrm{ $|\bT^s|$-almost everywhere in } \Omega \cup \overline{\Gamma_N}.
\end{aligned}
\end{equation}
Consequently, since $\beps^*(\bT)$ is strictly monotone, we can use Minty's method to show that
\begin{equation}
\label{key_ident}
\bT^r=\bT \quad \textrm{ almost everywhere in } \Omega,
\end{equation}
and using also~\eqref{ren6.6}, we see that
\begin{equation}
\label{key_ident2}
\diver \bT^s=0 \quad \textrm{ in } \Omega \text{ in the sense of distributions.}
\end{equation}

In addition, having the identification~\eqref{key_ident} and comparing~\eqref{lambdadef} and~\eqref{lambdaII}, we obtain that $\mu$ is absolutely continuous with respect to $\tilde{\mu}$ with
\begin{equation}
\label{muid}
\mu=\tilde{\mu} \quad \textrm{ $|\bT^s|$-almost everywhere in }\Omega \cup \overline{\Gamma_N},
\end{equation}
and hence, $\mu$ is in particular a singular measure, as was claimed.

We continue by showing further, more refined, properties of $\bT^s$. Recalling \eqref{muid} and the identification of $\tilde{\mu}\in \mathcal{M}(\overline{\Omega})$ as the weak-$\ast$ limit of $\bT^s\cdot\beps(\bu^{\varepsilon})$, we can find $\overline{\beps}$ such that
\begin{equation}\label{ks1}
\beps(\bu^{\varepsilon}) \rightharpoonup^* \overline{\beps} \quad \textrm{ weakly-$*$ in } L^{\infty}(\Omega\cup \overline{\Gamma_N}, |\bT^s|)
\end{equation}
and
\begin{equation}\label{ks2}
\mu= \frac{\bT^s}{|\bT^s|} \cdot \overline{\beps}\; |\bT^s|,
\end{equation}
where $\frac{\bT^s}{|\bT^s|}$ denotes the Radon--Nykod\'ym derivative of $\bT^s$ with respect to $|\bT^s|$. Hence, it follows from~\eqref{Minty3} that
\begin{equation}
\label{Minty4}
\frac{\bT^s}{|\bT^s|} \cdot (\overline{\beps} - \beps^*(\bB)) \ge 0 \quad \textrm{ $|\bT^s|$-almost everywhere in } \Omega\cup \overline{\Gamma_N},
\end{equation}
for any $\bB$ that is $|\bT|$-measurable.

We shall henceforth confine ourselves to the interior of $\Omega$. Therefore, in what follows, we denote
by $\bu^{\varepsilon}$ the standard mollification of~$\bu$, and we proceed as follows. Since $\beps^*(\mathbb{R}^{d\times d})$ is convex\footnote{For any $\beps_1$, $\beps_2 \in \beps^*(\mathbb{R}^{d\times d})$, i.e., $\beps_1=\beps^*(\bB_1)$ and $\beps_2=\beps^*(\bB_2)$ for certain $\bB_1\in \mathbb{R}^{d\times d}$ and $\bB_2\in \mathbb{R}^{d\times d}$, we aim to show that, for any $\lambda\in (0,1)$, there is a $\bB_3=\bB_3(\bB_1, \bB_2, \lambda)$ such that
\begin{equation}\label{f_star}
\bP(\bB_3)\coloneqq \beps^*(\bB_3) - \lambda \beps^*(\bB_1) - (1-\lambda) \beps^*(\bB_2) = \bO.
\end{equation}
This is however a consequence of Brouwer's fixed point theorem provided that $\bP(\bB_3) \cdot \bB_3 >0$ for all $\bB_3$ fulfilling $|\bB_3|=\varrho$ with some $\varrho>0$. Note that this last condition follows from noticing that
\[\bP(\bB_3) \cdot \bB_3 = \lambda (\beps^*(\bB_3)\cdot \bB_3 - \beps^*(\bB_1) \cdot \bB_3) + (1-\lambda) (\beps^*(\bB_3)\cdot \bB_3 - \beps^*(\bB_2) \cdot \bB_3)\]
and from the fact that both $\bB_1$ and $\bB_2$ satisfy the safety strain condition~\eqref{G2st_analytic}, as $\beps_1$ and $\beps_2$ belong to the interior of $\beps^*(\mathbb{R}^{d\times d})$ (cf.~the proof of Lemma~\ref{G2-equiv})}. and $\beps(\bu)\in \beps^*(\mathbb{R}^{d\times d})$ (cf.~\eqref{RShq}), we see that $\beps(\bu^{\varepsilon})\in \beps^*(\mathbb{R}^{d\times d})$ and, consequently, there exists a continuous $\bB^{\varepsilon}$ such that $\beps(\bu^{\varepsilon})=\beps^*(\bB^{\varepsilon})$. Using such a $\bB=\bB^\varepsilon$ in~\eqref{Minty4} and recalling~\eqref{ks1}, we see that
\begin{equation*}
\frac{\bT^s}{|\bT^s|} \cdot \beps(\bu^{\varepsilon}) \to \frac{\bT^s}{|\bT^s|} \cdot \overline{\beps}\quad \textrm{ strongly in } L^1(\Omega, |\bT^s|),
\end{equation*}
and therefore also,
\begin{equation}
\frac{\bT^s}{|\bT^s|} \cdot \beps(\bu^{\varepsilon}) \to \frac{\bT^s}{|\bT^s|} \cdot \overline{\beps}\quad \textrm{ strongly in } L^p(\Omega, |\bT^s|),
\label{cmeass}
\end{equation}
for all $p\in [1,\infty)$. Next, we infer from~\eqref{Minty4} that $\overline{\beps}\in \partial \beps^*(\mathbb{R}^{d\times d})$ holds $|\bT^s|$-almost everywhere; otherwise, there would exist a $\bB$ such that ~\eqref{Minty4} holds with the $\geq$ sign replaced by the $<$ sign, resulting in a contradiction. Since we have assumed the asymptotic Uhlenbeck setting~\eqref{A5}, we may now work under the assumption that $\lim_{t \rightarrow +\infty}\frac{t}{g(t)}=:\alpha>0$\footnote{In fact, relying on~\eqref{A4},~\eqref{A5},~\eqref{AE1} and the fact that $h(t) t \to 0$ as $t \to \infty$ (cf.~the comment after~\eqref{A3-g}), we first obtain
\begin{equation*}
 0 < \liminf_{t \to \infty} \tfrac{t}{g(t)} \leq \limsup_{t \to \infty} \tfrac{t}{g(t)} < \infty,
\end{equation*}
so the limit, if it exists, is positive and finite. As a consequence, we find
\begin{equation}
\label{eqn_t_gt_estimate}
 \big|\beps^* (\bT) - \tfrac{\bT}{g(|\bT|)}\big| \le \tfrac{C_2^{1/2}(1+|\bT|)}{g(|\bT|)}\, h^{1/2}(\bT) |\bT|^{1/2} = o(1) \quad \textrm{ as } |\bT|\to \infty,
\end{equation}
which, in turn, implies for each $\bT \in \mathbb{R}^{d\times d}_{sym}$ with $|\bT|=1$
\begin{equation*}
\lim_{t \to \infty} \big|\beps^* (t\bT) \cdot \bT - \tfrac{t}{g(t)}\big| = 0.
\end{equation*}
We then notice that $t \mapsto \beps^* (t\bT) \cdot \bT$ is bounded and monotone increasing by~\eqref{A1}, hence, the limit $\lim_{t \to \infty} \beps^* (t\bT) \cdot \bT$ exists (and is the same) for each $\bT \in \mathbb{R}^{d\times d}_{sym}$ with $|\bT|=1$, Thus, also the limit $\lim_{t \to \infty}\frac{t}{g(t)}$ exists as claimed.}, which implies that $\lim_{n\to \infty}\beps^*(n\bB)=\alpha \frac{\bB}{|\bB|}$ for each $\bB \in \mathbb{R}^{d\times d}_{sym}$ with $\bB \neq \b0$. Consequently, we see that $\overline{\beps}=\alpha \bT_I$ for some $\bT_{I}$ with $|\bT_{I}|=1$. Taking $\bB\coloneqq n\bB_I$ for some $\bB_I$ with $|\bB_I|=1$ in~\eqref{Minty4} and letting $n\to \infty$ we deduce that
$$
\frac{\bT^s}{|\bT^s|} \cdot \left(\bT_I - \bB_I\right) \ge 0 \qquad |\bT^s|\textrm{-almost everywhere in } \Omega,
$$
and consequently we have
\begin{equation}\label{gooddd}
\frac{\bT^s}{|\bT^s|} =\bT_I = \alpha^{-1} \overline{\beps} \qquad |\bT^s|\textrm{-almost everywhere in } \Omega.
\end{equation}
Next, we observe\footnote{To see this, we investigate the supremum of the mapping
$\bT \in \mathbb{R}^{d\times d}_{\rm sym} \mapsto|\beps^*(\bT)|^2$. Assuming first that the maximum value is attained at a point $\bT \in \mathbb{R}^{d\times d}_{\rm sym}$, we automatically deduce from the fact that~$\bT$ is then a stationary point of $|\beps^{*}(\cdot)|^2$ that
\[ 2\sum_{i,j=1}^d \beps^*_{ij}(\bT) \mathcal{A}_{ijkl} (\bT) = 0 \quad \textrm{ for all } k, l = 1, \dots, d.
\]
Multiplying this relation by $\beps^*_{kl}(\bT)$ and summing over $k,l=1, \dots, d$ we immediately obtain a
contradiction with~\eqref{A1}. Thus, the supremum of $\bT \in \mathbb{R}^{d\times d}_{\rm sym} \mapsto|\beps^*(\bT)|^2$ is attained for $|\bT|\to \infty$. In view of~\eqref{eqn_t_gt_estimate}, recalling the definition of $\alpha$, we thus have that $\lim_{|\bT|\to \infty}\beps^*(\bT) = \alpha$,
and the assertion follows.} that $|\beps(\bu)|\le \alpha$ holds for all~$\bu$ and~$\bT$ satisfying $\beps(\bu) = \beps^*(\bT)$, see~\eqref{RShq}. Thus, $|\beps(\bu^{\varepsilon})|^2 = |\beps^*(\bB^{\varepsilon})|^2 \le \alpha^2$, and we see that
\begin{equation*}
\alpha^{-1} \beps(\bu^{\varepsilon}) \to \frac{\bT^s}{|\bT^s|} \qquad \textrm{strongly in } L^2(\Omega, |\bT^s|),
\end{equation*}
which follows from the following calculations, valid for an arbitrary nonnegative $\varphi\in \mathcal{D}(\Omega)$:
$$
\begin{aligned}
&\lim_{\varepsilon\to 0_+}\int_{\Omega}\varphi \left|\beps(\bu^{\varepsilon})-\alpha \frac{\bT^s}{|\bT^s|}\right|^2 \ddd |\bT^s|=\lim_{\varepsilon\to 0_+}\int_{\Omega}\varphi \left(|\beps(\bu^{\varepsilon})|^2+\alpha^2-2\alpha \beps(\bu^{\varepsilon})\cdot \frac{\bT^s}{|\bT^s|}\right) \ddd |\bT^s|\\
&\quad \le \lim_{\varepsilon\to 0_+}2\int_{\Omega}\varphi \left(\alpha^2-\alpha \beps(\bu^{\varepsilon})\cdot \frac{\bT^s}{|\bT^s|}\right) \ddd |\bT^s| =0,
\end{aligned}
$$
where, for the last equality, we have used~\eqref{cmeass} and~\eqref{gooddd}. We thus deduce that
\[
\alpha^{-1} \beps(\bu^{\varepsilon}) \to \frac{\bT^s}{|\bT^s|} \qquad \textrm{strongly in } L^2_{\rm loc}(\Omega, |\bT^s|),
\]
and therefore, for a subsequence (not indicated),
\[
\alpha^{-1} \beps(\bu^{\varepsilon}) \to \frac{\bT^s}{|\bT^s|} \qquad \textrm{$|\bT^s|$-a.e. in $\Omega$}.
\]
As the sequence $(\alpha^{-1} \beps(\bu^{\varepsilon}))_{\varepsilon>0}$ is bounded in $L^\infty(\Omega,|\bT^s|)$, and therefore in $L^p(\Omega,|\bT^s|)$ for all $p \in [1,\infty]$, we can extract a further subsequence (not indicated), which weakly converges in $L^1(\Omega,|\bT^s|)$. Since this subsequence is also a.e. convergent in $\Omega$ with respect to the measure $|\bT^s|$, Vitali's theorem implies that
\[
\alpha^{-1} \beps(\bu^{\varepsilon}) \to \frac{\bT^s}{|\bT^s|} \qquad \textrm{strongly in } L^1(\Omega, |\bT^s|).
\]
Thus, noting once again that $(\alpha^{-1} \beps(\bu^{\varepsilon}))_{\varepsilon>0}$ is bounded in $L^\infty(\Omega,|\bT^s|)$, it
directly follows that
\begin{equation}\label{gooddd2}
\alpha^{-1} \beps(\bu^{\varepsilon}) \to \frac{\bT^s}{|\bT^s|} \qquad \textrm{strongly in } L^p(\Omega, |\bT^s|)
\end{equation}
for all $p\in [1,\infty)$. Next, we show that for any compact set $K\subset \Omega$ we have that
\begin{equation}\label{maxedf}
\spt |\bT^s| \cap K \subset \{x\in \Omega \colon M(|\beps (\bu)|)(x)=\alpha\}.
\end{equation}
Indeed, for arbitrary $\lambda\in (0,\alpha)$ we consider the set $K_{\lambda}\coloneqq \{x\in K \colon M(|\beps (\bu)|)(x)\le \alpha-\lambda\}$.
Note that it follows from the properties of the maximal function $M$ that $K_{\lambda}$ is a closed set. Moreover, we have $|\beps(\bu^{\varepsilon}(x))|\le \alpha - \lambda$ for every $x \in K_\lambda$, for sufficiently small $\varepsilon$ (strictly less than the distance of $K$ to $\partial \Omega$).
Thus, by weak-$\ast$ lower-semicontinuity (or, alternatively, by weak-$\ast$ closedness of convex closed sets) we also have that $|\bar{\beps}(x)|\le \alpha - \lambda$ for $|\bT^s|$-almost every $x \in K_{\lambda}$. On the other hand we already know that $|\bar{\beps}|= \alpha$, $|\bT^s|$-almost everywhere in $K_{\lambda}$. From this it immediately follows that $|\bT^s|(K_{\lambda})=0$. Since $\lambda \in (0,\alpha)$ was chosen to be arbitrary, we get~\eqref{maxedf}.
Thanks to the assumed asymptotic Uhlenbeck structure, we however also know that
$$
\lim_{R \rightarrow 0^+} R^{-d}\int_{B_{R}(x)} |\beps^*(\bT(y))| \ddd y = \alpha \quad \Longleftrightarrow \quad
\lim_{R \rightarrow 0^+} R^{-d}\int_{B_{R}(x)} |\bT(y)| \ddd y = +\infty.
$$
Consequently, we conclude from~\eqref{maxedf} that, for each compact set $K\subset \Omega$, we have
$$\spt |\bT^s| \cap K \subset \{x\in \Omega \colon M(|\bT|)(x)=\infty\}.$$

Finally, it is not difficult to deduce that
\begin{equation}\label{gooddd3}
\left(\frac{\bT}{|\bT|}\right)^{\varepsilon} \to \frac{\bT^s}{|\bT^s|} \qquad \textrm{strongly in } L^p(\Omega, |\bT^s|).
\end{equation}
The proof of this statement proceeds as follows. We begin by noting that
\begin{align*}
\left|\left(\alpha\frac{\bT}{|\bT|}\right)^{\varepsilon} - \beps(\bu^{\varepsilon})\right|
 & \le \left|\left(\alpha\frac{\bT}{|\bT|}-\frac{\bT}{g(|\bT|)}\right)^{\varepsilon}\right|+\left|\left(\frac{\bT}{g(|\bT|)} - \beps^*(\bT)\right)^{\varepsilon}\right| \\
 & \le \left|\left(\alpha-\frac{|\bT|}{g(|\bT|)}\right)^{\varepsilon}\right|
 + C_2^{1/2} \left(\left(h(|\bT|) (1+ |\bT|^3) g(|\bT|)^{-2} \right)^{\varepsilon}\right)^{\frac12},
\end{align*}
where we have used~\eqref{A5} and Jensen's inequality. Hence, using the fact that $|\bT^s|$ is supported only
on the set where $M(|\bT|)=\infty$, the definition of $\alpha$ and that $h(s)s \to 0 $ as
$s \to \infty$, we obtain
$$
\int_{\Omega} \left|\left(\alpha\frac{\bT}{|\bT|}\right)^{\varepsilon} - \beps(\bu^{\varepsilon})\right|\ddd |\bT^s| \to 0
$$
as $\varepsilon \to 0_+$. Consequently,~\eqref{gooddd3} follows from~\eqref{gooddd2}, which then finally completes the
proof of our main theorem, Theorem~\ref{T1}.

\begin{appendix}

\section{Tools}\label{Saux}
We complete the paper by recalling the properties of~$\cA$,~$F$ and~$F^*$ introduced in Section~\ref{Sec1}. Although such properties are easy consequences of the assumptions, we provide the detailed proofs in our setting. Furthermore, we provide the reformulation of the safety strain conditions~\eqref{proto-G2} and \eqref{proto-G2st} in the analytic forms~\eqref{G2_analytic} and~\eqref{G2st_analytic}, respectively. Finally, we state a density result for smooth, compactly supported functions.

\begin{Lemma}\label{basici}
Let~$\cA$ satisfy~\eqref{AE2-g} and~\eqref{A1-g}. Then, for all $\bT_1, \bT_2 \in \mathbb{R}^{d\times N}$ one has the following inequality:
\begin{align}
&(\bD(\bT_1) - \bD(\bT_2))\cdot (\bT_1-\bT_2)\ge h(|\bT_1|+|\bT_2|)|\bT_1-\bT_2|^2.\label{subst2}
\end{align}
Moreover, the set $\overline{\bD(\mathbb{R}^{d\times N})}$ is convex.
In addition, if~$\cA$ is symmetric then $F \colon \mathbb{R}^{d\times N}\to \mathbb{R}_+$ defined by~\eqref{potende} satisfies~\eqref{derpot} and is strictly convex. Furthermore,~$F^*$ defined in~\eqref{conjug} is strictly convex on $\bD(\mathbb{R}^{d\times N})$ and satisfies~\eqref{propFstar} and~\eqref{derpot2}.
\end{Lemma}
\begin{proof}
First, we focus on~\eqref{subst2}. Using the definitions of~$\cA$ and $(\cdot,\cdot)_{\cA}$ (cf.~\eqref{DefcA-g} and~\eqref{A1-g}) and the lower bound from assumption~\eqref{A1-g}, we have
\begin{equation*}
\begin{aligned}
(\bD(\bT_1) - \bD(\bT_2))\cdot (\bT_1-\bT_2)
&\quad =(\bT_1-\bT_2)\cdot \int_{0}^1 \frac{\ddd}{\ddd s} \bD(\bT_2+s(\bT_1-\bT_2))\ddd s\\
&\quad =\int_{0}^1 (\bT_1-\bT_2,\bT_1-\bT_2)_{\cA(\bT_2+s(\bT_1-\bT_2))}\ddd s\\
&\quad \ge \int_{0}^1 h(|\bT_2+s(\bT_1-\bT_2)|) |\bT_1-\bT_2|^2\ddd s\\
&\quad \ge h(|\bT_2|+|\bT_1|)|\bT_1-\bT_2|^2,
\end{aligned}
\end{equation*}
where the last inequality follows from the fact that $h$ is nonincreasing. The convexity of the set $\overline{\bD(\mathbb{R}^{d\times N})}$ then follows from~\cite{Ro70}\footnote{In fact, in~\cite{Ro70}, Rockafellar works in the more general context of maximal monotone operators, which covers our case here since~$\bD$ is monotone and continuous. In particular, the convexity of the closure of $\bD(\mathbb{R}^{d \times N})$ follows from Theorem 2 in~\cite{Ro70}.}. Next, we verify the formula~\eqref{derpot}. Thanks to first the assumed symmetry property $\cA_{j\mu i\nu } = \cA_{i\nu j \mu}$ for all $i,j=1,\ldots, d$ and $\nu,\mu=1,\ldots, N$, we have that
\begin{equation*}
\begin{aligned}
\frac{\partial F(\bT)}{\partial \bT_{i\nu}} & = \int_0^1 \bD_{i\nu}(t\bT) + \sum_{j=1}^d \sum_{\mu=1}^N \frac{\partial \bD_{j\mu}(t\bT)}{\partial \bT_{i\nu}} \bT_{j\mu} t \ddd t\\
& = \int_0^1 \bD_{i\nu}(t\bT) + \sum_{j=1}^d \sum_{\mu=1}^N \frac{\partial \bD_{i\nu}(t\bT)}{\partial \bT_{j\mu}} \bT_{j\mu} t \ddd t
= \int_0^1 \frac{\ddd}{\ddd t} \left(t \bD_{i\nu}(t \bT)\right) \ddd t = \bD_{i\nu}(\bT).\\
\end{aligned}
\end{equation*}

With this identification, the strict convexity of~$F$ follows directly from~\eqref{subst2}. Consequently, this yields the strict convexity of~$F^*$ on (the open set) $\bD(\mathbb{R}^{d\times N})$, and the identity~\eqref{derpot2} is a standard result from convex analysis, cf. \cite[Chapter I.5]{EKETEM99}. To verify also~\eqref{propFstar} in detail, let us first consider $\bB \notin \overline{\bD(\mathbb{R}^{d\times N})}$. We start by observing that thanks to~\eqref{AE2-g} and applying Corollary 1.1 to Brouwer's fixed point theorem on p.~279 in~\cite{GR}, one has $\bD(\bT_0) = \b0$ for some $\bT_0 \in \mathbb{R}^{d\times N}$.
Then, thanks to the convexity of the set $\overline{\bD(\mathbb{R}^{d\times N})}$, there exists a $\lambda \in (0,1)$ such that
$$
t\bB \in \bD(\mathbb{R}^{d\times N}) \textrm{ for all } t\in [0,\lambda)
$$
and $\lambda \bB \in \partial \bD(\mathbb{R}^{d\times N})$. Consequently, there exists a sequence $\bT_t$ such that $t\bB = \bD(\bT_t)$ for $t \in [0,\lambda)$, with $|\bT_t| \to \infty$ as $t \to \lambda$. Thus, using the definition of~$F^*$ (cf.~\eqref{conjug}) and noting that, thanks to the convexity of~$F$, $F(\bT_0) \geq F(\bT_t) + \bD(\bT_t)\cdot(\bT_0 - \bT_t)$, we have by~\eqref{AE2-g} and~\eqref{AE1-g} that
$$
F^*(\bB)\ge \bB \cdot \bT_t - F(\bT_t)=\frac{1-t}{t} \bD(\bT_t) \cdot \bT_t +\bD(\bT_t)\cdot \bT_t -F(\bT_t)\ge \frac{1-t}{t} (C_1 |\bT_t|-C_0) - F(\bT_0) - C_2|\bT_0|.
$$
Finally, letting $t\to \lambda <1$, we get $F^*(\bB)=\infty$. Similarly, if $\bB = \bD(\bT_B)$ for some $\bT_B \in \mathbb{R}^{d\times N}$, we deduce by the convexity of~$F$ that
$$
\bB \cdot \bT - F(\bT)= \bD(\bT_B) \cdot \bT - F(\bT) \le \bD(\bT_B) \cdot \bT_B -F(\bT_B)
$$
for every $\bT \in \mathbb{R}^{d\times N}$. Hence, the second identity in~\eqref{propFstar} follows by setting $\bT=\bT_B$ in the definition of~$F^*$.
\end{proof}

\begin{Lemma}
\label{G2-equiv}
If $\bD \colon \mathbb{R}^{d\times N} \to \mathbb{R}^{d\times N}$ satisfies~\eqref{AE2-g},~\eqref{AE1-g} and~\eqref{A1-g}, then the conditions~\eqref{proto-G2} and~\eqref{G2_analytic} are equivalent. Similarly, if $\beps^* \colon \mathbb{R}^{d\times d}_{sym} \to \mathbb{R}^{d\times d}_{sym}$ satisfies~\eqref{AE2},~\eqref{AE1} and~\eqref{A1}, then the conditions~\eqref{proto-G2st} and~\eqref{G2st_analytic} are equivalent.

\begin{proof}
We only prove the assertion on~$\bD$, since the proof of the assertion on~$\beps^*$ is essentially the same. First, we assume~\eqref{proto-G2} to be satisfied. Then, in view of Lemma~\ref{basici}, we have $\nabla \bu_0 = \bD(\bT_0)$ almost everywhere in $\Omega$ for some measurable function $\bT_0$, which satisfies $\| \bT_0\|_\infty \leq C$ for a constant $C< \infty$.
For each Lebesgue point $x$ of $\bT_0$, we then infer from~\eqref{AE1-g} and~\eqref{A1-g} (note that $h$ is nonincreasing) for every $\bT \in \mathbb{R}^{d\times N}$ with $|\bT|=1$
that
\begin{align*}
 & \liminf_{n \to \infty} \big( \bD(n\bT) - \bD(\bT_0(x)) \big) \cdot \bT \\
 & = \liminf_{n \to \infty} \big( \bD(n\bT) - \bD(\bT_0(x)) \big) \cdot (\bT - n^{-1} \bT_0(x)) \\
 & \geq \liminf_{n \to \infty} \int_{0}^1 h( |sn \bT + (1-s)\bT_0(x) |) \ddd s \, |n \bT - \bT_0(x)|^2 n^{-1} \\
 & \geq \liminf_{n \to \infty} \int_{0}^{1/n} h( sn + (1-s) C) \ddd s \, |n \bT - \bT_0(x)|^2 n^{-1} \geq h(C+1) > 0,
\end{align*}
which proves~\eqref{G2_analytic} (possibly after adapting the choice of $C_1$ in~\eqref{AE2-g}). For the reverse implication we assume~\eqref{G2_analytic} to be satisfied. Thus, we find a constant $t_c$ such that we have
\begin{equation*}
\big(\bD(\bT)-\nabla \bu_0(x)\big) \cdot \bT = \big(\bD(|\bT| \hat{\bT}) -\nabla \bu_0(x)\big) \cdot \hat{\bT} |\bT| \ge \frac{C_1|\bT|}{2} \qquad \text{for all }|\bT|\ge t_c \text{ and a.e. } x \in \Omega\,,
\end{equation*}
where we have denoted $\hat{\bT} \coloneqq \bT / |\bT|$. Since~$\bD$ is continuous, according to (a standard consequence of) Brouwer's fixed-point theorem there exists a measurable map $\bT_0 \colon \Omega \to \mathbb{R}^{d\times N}$ with $\|\bT_0\|_{\infty} \le t_c$, which satisfies $\bD(\bT_0(x))=\nabla \bu_0(x)$ for a.e.~$x \in \Omega$. Thus, we arrive at~\eqref{proto-G2} with the compact set $K$ defined as
$\{\bD(\bT)\, :\, |\bT| \leq t_c \}$.
\end{proof}
\end{Lemma}

\begin{Lemma}[Density of smooth compactly supported functions]\label{dense}
Let $\Omega \subset \mathbb{R}^d$ be a domain with Lipschitz boundary. Assume that $\Gamma_D \subset \partial \Omega$ is a relatively open Lipschitz set. Then, for any $\bu \in W^{1,1}_{\Gamma_D}(\Omega)^d$ with $\beps (\bu)
\in L^{\infty}(\Omega)^{d\times d}$, there exists a sequence $(\bu^n)_{n\in \mathbb{N}}$ such that $\bu^n \in \mathcal{C}^1_{\Gamma_D}(\overline{\Omega})^d$ for all $n\in \mathbb{N}$ and
\begin{align}\label{AAA1}
\bu^n &\to \bu &&\textrm{strongly in } W^{1,p}_{\Gamma_D}(\Omega)^d \; \textrm{ for all } 1\le p<\infty,\\
\beps(\bu^n) &\rightharpoonup^* \beps(\bu) &&\textrm{weakly-$*$ in } L^{\infty}(\Omega)^{d\times d}.\label{AAA2}
\end{align}
\end{Lemma}

\begin{proof}
Since $\Omega$ is a Lipschitz domain, we know that the boundary can be covered by a finite number of, say $k$, open sets, where the boundary is described as a graph of a Lipschitz mapping. Moreover, after possibly adding to such a covering a proper open set $\Omega_0 \subset \overline{\Omega}_0\subset \Omega$, we can find a corresponding partition of unity
$\{\tau_i\}_{i=0}^k$ and decompose $\bu=\sum_{i=0}^k \tau_i \bu=:\sum_{i=1}^k \bu_i$
where, by Korn's inequality, we have $\bu_i \in W^{1,p}(\Omega)^d$ for all $p\in [1,\infty)$ and $\beps(\bu_i)\in L^{\infty}(\Omega)^{d\times d}$, for each $i \in
\{1,\ldots,k \}$. Next, we can follow step by step the Appendix in~\cite{BMRS2014} to show that each $\bu_i$ can be approximated by a smooth function satisfying~\eqref{AAA1},~\eqref{AAA2} whenever one considers only the interior of $\Omega$ or the covering near $\Gamma_{D}$ which does not intersect with $\partial \Gamma_D$. Similarly, the same results can also be shown to hold (by changing the procedure in~\cite{BMRS2014} so that one performs shifts in the outward direction instead of the inward direction)
if we work in a neighborhood of $\Gamma_N\coloneqq \partial \Omega \setminus \overline{\Gamma_D}$ which does not intersect with $\partial \Gamma_D$. Hence, it only remains to check what happens if we consider a part of the partition of unity which intersects with $\partial \Gamma_D$. Without loss of generality, we show how the approximation can be done when a particular choice of covering and $d\ge 3$ is considered (for $d=2$ the proof is even simpler). Hence, let $a\in \mathcal{C}^{0,1}((-1,1)^{d-1})$ and $b\in \mathcal{C}^{0,1}((-1,1)^{d-2})$ be Lipschitz functions fulfilling\footnote{We assume these restrictions on $a$ and $b$
merely for the sake of simplicity of the presentation. Otherwise, we would need to change the geometry slightly and rescale everything.} $\|a\|_{\infty}, \|b\|_{\infty}\le 2^{-1}$ and $\|\nabla a\|_{\infty}, \|\nabla b\|_{\infty}\le L$ for some $L>1$. Next, consider $\Omega$ and $\Gamma_D$ given as
$$
\begin{aligned}
\Omega& \coloneqq \{x\in \mathbb{R}^d \colon x=(x',x_d), \, x'\in (-1,1)^{d-1},\, a(x')<x_d<1\},\\
\Omega^c& \coloneqq \{x\in \mathbb{R}^d \colon x=(x',x_d), \, x'\in (-1,1)^{d-1},\, a(x')>x_d>-1\},\\
\partial \Omega& \coloneqq \{x\in \mathbb{R}^d \colon a(x')=x_d\},\\
\Gamma_D& \coloneqq \{x\in \partial \Omega \colon x=(x',x_d)=(x'',x_{d-1},x_d), \, b(x'')<x_{d-1}\},\\
\Gamma_N & \coloneqq \{x\in \partial \Omega \colon x=(x',x_d)=(x'',x_{d-1},x_d), \, b(x'')>x_{d-1}\}.
\end{aligned}
$$
We consider a function~$\bu$ satisfying the assumptions, and since we are operating in a localized setting, we may assume that $\bu(x)$ can be extended by zero whenever $x'\notin (-1/2,1/2)^d$ so that it remains a Sobolev function. Finally, we focus on a proper approximation of this function.

First of all, let us assume that there exists a constant $C$ such that for all $x\in \Omega$ we have
\begin{equation}
|\bu(x)|\le C\|\beps(\bu)\|_{\infty} \dist(x,\Gamma_D). \label{need}
\end{equation}
Then, by taking an arbitrary nonnegative function $\varphi_n\in \mathcal{C}^{\infty}(\mathbb{R}^d)$ such that $\varphi_n(x)=0$ if $\dist (x,\Gamma_D)<n^{-1}$, $\varphi_n(x)=1$ if $\dist (x,\Gamma_D)> 2n^{-1}$ and $|\nabla \varphi|\le Cn$, and defining $\tilde{\bu}^n\coloneqq \bu \varphi_n$, we see that $\tilde{\bu}^n$ and $\nabla \tilde{\bu}^n$ converge pointwise to~$\bu$ and $\nabla \bu$, respectively. Moreover, we see that
$$
|\beps(\tilde{\bu}^n(x))|\le C|\beps(\bu(x))| + Cn|\bu(x)|\chi_{\{x \colon n^{-1}<\dist (x,\Gamma_D)<2n^{-1}\}}\le C.
$$
Consequently, we see that $\tilde{\bu}^n$ fulfills~\eqref{AAA1} and~\eqref{AAA2}. Moreover, $\tilde{\bu}^n$ is identically zero in an
$n^{-1}$-neighborhood of $\Gamma_D$. Therefore, we can shift $\tilde{\bu}^n$ in the outward direction by $2^{-1}n^{-1}$ and the resulting function will still be identically zero near $\Gamma_D$. Finally, by applying a convolution with a mollification kernel, we can construct the desired sequence of smooth functions fulfilling all requirements; we refer to~\cite{BMRS2014} for the details. Thus, it only remains to check the validity of~\eqref{need}.

First, we show that~\eqref{need} holds on the set
$$
\Omega_D\coloneqq \{x\in \Omega \colon b(x'') < x_{d-1}\}.
$$
Notice that since $\Omega$ is Lipschitz, we know that $\dist (x,\partial \Omega) \sim |x_d - a(x')|$, where the equivalence constant depends on the Lipschitz constant~$L$. Hence, for $x\in \Omega_D$ arbitrary,
we infer from $\bu=\mathbf{0}$ on $\Gamma_D$ that
\begin{equation}\label{desd}
|u_d(x)|=\bigg|\int_{a(x')}^{x_d} \frac{{\rm d}}{{\rm d}s} u_d(x', s)\ddd s\bigg| \le |x_d-a(x')| \|\beps(\bu)\|_{\infty}
\end{equation}
and that consequently~\eqref{need} holds for $u_d$ on $\Omega_D$.

Next, we show that~\eqref{need} holds also for $u_{d-1}$ on $\Omega_D$. Indeed, let us consider an
arbitrary $x\in \Omega_D$ and for any $t\ge 0$ define $y(t)\coloneqq x+t(0,\ldots, 1,-2L)$. Then, we find the smallest $t_0>0$ such $y(t_0)\in \partial \Omega$, i.e., we look for the smallest $t>0$
that solves
$$
a(x'',x_{d-1}+t)=x_d-2Lt.
$$
To get an estimate as well the upper bound for such a $t$, we first notice that since $x\in \Omega$, we have
$$
a(x'',x_{d-1}+0)<x_d-2L0.
$$
On the other hand, using the Lipschitz continuity of $a$, we have
$$
\begin{aligned}
x_d-2Lt &=x_d-a(x') + a(x')-a(x'',x_{d-1}+t)+a(x'',x_{d-1}+t)-2Lt\\
&\le a(x'',x_{d-1}+t) +|x_d-a(x')|-Lt.
\end{aligned}
$$
Hence we see that whenever $t>|x_d-a(x')|/L$ then $y(t)\in \Omega^c$. Consequently, due to continuity of $a$ there
exists a $t_0>0$ fulfilling in addition $t_0\le C(L)|x_d-a(x')|\le C(L)\dist (x,\partial \Omega)$ such that
$y(t_0)\in \partial \Omega$. Moreover, since $x\in\Omega_D$, it follows directly from the definition that $b(y''(t_0))<y_{d-1}(t_0)$ and consequently $y(t_0)\in \Gamma_D$. Since $\bu(y(t_0))=\b0$, we deduce that
$$
\begin{aligned}
|u_{d-1}(x)-2L u_d(x)| &= \left|\int_0^{t_0} \frac{\ddd}{\ddd t} \big( u_{d-1}(y(t)) -2L u_d(y(t)) \big) \ddd t\right| \\
&= \left|\int_0^{t_0} \beps_{d-1,d-1} (\bu(y(t)) -4L\beps_{d-1,d} (\bu(y(t)) +4L^2\beps_{d,d} (\bu(y(t))\ddd t\right| \\
&\le C(L) t_0 \|\beps(\bu)\|_{\infty}\le C(L) \|\beps(\bu)\|_{\infty} \dist(x,\Gamma_D).
\end{aligned}
$$
Thus, since we have already proven~\eqref{need} for $u_d$, we see that it holds for $u_{d-1}$ in $\Omega_D$ as well. Finally, we show the validity of~\eqref{need} in $\Omega_D$ also for $u_i$ with $i=1,\ldots,d-2$. To this end, for $x\in \Omega_D$ we set
$$
y(t)=x + t(1,0,\ldots,0,L , -3L^2).
$$
Using the fact that $x\in \Omega_D$ and the Lipschitz continuity of $a$ and $b$, we see that
$$
b(y''(t))-y_{d-1}(t)\le b(x'')-x_{d-1} +|b(y''(t)-b(x'')|-L t <0
$$
and, by recalling $L>1$,
$$
a(y'(t))-y_d(t)\ge a(x')-x_d -|a(y'(t))-a(x')| +3L^2 t \ge a(x')-x_d +L^2t.
$$
Similarly as above, we find the smallest $t_0>0$ such that $y(t_0)\in \partial \Omega$ and due to the above properties, we see that $y(t_0)\in \Gamma_D$ and $t_0\le |x_d-a(x')|/L^2 \le C(L)\dist (x,\Gamma_D)$. Then we have
$$
\begin{aligned}
|u_1 (x)+L u_{d-1}(x) -3L^2 u_d(x)|&= \bigg|\int_0^{t_0} \frac{\ddd}{\ddd t} (u_1(y(t))+L u_{d-1}(y(t))-3L^2 u_d(y(t)))\ddd t\bigg|\\
&= \bigg|\int_0^{t_0} \beps_{1,1}(\bu(y(t)))+ L^2\beps_{d-1,d-1}(\bu(y(t)))+9 L^4\beps_{d,d}(\bu(y(t))) \\
&\quad + 2L\beps_{1,d-1}(\bu(y(t)))- 6 L^2\beps_{1,d}(\bu(y(t))) - 6 L^3\beps_{d-1,d}(\bu(y(t))) \ddd t\bigg|\\
&\le C(L)t_0\|\beps(\bu)\|_{\infty} \le C(L) \|\beps(\bu)\|_{\infty}\dist(x,\Gamma_D).
\end{aligned}
$$
Consequently, since~$u_{d-1}$ and~$u_d$ fulfill~\eqref{need} we get the same also for~$u_1$ and after the same procedure for all~$u_i$.

Next, we show the validity of~\eqref{need} in $\Omega\setminus \Omega_D$.
Thanks to the definition, we have for all $x$ from this set that $\dist(x,\Gamma_D)\sim \max(|x_d-a(x')|,|x_{d-1}-b(x'')|)$.
First, we introduce the set
\begin{align*}
\Omega^1_{\Gamma_N} \coloneqq
\{x\in \Omega \colon & \textrm{there exists } y\in \Omega \textrm{ such that } b(y'')=y_{d-1}
\textrm{ and } t\in [0,1], \\
& \textrm{satisfying }
x=(y'', y_{d-1} - t(2L)^{-1}(y_d-a(y')),y_d)
\}.
\end{align*}
Note here that if $x(t)=(y'', y_{d-1}- t (2L)^{-1}(y_d-a(y')),y_d)\in \Omega$ for some $t\in [0,1]$ then $x(t)\in \Omega$ for all $t\in [0,1]$. This easily follows from the inequality
$$
x_d(t)-a(x'(t))=y_d-a(y') + a(y')-a(x'(t))\ge y_d-a(y') -L|y'(t)-x'(t)|=(1-t/2)(y_d-a(y'))\ge 0.
$$
Therefore, we can use the same procedure as above and conclude that
$$
\begin{aligned}
|u_{d-1}(x(t))|&\le \left|\int_0^t\frac{\ddd}{\ddd s} u_{d-1}(x(s))\ddd s \right| + |u_{d-1}x(0)|\le L^{-1}|y_d - a(y')|\left|\int_0^t \frac{\partial u_{d-1}(x(s))}{\partial x_{d-1}} \ddd s \right| + |u_{d-1}(y)|\\
& \le C(L) \|\beps(\bu)\|_{\infty}|y_d - a(y')| + C(L) \|\beps(\bu)\|_{\infty} \dist(x,\Gamma_D),
\end{aligned}
$$
where for the last inequality we used the fact that $y\in \overline{\Omega}_{D}$. Finally, since
$$
\begin{aligned}
|y_d - a(y')|&\le |x_d(t)-a(x'(t))| + |a(x'(t))-a(y')|\le |x_d(t)-a(x'(t))| + L|x'(t)-y'|\\
& \le |x_d(t)-a(x'(t))|+|y_d - a(y')|/2
\end{aligned}
$$
we have that $|y_d-a(y')|\le 2|x_d(t)-a(x'(t))|\le C \dist (x(t),\Gamma_D)$ and therefore $u_{d-1}$ satisfies~\eqref{need} in $\Omega_{\Gamma_N}^1$.
Then, for $x(t)\in \Omega^1_{\Gamma_N}$ let us consider $z(s)\coloneqq x(t)+(0,\ldots, s,-s)$ and find the smallest $s_0$ such that either $\bu(z(s_0))=\b0$ or $z(s_0)\in \overline{\Omega}_D$. Note that $s_0\le (2L)^{-1}(y_d-a(y'))$. Consequently, using the triangle inequality, the fact $\dist(x(t),\Gamma_D)\sim \dist(z(s_0),\Gamma_D)$, we get
$$
\begin{aligned}
|u_{d-1}(x(t))-u_d(x(t))| &\le \int_0^{s_0} \left|\frac{\ddd}{\ddd s} (u_{d-1}(z(s))-u_{d}(z(s)))\right| \ddd s+|u_{d-1}(z(s_0))-u_d(z(s_0))|\\
&\le Cs_0 \|\beps(\bu)\|_{\infty} + C(L) \|\beps(\bu)\|_{\infty} \dist(z(s_0),\Gamma_D)
\end{aligned}
$$
and we see that~\eqref{need} holds also for $u_d$ in $\Omega^1_{\Gamma_N}$. Then, following the scheme above we can get the same result also for all $u_i$ with $i=1,\ldots, d-2$. Next, we switch to the set
\begin{align*}
\Omega^2_{\Gamma_N}\coloneqq \{x\in \Omega \colon & \textrm{there exists } y\in \Omega \textrm{ such that } b(y'')=y_{d-1} \textrm{ and } t\ge 0, \\
 & \textrm{satisfying } x=(y'', y_{d-1}- (y_d-a(y'))(2L)^{-1},y_d-t)\}.
\end{align*}
Clearly, if $x(t)=(y'', y_{d-1}- (y_d-a(y'))(2L)^{-1},y_d-t)\in \Omega^2_{\Gamma_N}$, then for all $s\in [0,t]$ we have $x(s)\in \Omega^2_{\Gamma_N}$ as well,
and we have
$$
\begin{aligned}
|u_d(x(t))| &\le \int_0^{|y_d - a(x'(t))|} \left|\frac{\partial u_d(x(t))}{\partial x_d}\right| \ddd t+|u_{d}(x(0))| \\
 & \le C \|\beps(\bu)\|_{\infty} (|y_d - a(x'(t))|+\dist(x(0),\Gamma_N))
 \le C(L) \|\beps(\bu)\|_{\infty} |y_d -a(y')|.
\end{aligned}
$$
Moreover, using
$$
\dist(x(t),\Gamma_D) \ge C|b(x''(t))-x'_{d-1}(t)|=C(L)|y_d - a(y')|,
$$
we see that~\eqref{need} holds also for $u_d$ in $\Omega^2_{\Gamma_N}$. For the other $u_i$ we can now follow
step by step the computation above by choosing a straight line connecting $x(t)$ and $\Omega^1_{\Gamma_N}$.
Finally, since for all $x\notin \Omega^1_{\Gamma_N} \cap \Omega^2_{\Gamma_N}$ we have
$\dist(x,\Gamma_D)\ge \varepsilon>0$ for some $\varepsilon$, we can complete the proof of~\eqref{need} by noting that $\bu\in L^{\infty}$.
\end{proof}

\section{The proofs of Lemmas~\ref{AL1}--\ref{AL5}}
\label{app_minimizers_ws}

\begin{proof}[Proof of Lemma~\ref{AL1}]
First, we focus on the uniqueness. Let $(\bu_1,\bT_1)$ and $(\bu_2,\bT_2)$ be two weak solutions to~\eqref{TT1*-G}. Subtracting the weak formulation~\eqref{TT1*-G} for $(\bu_1,\bT_1)$ from that for $(\bu_2,\bT_2)$
we deduce that
$$
\int_{\Omega} (\bT_1 - \bT_2)\cdot \nabla \bw \ddd x =0 \qquad \textrm{ for all } \bw \in W_{\Gamma_D}^{1,\infty}(\Omega)^N.
$$
Hence, setting $\bw=\bu_1-\bu_2 \in W_{\Gamma_D}^{1,\infty}(\Omega)^N$, we get
$$
\int_{\Omega} (\bT_1 - \bT_2)\cdot (\bD(\bT_1)-\bD(\bT_2)) \ddd x =0.
$$
Using the strict monotonicity~\eqref{subst2} of~$\bD$, we then deduce that $\bT_1=\bT_2$ a.e. in $\Omega$. Consequently, we also get $\nabla \bu_1 = \nabla \bu_2$. We thus see that~$\bT$ is given uniquely, and the same holds true also for~$\bu$ provided that either $\Gamma_D$ is of positive measure or that the mean value of~$\bu$ is fixed (recall here the definition of $W^{1,\infty}_{\Gamma_D}(\Omega)^N$ and that $\Omega$ is connected).

Next, we start from a weak solution $(\bu,\bT)$ to~\eqref{TT1*-G} and want to show~\eqref{minimiz-prim} and~\eqref{minimiz}. It is evident that $\bT\in \mathcal{S}$ and $\bu \in \mathcal{S}^*$. Using the convexity of~$F$, see Lemma~\ref{basici}, combined with~\eqref{derpot}, we get for all $\bW \in \mathcal{S}$ the inequality
$$
\begin{aligned}
\int_{\Omega} F(\bW)-F(\bT) -\nabla \bu_0 \cdot (\bW -\bT) \ddd x &\ge \int_{\Omega} \bD(\bT)\cdot (\bW-\bT) -\nabla \bu_0 \cdot (\bW -\bT) \ddd x \\
&= \int_{\Omega} \nabla (\bu-\bu_0) \cdot (\bW -\bT) \ddd x =0,
\end{aligned}
$$
where the last equality follows from the definition of $\mathcal{S}$ and the fact that $\bu-\bu_0 \in W^{1,\infty}_{\Gamma_D}(\Omega)^N$. Hence,~\eqref{minimiz} is established. To prove also~\eqref{minimiz-prim}, we first notice that the left-hand side of that inequality is finite. Indeed, since~$\bT$ is finite almost everywhere, we have that $\nabla \bu \in \bD(\mathbb{R}^{d\times N})$ almost everywhere, and using~\eqref{propFstar}
we deduce that
$$
\int_{\Omega}F^*(\nabla \bu) \ddd x = \int_{\Omega} \nabla \bu \cdot \bD^{-1}(\nabla \bu) - F(\bD^{-1}(\nabla \bu))\ddd x = \int_{\Omega} \bD(\bT)\cdot \bT - F(\bT) \ddd x <\infty.
$$
Next, we distinguish two possibilities. First, if $\bv\in \mathcal{S}^*$ is such that the set $\{x\in \Omega \colon \nabla \bv(x) \notin \overline{\bD(\mathbb{R}^{d\times N})}\}$ is of positive measure, we simply deduce that the inequality~\eqref{minimiz-prim} holds true since the right-hand side is infinite in this case. Otherwise, if $\nabla \bv \in \overline{\bD(\mathbb{R}^{d\times N})}$ a.e. in $\Omega$, we can use the convexity of~$F^*$ to deduce with the aid of Lemma~\ref{basici} that
$$
\begin{aligned}
&\int_{\Omega} F^*(\nabla \bv)-F^*(\nabla \bu) -\bef \cdot (\bv -\bu)\ddd x - \int_{\Gamma_N} \bg \cdot (\bv -\bu)\ddd S \\
&\ge \int_{\Omega} \bD^{-1}(\nabla \bu) \cdot (\nabla \bv - \nabla \bu) -\bef \cdot (\bv -\bu)\ddd x - \int_{\Gamma_N} \bg \cdot (\bv -\bu)\ddd S\\
&= \int_{\Omega} \bT \cdot (\nabla \bv - \nabla \bu) -\bef \cdot (\bv -\bu)\ddd x - \int_{\Gamma_N} \bg \cdot (\bv -\bu)\ddd S=0,
\end{aligned}
$$
where the last equality follows from~\eqref{TT1*-G} and the fact that $\bv-\bu \in W^{1,\infty}_{\Gamma_D}(\Omega)^N$.
\end{proof}

\begin{proof}[Proof of Lemma~\ref{AL3}]
First, we show that the infimum is finite. To obtain the upper bound, it suffices to show that the set $\mathcal{S}$ is nonempty. Indeed, by considering the problem
$$
-\diver (|\nabla \bv|^{d-1}\nabla \bv) = \bef\, \textrm{ in } \Omega, \qquad |\nabla \bv|^{d-1}\nabla \bv \cdot \bn = \bg\, \textrm{ on } \Gamma_N, \quad \bv = \b0 \textrm{ on } \Gamma_D,
$$
which has a unique solution $\bv\in W^{1,d+1}_{\Gamma_D}(\Omega)^N$ (recall again that either $\Gamma_D$ is of positive measure or the mean value is fixed), we see that the function $|\nabla \bv|^{d-1}\nabla \bv$ belongs to $\mathcal{S}$. In order to establish also the lower bound, we first use the definition of~$F$ (cf.~\eqref{potende}) and Fubini's theorem to deduce, with $\bW\coloneqq\frac{\bT}{|\bT|}$, that
$$
\begin{aligned}
\int_{\Omega}\!F(\bT) - \nabla \bu_0 \cdot \bT\ddd x &= \int_{\Omega}\int_0^1\! (\bD(t\bT)-\nabla \bu_0)\cdot \bT \ddd t \ddd x =\int_0^{\infty}
\int_{\{|\bT(x)|>t\}}\!(\bD(t\bW)-\nabla \bu_0)\cdot \bW \ddd x \ddd t.
\end{aligned}
$$
Next, using the safety strain condition~\eqref{G2_analytic}, we find $t_c >0$ such that, for all $t\ge t_c$ and almost all $x\in \Omega$, we have $(\bD(t\bW)-\nabla \bu_0)\cdot \bW\ge C_1/2$. Consequently, using~\eqref{AE1-g} and the fact that $\bu_0$ is Lipschitz, we deduce that
$$
\begin{aligned}
\int_{\Omega}F(\bT) - \nabla \bu_0 \cdot \bT\ddd x & =\int_0^{t_c} \int_{\{|\bT(x)|>t\}}(\bD(t\bW)-\nabla \bu_0)\cdot \bW \ddd x \ddd t\\
& \quad {}+\int_{t_c}^{\infty} \int_{\{|\bT(x)|>t\}}(\bD(t\bW)-\nabla \bu_0)\cdot \bW \ddd x \ddd t\\
&\ge -C(C_2,\Omega,t_c,\|\nabla \bu_0\|_{\infty}) + \frac{C_1}{2}\int_{t_c}^{\infty} |\{x\in \Omega \colon |\bT(x)|>t\}| \ddd t\\
&\ge -C(C_1,C_2,\Omega,t_c,\|\nabla \bu_0\|_{\infty}) + \frac{C_1}{2}\int_{0}^{\infty} |\{x\in \Omega \colon |\bT(x)|>t\}| \ddd t\\
&=\frac{C_1\|\bT\|_1}{2}-C,
\end{aligned}
$$
and the lower bound for the infimum follows from this lower bound. Note also that the above computations imply that every minimizing sequence for $J$ in $\mathcal{S}$ is bounded in $L^1(\Omega)^{d\times N}$. Therefore, if the infimum of the mapping $\bT \mapsto J(\bT)$ is attained for some $\bT \in L^1(\Omega)^{d\times N}$, then necessarily $\|\bT\|_1\le C$ with a constant $C$ depending only on the data.

In order to prove the second part of the lemma, we consider a minimizer~$\bT$ of $J$ in $\mathcal{S}$.
First, since~$F$ is strictly convex, see Lemma~\ref{basici}, and since the set $\mathcal{S}$ is closed and convex in $L^1(\Omega)^{d \times N}$, we see that there is at most one minimizer. Setting $\bW=\bT + \lambda \tilde{\bW}$ in~\eqref{minimiz} with arbitrary $\tilde{\bW}\in L^1(\Omega)^{d\times N}$ satisfying
\begin{equation}\label{xmen}
\int_{\Omega} \tilde{\bW} \cdot \nabla \bw\ddd x=0 \qquad \textrm{ for all } \bw \in W^{1,\infty}_{\Gamma_D}(\Omega)^d,
\end{equation}
we deduce the following Euler--Lagrange equation for the minimization problem~\eqref{minimiz} for the functional $J$:
\begin{equation}
\int_{\Omega} (\bD(\bT)-\nabla \bu_0) \cdot \tilde{\bW} \ddd x=0 \qquad \textrm{ for all } \tilde{\bW} \textrm{ satisfying~\eqref{xmen}}.
\label{ELmin}
\end{equation}
Next, we find a unique~$\bu$ such that $\bu-\bu_0 \in W^{1,2}_{\Gamma_D}(\Omega)^N$ and
\begin{equation}
\int_{\Omega}
\nabla \bu \cdot \nabla \bw \ddd x= \int_{\Omega} \bD(\bT)\cdot \nabla \bw \ddd x \qquad \textrm{ for all } \bw\in W^{1,2}_{\Gamma_D}(\Omega)^N.
\label{findbu}
\end{equation}
Our final goal is to show that $\nabla \bu = \bD(\bT)$ a.e. in $\Omega$, which will directly imply that $\bu \in W^{1,\infty}(\Omega)^N$ and hence, as $T \in \mathcal{S}$, that the couple $(\bT,\bu)$ is a weak solution.
To prove that $\nabla \bu = \bD(\bT)$, we set $\tilde{\bW}\coloneqq\nabla \bu -\bD(\bT)$ in~\eqref{ELmin} (which is an admissible choice thanks to~\eqref{findbu}) to deduce that
\begin{equation*}
\|\bD(\bT)\|_2^2=-\int_{\Omega} \nabla \bu_0 \cdot (\nabla \bu -\bD(\bT)) \ddd x+\int_{\Omega} \bD(\bT)\cdot \nabla \bu \ddd x,
\end{equation*}
and setting $\bv=\bu - \bu_0$ in~\eqref{findbu} we deduce that
\begin{equation*}
\|\nabla \bu\|_2^2= \int_{\Omega} \bD(\bT)\cdot \nabla(\bu-\bu_0)\ddd x+\int_{\Omega}
\nabla \bu \cdot \nabla \bu_0 \ddd x.
\end{equation*}
Summing the resulting identities we obtain
$$
\begin{aligned}
\|\bD(\bT)\|_2^2+\|\nabla \bu\|_2^2=2\int_{\Omega} \bD(\bT) \cdot \nabla \bu\ddd x \implies \|\bD(\bT)-\nabla \bu\|_2^2 =0,
\end{aligned}
$$
which completes the proof.
\end{proof}
\begin{proof}[Proof of Lemma~\ref{AL2}]
First, we show that the infimum of $J^*$ is finite. For this purpose it is enough to check that
\begin{equation}\label{check}
\left|\int_{\Omega} F^*(\nabla \bu_0)\ddd x\right| <\infty
\end{equation}
(keeping in mind condition~\eqref{D3} in the case $\Gamma_D = \emptyset$). However, this is a direct consequence of the condition~\eqref{proto-G2}, which guarantees the existence of a number $t_c$ and a measurable function $\bT_0:\Omega \to \mathbb{R}^{d\times N}$ such that one has $\|\bT_0\|_{\infty} \le t_c$ and $\nabla \bu_0=\bD(\bT_0)$ a.e.~in $\Omega$. Consequently,
$$
\left|\int_{\Omega} F^*(\nabla \bu_0)\ddd x\right| = \left|\int_{\Omega} \bT_0 \cdot \nabla \bu_0 - F(\bT_0)\ddd x\right| \leq C(C_2,t_c, \|\nabla u_0\|_{\infty},\Omega) < \infty,
$$
which in turn shows that the infimum of $J^*$ over $\mathcal{S}^*$ is finite. Therefore, we can find a minimizing sequence $\bu^n \in \mathcal{S}^*$, which, because of the coercivity of $J^*$ (recalling $F^*=\infty$ outside of the compact set $\overline{\bD(\mathbb{R}^{d\times N})}$ and the assumptions on $\bg,\bef,\bu_0$), satisfies
$$
\|\bu^n\|_{1,\infty} \le C.
$$
Consequently, using the weak-$*$ lower semicontinuity of convex functionals, we deduce that the infimum is attained for some $\bu\in \mathcal{S}^*$ (note that $S^\ast$ is closed and convex in $W^{1,\infty}(\Omega)^N$). Consequently, since $F^*(\nabla \bu)$ is finite almost everywhere, we have that $\nabla \bu \in \overline{\bD(\mathbb{R}^{d\times N})}$ almost everywhere in $\Omega$. Next, we show that there exists a $\bT \in L^1(\Omega)^{d\times N}$ such that $\bD(\bT)=\nabla \bu$ almost everywhere in $\Omega$, and which satisfies~\eqref{nerovnost-g}. To this end, let us assume that $\bv \in \mathcal{S}^*$ is arbitrary with
$$
\int_{\Omega}F^*(\nabla \bv)\ddd x < \infty
$$
(which directly yields that $\nabla \bv$ takes values in $\overline{\bD(\mathbb{R}^{d\times N})}$ a.e. in $\Omega$). Then, because of the convexity of~$F^*$, it follows that also
$$
\int_{\Omega}F^*((1-\lambda)\nabla \bu + \lambda \nabla \bv)\ddd x <\infty \qquad \textrm{ for all } \lambda \in [0,1].
$$
Therefore, we may now use $(1-\lambda)\bu + \lambda \bv$ instead of~$\bv$ in~\eqref{minimiz-prim}, and after dividing by $\lambda$ we get the inequality
\begin{equation*}
\int_{\Omega} \frac{F^*(\nabla \bu)-F^*((1-\lambda)\nabla \bu + \lambda \nabla \bv)}{\lambda}\ddd x \le \int_{\Omega} \bef \cdot (\bu-\bv)\ddd x + \int_{\Gamma_{N}} \bg \cdot (\bu - \bv)\ddd S,
\end{equation*}
which by the convexity of~$F^*$ leads to
\begin{equation}
\int_{\Omega} \bD^{-1}((1-\lambda)\nabla \bu + \lambda \nabla \bv)\cdot (\nabla \bu - \nabla \bv)\ddd x \le \int_{\Omega} \bef \cdot (\bu-\bv)\ddd x + \int_{\Gamma_{N}} \bg \cdot (\bu - \bv)\ddd S, \label{hruza2}
\end{equation}
provided that for almost all $x\in \Omega$ and all $\lambda \in (0,1)$ we have
\begin{equation}
(1-\lambda)\nabla \bu(x) + \lambda \nabla \bv(x)\in \bD(\mathbb{R}^{d\times N}).\label{hr23}
\end{equation}
However, to justify~\eqref{hr23}, it is enough to show that at least one of the following inclusions holds:
\begin{equation}
\nabla \bu (x) \in \bD(\mathbb{R}^{d\times N}) \quad \text{or} \quad
\nabla \bv (x) \in \bD(\mathbb{R}^{d\times N}).
\label{hr24}
\end{equation}
Indeed, assume for example that the second holds (the arguments for the other inclusion are exactly the same); we shall then show that for any $\lambda \in (0,1)$ there exists an $\varepsilon >0$ such that for all $\bT \in B_{\varepsilon}(\b0)$ we have
\begin{equation}\label{previous}
(1-\lambda)\nabla \bu(x) + \lambda \nabla \bv(x) + \bT \in \overline{\bD(\mathbb{R}^{d\times N})},
\end{equation}
which then necessarily implies~\eqref{hr23}. In order to prove that the assertion~\eqref{previous} holds, we note that
$$
(1-\lambda)\nabla \bu(x) + \lambda \nabla \bv(x) + \bT =(1-\lambda)\nabla \bu(x) + \lambda (\nabla \bv(x) + \lambda^{-1}\bT),
$$
and, using the convexity of $\overline{\bD(\mathbb{R}^{d\times N})}$ (see Lemma~\ref{basici}), we observe that it suffices to check that $(\nabla \bv(x) + \lambda^{-1}\bT) \in \overline{\bD(\mathbb{R}^{d\times N})}$. However, since $\nabla \bv(x)$ belongs to the open set $\bD(\mathbb{R}^{d\times N})$, the claim follows by choosing $\varepsilon$ sufficiently small.

In view of the equality $\nabla \bu_0 = \bD(\bT_0)$ guaranteed by~\eqref{proto-G2}, we now take advantage of~\eqref{hruza2}--\eqref{hr24} with $\bv\coloneqq\bu_0$. Defining
$$
\bT_{\lambda}\coloneqq\bD^{-1}((1-\lambda)\nabla \bu(x) + \lambda \nabla \bu_0(x)),
$$
we have that
\begin{equation}
\bD(\bT_{\lambda}) \to \nabla \bu \quad \textrm{a.e. in } \Omega \label{hr56}
\end{equation}
as $\lambda \to 0_+$. In addition, we see that~\eqref{hruza2} now reduces to
\begin{equation*}
\begin{split}
\frac{1}{1-\lambda}\int_{\Omega} \bT_{\lambda}\cdot (\bD(\bT_{\lambda}) - \nabla \bu_0)\ddd x&=\int_{\Omega} \bD^{-1}((1-\lambda)\nabla \bu + \lambda \nabla \bu_0)\cdot (\nabla \bu - \nabla \bu_0)\ddd x\\
& \le \int_{\Omega} \bef \cdot (\bu-\bu_0)\ddd x + \int_{\Gamma_{N}} \bg \cdot (\bu - \bu_0)\ddd S\le C,
\end{split}
\end{equation*}
where the last inequality follows from the assumptions on $\bg,\bef,\bu_0$ and the fact that $\bu \in W^{1,\infty}(\Omega)^N$ by construction. The condition~\eqref{proto-G2} then implies that for all $\lambda \in (0,\textstyle{\frac{1}{2}})$ we have
$$
\|\bT_{\lambda}\|_1 \le C(\Omega,t_c,\bg,\bef,\bu_0).
$$
Consequently, using~\eqref{hr56}, the strict monotonicity of~$\bD$, the above estimate and Fatou's lemma, we deduce that there exists a (unique) function $\bT \in L^1(\Omega)^{d\times N}$ with $\|\bT\|_1\le C(\Omega,t_c,\bg,\bef,\bu_0)$ such that $\bT_{\lambda} \to \bT$ a.e. in $\Omega$ as $\lambda \rightarrow 0_+$, and $\nabla \bu = \bD(\bT)$, as asserted. Obviously, the minimizer~$\bu$ is also unique, by strict convexity of~$F^*$ on the set $\bD(\mathbb{R}^{d\times N})$ from Lemma~\ref{basici} (combined with an inclusion of the form~\eqref{hr23} for two potential minimizers~$\bu, \bv$).

In addition, since~$\bT$ is finite almost everywhere, we see that the first part of~\eqref{hr24} automatically holds, and consequently~\eqref{hruza2} is valid for any choice of $\bv$. Next we verify~\eqref{nerovnost-g}. For this purpose we consider in~\eqref{hruza2} a test function $\bv \in W^{1,\infty}(\Omega)^N$ such that $\bv -\bu_0 \in W^{1,\infty}_{\Gamma_D}(\Omega)^N$ and $\bD(\tilde{\bT})=\nabla \bv$ for some function $\tilde{\bT} \in L^1(\Omega)^{d\times N}$. We now want to let $\lambda \to 0_+$ on the left-hand side of~\eqref{hruza2} and we therefore begin by expressing~\eqref{hruza2} as
\begin{equation}
\int_{\Omega} \bD^{-1}((1-\lambda)\bD(\bT) + \lambda \bD(\tilde{\bT}))\cdot (\bD(\bT)-\bD(\tilde{\bT}))\ddd x \le \int_{\Omega} \bef \cdot (\bu-\bv)\ddd x + \int_{\Gamma_{N}} \bg \cdot (\bu - \bv)\ddd S. \label{hruza3}
\end{equation}
By rewriting the left-hand side of this as
$$
\begin{aligned}
&\int_{\Omega} \bD^{-1}((1-\lambda)\bD(\bT) + \lambda \bD(\tilde{\bT}))\cdot (\bD(\bT)-\bD(\tilde{\bT}))\ddd x\\
&= \frac{1}{1-\lambda}\int_{\Omega} \left(\bD^{-1}((1-\lambda)\bD(\bT) + \lambda \bD(\tilde{\bT}))-\bD^{-1}(\bD(\tilde{\bT}))\right)\cdot \left([(1-\lambda)\bD(\bT)+\lambda\bD(\tilde{\bT})]-[\bD(\tilde{\bT})]\right)\ddd x\\
&\quad +\int_{\Omega} \tilde{\bT}\cdot (\bD(\bT)-\bD(\tilde{\bT}))\ddd x,
\end{aligned}
$$
we see that, because of the monotonicity of $\bD^{-1}$ (see Lemma~\ref{basici}), the first term is nonnegative and the second term is integrable. Therefore, we can use Fatou's lemma to pass to the limit $\lambda \to 0_+$ in~\eqref{hruza3} to arrive at the claim
\begin{equation}
\int_{\Omega} \bT\cdot (\nabla \bu -\nabla \bv)\ddd x \le \int_{\Omega} \bef \cdot (\bu-\bv)\ddd x + \int_{\Gamma_{N}} \bg \cdot (\bu - \bv)\ddd S. \label{krasa1}
\end{equation}

Finally, we prove the last assertion of the lemma. By hypothesis, we can approximate any $\bw \in W^{1,\infty}_{\Gamma_D}(\Omega)^N$ by a sequence $\bw^n$ in $W^{1,\infty}_{\Gamma_D}(\Omega)^N$ in the sense of~\eqref{novyas}, and we now want to use
$$
\bv=\bu-\varepsilon\bw^n
$$
as a test function in inequality~\eqref{krasa1}. To this end, we need to justify this choice and must thus check that for sufficiently small $\varepsilon$ there exists a $\bT_{\varepsilon}^n \in L^1(\Omega)^{d\times N}$ with $\bD(\bT_{\varepsilon}^n) =\nabla \bu -\varepsilon \nabla \bw^n$. However, since $\nabla \bw^n$ is supported on the set
$\Omega_{n}\coloneqq\{x\in \Omega \colon |\bT(x)|\le n\}$, we see that $\bT_{\varepsilon}^n=\bT$ in $\Omega \setminus \Omega_{n}$. On the other hand, if $x\in \Omega_n$, then $\nabla \bu(x) \in \bD(\overline{B_n(\b0)})\subset \bD(B_{2n}(\b0)) \subset \overline{\bD(\mathbb{R}^{d\times N})}$. Since the first set is closed, the second open and the last closed, we see that for each $n$ there exists a $\delta$ such that
$$
0<\delta \le |\bT-\bQ| \quad \textrm{ for all } \bT \in \bD(\overline{B_n(\b0)}) \textrm{ and } \bQ\in \partial \bD(\mathbb{R}^{d\times N}).
$$
Consequently, setting $\varepsilon$ so small that $2\varepsilon |\nabla \bw^n|\le \delta$, we can find a closed set $K$ such that
$$
\nabla u(x) - \varepsilon \nabla \bw^n(x) \in K \subset \bD(\mathbb{R}^{d\times N}) \quad \textrm{a.e. in } \Omega_n,
$$
and therefore there exists a $\bT_{\varepsilon}^n \in L^{\infty}(\Omega_n)^{d\times N}$ such that $\bD(\bT_{\varepsilon}^n) =\nabla \bu -\varepsilon \nabla \bw^n$.
Thus, $\bv\coloneqq\bu-\varepsilon \bw^n$ is an admissible choice in~\eqref{krasa1} for $\varepsilon$ sufficiently small, and we therefore deduce that
\begin{equation}
\int_{\Omega} \bT\cdot \nabla \bw^n\ddd x \le \int_{\Omega} \bef \cdot \bw^n\ddd x + \int_{\Gamma_{N}} \bg \cdot \bw^n \ddd S. \label{krasa5}
\end{equation}
Using the weak-$*$ density property assumed in the statement of the lemma we see that~\eqref{krasa5} holds also for $\bw$, and since $\bw$ is arbitrary, it holds in fact with the equality sign. Consequently, the couple $(\bu,\bT)$ is a weak solution.
\end{proof}

\begin{proof}[Proof of Lemma~\ref{AL4}]
We start by showing that the infimum for the minimization problem~\eqref{minimiz5} is finite and attained for some $\bT \in \mathcal{S}^m$. To this end, we observe that $\mathcal{S}$ is a nonempty subset of $\mathcal{S}^m$ and the functional in~\eqref{minimiz5} coincides with $J$ on $\mathcal{S}$; hence, the infimum of the functional $\mathcal{J}$ is bounded from above via Lemma~\ref{AL3}. Moreover, using a Fubini-type argument and involving the definition of the recession function $F_\infty$, we can rewrite the functional in~\eqref{minimiz5} similarly as in the proof of Lemma~\ref{AL3} to then infer from the condition~\eqref{G2_analytic} (or~\eqref{proto-G2}) a lower bound of the form
\begin{equation*}
\mathcal{J}(\bT) = \int_{\Omega} F(\bT^r)\ddd x+\mathcal{F}_{\infty}(\bT^s)-\langle \nabla \bu_0,\bT\rangle \geq \frac{C_1}{2} |\bT|(\Omega \cup \overline{\Gamma_N}) - C
\end{equation*}
for a constant $C>0$ depending only on the data. This implies that the infimum is bounded from below. Moreover, since bounded sequences in $\mathcal{M}(\Omega \cup \overline{\Gamma_N})$ are relatively compact with respect to weak-$*$ convergence and since $\mathcal{S}^m$ is weakly-$*$ closed, we obtain the existence of a minimizer $\bT \in \mathcal{S}^m$ by Reshetnyak's lower semicontinuity result, see~\cite[Theorem 2.38]{AMBFUSPAL00}. For its rigorous application, we first extend the measures in $\mathcal{S}^m$ by zero to an open set $\Omega'$ containing $\Omega \cup \overline{\Gamma_N}$, then we rewrite $\mathcal{J}$ with the help of the function $\bar{F} \colon [0,\infty) \times \mathbb{R}^{d\times N} \to \mathbb{R}$, defined for $\bT \in \mathbb{R}^{d\times N}$ as $\bar{F}(s,\bT) \coloneqq sF(\bT/s)$ for $s >0$ and $\bar{F}(0,\bT)\coloneqq F_\infty(\bT)$ for $s=0$, in terms of a positively $1$-homogeneous integrand as
\begin{equation*}
 \mathcal{J}(\bT) + F(0) \mathcal{L}^d(\Omega' \setminus \Omega) = \int_{\Omega'} \bar{F} \left(\frac{\langle \mathcal{L}^d,\bT\rangle}{|\langle \mathcal{L}^d,\bT\rangle|}\right) \ddd |\langle\mathcal{L}^d,\bT\rangle| -\langle \nabla \bu_0,\bT\rangle ,
\end{equation*}
for which, because of the linear growth assumption on~$F$, Reshetnyak's lower semicontinuity theorem can finally be applied.

Next, we prove the restricted uniqueness assertion for the set of minimizers. Given two minimizers $\bT,\bar{\bT} \in \mathcal{S}^m$, we immediately obtain the identity $\bT^r = \bar{\bT}^r$ for the absolutely continuous parts with respect to the Lebesgue measure. Indeed, in view of the convexity of~$F$, $\mathcal{F}_{\infty}$ and $\mathcal{S}^m$, the fact that~$\bT$ and $\bar{\bT}$ are minimizers implies that also $(\bT + \bar{\bT})/2 \in \mathcal{S}^m$ is a minimizer, with
\begin{align*}
 2 \int_{\Omega} F \big((\bT^r + \bar{\bT}^r)/2\big) \ddd x & = \int_{\Omega} \big( F(\bT^r) +F(\bar{\bT}^r) \big) \ddd x, \\
 \mathcal{F}_{\infty}(\bT^s + \bar{\bT}^s) & = \big( \mathcal{F}_{\infty}(\bT^s) + \mathcal{F}_{\infty}(\bar{\bT}^s) \big).
\end{align*}
Since~$F$ is in fact strictly convex (see Lemma~\ref{basici}), the first identity immediately gives $\bT^r = \bar{\bT}^r$ a.e.~ in $\Omega$. Concerning the singular parts $\bT^s, \bar{\bT}^s$ of the minimizers, we note that $\mathcal{F}_{\infty}$ is only convex, but not strictly convex. Hence, uniqueness does not in general follow from the second identity. However, by the definition of $\mathcal{S}^m$, we still have
\begin{equation}\label{TsTsbar}
 \langle\bT^s - \bar{\bT}^s, \nabla \bw\rangle = 0 \qquad \textrm{ for all } \bw\in \mathcal{C}^1_{\Gamma_D}(\overline{\Omega})^N .
\end{equation}

Finally, we obtain $\nabla \bu=\bD(\bT^r)$ for the unique minimizer~$\bu$ to the primal problem exactly as in the proof of Lemma~\ref{AL3} by considering variations $\bT + \lambda \tilde{\bW}$ about the minimizer~$\bT$, with an arbitrary $\tilde{\bW}\in L^1(\Omega)^{d\times N}$ that satisfies~\eqref{xmen} (and hence, $\bT + \lambda \tilde{\bW}$ is a competitor in $\mathcal{S}^m$), and then by repeating the same argument, with~$\bT$ replaced by $\bT^r$. With the identity~\eqref{TsTsbar} in hand, the remaining assertion~\eqref{nerovnost-g} then follows from Lemma~\ref{AL2}.
\end{proof}

Next we prove the partial reverse implication: that solutions provide minimizers.

\begin{proof}[Proof of Lemma~\ref{AL5}]
We begin by proving the minimality of~$\bu$. To this end, we initially consider a function $\bv \in \mathcal{S}^*$ with $\nabla \bv =\bD(\tilde{\bT})$ for some $\tilde{\bT}\in L^1(\Omega)^{d\times N}$. Then, in view of the inequality~\eqref{TH00_min} (or inequality~\eqref{TT2-G}, respectively) and the convexity of~$F^*$, we have that
\begin{align}\label{minimal-ini}
 J^*(\bu) & = \int_{\Omega} F^*(\nabla \bu)- \bef \cdot \bu \ddd x -\int_{\Gamma_N} \bg \cdot \bu\ddd S \nonumber\\
 & \leq \int_{\Omega} F^*(\nabla \bu)- \bT \cdot (\nabla \bu - \nabla \bv) \ddd x - \int_\Omega \bef \cdot \bv \ddd x -\int_{\Gamma_N} \bg \cdot \bv \ddd S \nonumber\\
 & \leq J^*(\bv) + \int_{\Omega} \Big( \frac{\partial F^*(\nabla \bu)}{\partial \bB} - \bT \Big) \cdot (\nabla \bu - \nabla \bv) \ddd x = J^*(\bv),
\end{align}
where the last equality is a consequence of $\nabla \bu = \bD(\bT)$ in $\Omega$ and the formula~\eqref{derpot2}. In order to obtain the full minimality property among all functions in $\mathcal{S}^*$, we still need to admit functions $\bv \in \mathcal{S}^* $ with $\nabla \bv \in \overline{\bD(\mathbb{R}^{d\times N})}$ a.e.~in $\Omega$ (recalling that $F^* = \infty$ outside of $\overline{\bD(\mathbb{R}^{d\times N})}$). However, thanks to the condition~\eqref{proto-G2}, we have $\nabla \bu_0(x) = \bD(\bT_0(x))$ with $\bT_0(x) \in B_{t_c}(0)$ for some $t_c$, uniformly for a.e.~$x \in \Omega$, and therefore, by convexity of $\overline{\bD(\mathbb{R}^{d\times N})}$, we have the representation
\begin{equation*}
 \nabla \bv_t \coloneqq \nabla (t \bu_0 + (1-t) \bv) = \bD(\tilde{\bT}_t)
\end{equation*}
with functions $\tilde{\bT}_t \in L^1(\Omega)^{d\times N}$, for all $t \in (0,1]$. At this stage, we invoke~\eqref{minimal-ini} with $\bv = \bv_t$, and the convexity of $J^*$, to find, for all $t \in (0,1]$, that
\begin{equation*}
 J^*(\bu) \leq J^*(\bv_t) \leq t J^*(\bu_0) + (1-t) J^*(\bv) \,.
\end{equation*}
Since $J^*(u_0)$ is finite, cf.~\eqref{check}, the minimality of~$\bu$ follows in the limit $t \searrow 0$.

Next we proceed to proving the minimality of~$\bT$, under the additional assumption $\tilde{\bg} = \b0$. In this case~$\bT$ obviously belongs to $\mathcal{S}$, as a consequence of~\eqref{TH00-G}, the $L^1$-regularity of~$\bT$ and a density argument. Moreover, the minimality of~$\bT$ for $J^*$ follows similarly as above, by taking advantage of the convexity of~$F$, the identity $\nabla \bu = \bD(\bT)$ and the formula~\eqref{derpot}: indeed, we have, for any $\bar{\bT} \in \mathcal{S}$,
\begin{align*}
J(\bT) & = \int_{\Omega} F(\bT)-\nabla \bu_0 \cdot \bT \ddd x \leq \int_{\Omega} F(\bar{\bT}) - \nabla \bu_0 \cdot \bar{\bT} \ddd x + \int_\Omega \Big( \frac{\partial F(\bT)}{\partial \bT} - \nabla \bu_0 \Big) \cdot (\bT - \bar{\bT}) \ddd x \\
 & = J(\bar{\bT}) + \int_\Omega \big( \nabla \bu - \nabla \bu_0 \big) \cdot (\bT - \bar{\bT}) \ddd x = J(\bar{\bT}),
\end{align*}
where the last equality is true, since $\bu - \bu_0 \in W^{1,\infty}_{\Gamma_D}(\Omega)^N$ and $\bT, \bar{\bT}$ both belong to $\mathcal{S}$. This proves the minimality of~$\bT$ and concludes the proof.
\end{proof}
\end{appendix}

\bibliographystyle{plain}

\end{document}